\documentclass[onefignum,onetabnum]{siamart190516}

\usepackage{lipsum}
\usepackage{amssymb}
\usepackage{amsfonts}
\usepackage{graphicx}
\usepackage{epstopdf}
\usepackage{algorithmic}
\ifpdf
  \DeclareGraphicsExtensions{.eps,.pdf,.png,.jpg}
\else
  \DeclareGraphicsExtensions{.eps}
\fi

\newsiamremark{remark}{Remark}
\newsiamremark{hypothesis}{Hypothesis}
\crefname{hypothesis}{Hypothesis}{Hypotheses}
\newsiamthm{claim}{Claim}

\headers{Exponential Convergence for Solving the Helmholtz Equation}{Y. Chen, T.Y. Hou, and Y. Wang}

\title{Exponentially convergent multiscale methods for high frequency heterogeneous Helmholtz equations\thanks{Submitted to the editors DATE: TBA.
\funding{This research is in part supported by NSF Grants DMS-1912654 and DMS 2205590. Y. Chen is also grateful to the support from the Caltech Kortchak Scholar Program.}}}

\author{Yifan Chen\thanks{Applied and Computational Mathematics, Caltech
  (\email{yifanc@caltech.edu}, \email{hou@cms.caltech.edu},  \email{roywang@caltech.edu}).}
  \and Thomas Y. Hou\footnotemark[2]
\and Yixuan Wang\footnotemark[2]}

\usepackage{amsopn}

\usepackage{subfig}
\usepackage{float} 
\usepackage{hyperref} 
\usepackage{upgreek} 
\usepackage{xcolor} 
\usepackage{physics}
\usepackage{tikz}
\usepackage{mathdots}
\usepackage{yhmath}
\usepackage{cancel}
\usepackage{color}
\usepackage{siunitx}
\usepackage{array}
\usepackage{multirow}
\usepackage{tabularx}
\usepackage{booktabs}
\usetikzlibrary{fadings}
\usetikzlibrary{patterns}
\usetikzlibrary{shadows.blur}
\usetikzlibrary{shapes}
\usepackage{pgfplots}

\usepackage{diagbox} 

\newcommand{\cF}{\mathcal{F}}
\newcommand{\cN}{\mathcal{N}}

\newcommand{\cH}{\mathcal{H}}
\newcommand{\cT}{\mathcal{T}}
\newcommand{\cE}{\mathcal{E}}

\newcommand{\bR}{\mathbb{R}}

\newcommand{\bN}{\mathbb{N}}

\newcommand{\rN}{\mathrm{N}}

\newcommand{\sfh}{\mathsf{h}}
\newcommand{\sfb}{\mathsf{b}}
\newcommand{\sfs}{\mathsf{s}}
\newcommand{\sfp}{\mathsf{p}}

\newcommand{\dist}{\operatorname{dist}}

\newtheorem{assumption}{Assumption}

 \pgfplotsset{compat=1.16} 
 
\newcommand{\yc}[1]{{\color{black} #1}}
\newcommand{\yw}[1]{{\color{black} #1}}

\begin{document}

\maketitle

\begin{abstract}
In this paper, we present a multiscale framework for solving the Helmholtz equation in heterogeneous media without scale separation and in the high frequency regime where the wavenumber $k$ can be large. The main innovation is that our methods achieve a nearly exponential rate of convergence with respect to the computational degrees of freedom, using a coarse grid of mesh size $O(1/k)$ without suffering from the well-known pollution effect. The key idea is {a non-overlapped domain decomposition and its associated} coarse-fine scale decomposition of the solution space that adapts to the media property and wavenumber; this decomposition is inspired by the multiscale finite element method (MsFEM). We show that the coarse part is of \textit{low complexity} in the sense that it can be approximated with a nearly exponential rate of convergence via local basis functions{, due to the compactness of a restriction operator that maps Helmholtz-harmonic functions to their interpolation residues on edges,} while the fine part is \textit{local} such that it can be computed efficiently using the local information of the right hand side. The combination of the two parts yields the overall nearly exponential rate of convergence {of our multiscale method}. {Our method draws many connections to multiscale methods in the literature, which we will comment in detail.} We demonstrate the effectiveness of our methods theoretically and numerically; an exponential rate of convergence is consistently observed and confirmed. In addition, we observe the robustness of our methods regarding the high contrast in the media numerically. \yc{We specifically focus on 2D problems in our exposition since the geometry of non-overlapped domain decomposition is 
simplest to explain in such cases; generalizations to 3D will be outlined at the end.}
\end{abstract}

\begin{keywords}
 The Helmholtz equation, Heterogeneous Media, High Frequency, Exponential Convergence, Multiscale Methods, High Contrast.
\end{keywords}

\begin{AMS}
  65N12, 65N15, 65N30, 31A35.
\end{AMS}

\section{Introduction}
This paper focuses on solving the Helmholtz equation in heterogeneous media and high frequency regimes. 
We consider the model problem in a bounded domain $\Omega \subset \bR^d$ with a Lipschitz boundary $\Gamma$. For generality, the boundary can contain three disjoint parts $\Gamma = \Gamma_D\cup \Gamma_N \cup \Gamma_R$ where $\Gamma_D, \Gamma_N$ and $\Gamma_R$ correspond to the Dirichlet, Neumann and Robin type conditions, respectively.  Given positive constants $A_{\min}$, $A_{\max}$, $\beta_{\min}$, $\beta_{\max}$, $V_{\min}$, $V_{\max}$ and functions $A, \beta, V: \Omega \to \bR$ that satisfy $A_{\min}\leq A(x)\leq A_{\max}$,
$\beta_{\min}\leq \beta(x)\leq \beta_{\max}$ and $V_{\min}\leq V(x)\leq V_{\max}$,
the Helmholtz equation with homogeneous boundary conditions\footnote{For simplicity of presentation, homogeneous boundary conditions are considered here. Generalization to non-homogeneous data is straightforward; see Section \ref{subsec: High Wavelength Example} or \yc{\cite{chen2020exponential}} (Section 5.3).} is formulated as follows:
\begin{equation}
\label{eqn:Helmholtz smooth k}
\left\{
\begin{aligned}
-\nabla \cdot(A\nabla u)-k^{2}V^2 u&=f, \ \text{in} \ \Omega\\
u&=0, \ \text{on} \ \Gamma_D\\
A\nabla u\cdot\nu&=T_ku, \ \text{on} \  \Gamma_N \cup \Gamma_R \, .
\end{aligned}
\right.
\end{equation}
Here, $\nu$ is the outer normal to the boundary. The boundary operator satisfies $T_k u=0$ for $x \in \Gamma_N$ and $T_k u=ik\beta u$ for $x \in \Gamma_R$, where $i$ denotes the imaginary number. The wavenumber $k$ is real and positive, and functions $u$ and $f$ are complex-valued. The aim of this paper is to design a multiscale method for solving \eqref{eqn:Helmholtz smooth k} that achieves a nearly exponential rate of convergence with respect to the computational degrees of freedom. {This is a challenging problem due to combined difficulties of heterogeneity and high frequency. We review the related literature of this research field in Section \ref{sec: literature for Helmholtz} and discuss our methodology as well as its motivations and related work in Section \ref{sec: contribution, motives}.}

\subsection{{Literature for Solving Helmholtz Equations}}
\label{sec: literature for Helmholtz}
 The Helmholtz equation has been widely used in studying wave propagation in complex media. Numerical simulation of this equation still remains a challenging task, especially in the regime where the wavenumber $k$ is large. The main numerical difficulty lies in the highly oscillatory pattern of the solution. Furthermore, the operator in the equation is indefinite, which leads to severe instability issues for standard numerical solvers such as the finite element method (FEM). Indeed, a well-known pre-asymptotic effect called the pollution effect \cite{babuska1997pollution} 
can occur --- that is, in order to get a reasonably accurate solution, the mesh size $H$ in the FEM needs to be much smaller than $1/k$. For example, for a standard P1-FEM approach, the mesh size needs to satisfy $H = O(1/k^2)$ \yc{for quasi-optimality of the solution} \cite{aziz1988two, babuska1997pollution}. This constraint on $H$ is much stronger than the typical condition in the approximation theory for representing an oscillatory function with frequency $k$, where $H = O(1/k)$, i.e., a fixed number of grid points per wavelength, would suffice for an accurate approximate solution.  

 In the literature, there have been many attempts to overcome or alleviate the difficulty associated with the pollution effect, so that a mesh size of $H=O(1/k)$ can be used. We highlight two classes of methods, namely the $hp$-FEM and multiscale methods, which can theoretically deal with the pollution effect under their respective model assumptions. 
 The $hp$-FEM is proposed in \cite{melenk2010convergence,melenk2011wavenumber}, which is a FEM using local high order polynomials. It is shown that by choosing the degrees of local polynomials $p = O(\log k)$, the pollution effect can be suppressed in principle for the Helmholtz equation with constant $A, V$ and $\beta$. Nevertheless, to the best of our knowledge, there have been no theoretical results for this methodology when these coefficients become rough. \yc{There have been some recent developments for $hp$-FEM methods when piecewise regularity of the coefficients is assumed \cite{bernkopf2022wavenumber, lafontaine2022wavenumber}.} In general, it is well-known that polynomials might behave arbitrarily badly even for elliptic equations with rough coefficients \cite{babuvska2000can}.
 
Multiscale methods, on the other hand, have long been developed to address the difficulty associated with rough coefficients in elliptic equations. In particular, we mention the LOD and Gamblets related approaches \cite{malqvist_localization_2014,henning2013oversampling,owhadi2014polyharmonic,owhadi_multigrid_2017,owhadi2019operator,chen2019subsampled,chen2020multiscale}, variants of the multiscale finite element method (MsFEM) \cite{hou_multiscale_1997,efendiev_generalized_2013,hou2015optimal,chung2018constraint,fu2019edge,chen2020exponential} \yc{and generalized finite element based on partition of unity methods (PUM) \cite{babuska2011optimal, smetana2016optimal,buhr2018randomized,chen2020randomized,babuvska2020multiscale, schleuss2020optimal, ma2021error, ma2021novel}}, which are most related to this paper. Recently, the LOD method has been generalized to the case of Helmholtz equations with high wavenumber and heterogeneous media \cite{peterseim2017eliminating,gallistl2015stable,brown2017multiscale,peterseim2020computational}. They show that with a coarse mesh of size $O(H)$ and localized multiscale basis functions of support size $O(H\log(1/H)\log k)$, the pollution effect can be overcome once the stability constant of the solution operator of the Helmholtz equation is of at most polynomial growth. An error of at most $O(H)$ is established. \yc{Very recently, there is also a super-localized version of LOD-type method for the Helmholtz equations, proposed in \cite{freese2021super}, where the support of basis functions is further reduced to $O(H\log^{(d-1)/d}(k/H))$.} 

From the perspective of MsFEM methodology, the authors in \cite{fu2019wavelet} introduce WMsFEM to address the pollution effect successfully. Their basis functions are all of local support size $O(H)$. On the theoretical side, they require $O(k)$ number of basis functions in each element in order to achieve $O(H)$ accuracy. In contrast, our method in this paper, which can be viewed as a generalization of MsFEM, only requires $O(\log^{d+1} k)$ number of basis function of support size $O(H)$ in each element to handle the pollution effect and to achieve $O(H)$ accuracy. More importantly, our method yields an overall exponential rate of convergence regarding the number of basis functions, thanks to a systematic decomposition and treatment of coarse and fine scale parts of the solution.

\yc{In the literature, multiscale methods with exponential convergence for elliptic equations with rough media first appeared in \cite{babuska2011optimal},  which is based on local optimal basis approximation combined with the partition of unity method (PUM).
There has been a number of recent papers that are actively working on improving the methodology \cite{smetana2016optimal,buhr2018randomized,chen2020randomized,babuvska2020multiscale, schleuss2020optimal, ma2021error, ma2021novel}, aiming for more refined continuous and discrete analysis, randomized computation, efficient implementation, and generalization beyond elliptic equations. Our initial work \cite{chen2020exponential} on exponentially convergent multiscale methods for elliptic equations draws many motivations from these results, especially the Caccioppoli-type inequality that is essential for proving the exponential convergence. Different from the PUM based approach, our method is based on non-overlapped domain decomposition. More comparisons will be discussed in Subsection \ref{sec: contribution, motives}.
While revising this paper on solving the Helmholtz equations, we found that the authors in \cite{ma2021wavenumber} also proposed an exponentially convergent method for the Helmholtz equations using the PUM-based optimal local approximation methodology. }

In addition to those  methods mentioned above, there have also been several algorithms based on the MsFEM methodology \cite{oberai1998multiscale,fu2017fast} or the HMM methodology \cite{ohlberger2018new} with particular empirical success for solving the Helmholtz equation. It is also worth noting that, in conjunction with designing a good discretization scheme as above, one could also consider fast solvers for the discrete linear system. See, for example, the method of sweeping preconditioners \cite{engquist2011sweeping,engquist2012sweeping,poulson2013parallel}, where a preconditioning matrix is constructed to compute approximations of the Schur complements successively. Very recently, the LOD approach has also been combined with the hierarchical approach of Gamblets \cite{hauck2021multi} to get a multiresolution solver for the discrete system. 

\subsection{Main Contributions  {and Motivations}}
\label{sec: contribution, motives}
In this paper, we propose a multiscale framework for solving the Helmholtz equation in rough media and high frequency regimes, specifically in dimension $d=2$ where the mesh geometry of the non-overlapped domain decomposition is simplest to describe. Generalization to higher dimensions will be elaborated at the end of this paper. Our idea is based on a multiscale method in our previous work \cite{chen2020exponential} for solving elliptic equations with rough coefficients in an exponentially convergent manner. This paper aims to extend this framework to the more challenging Helmholtz equation where the operator is non-Hermitian and indefinite. It is perhaps surprising that the techniques in multiscale methods for elliptic equations can be systematically adapted  to the Helmholtz equation. \yc{Indeed, it has been proved in \cite{engquist2018approximate} that the Green function of the Helmholtz equations needs fundamentally polynomial on $k$ number of degrees of freedom to approximate, where they consider basis functions independent of the right hand side. Here, our results demonstrate that one can actually compress the solution operator exponentially efficiently by adding a number of local basis functions that depend on the local information of the right hand side. This shows that one can still achieve significant compression of the high frequency Helmholtz solution operator with rough coefficients by developing a data-driven compression operator adapted to the right hand side. }

We outline the main contributions of this paper below.
\begin{enumerate}
    \item In studying the solution behavior of the Helmholtz equation \eqref{eqn:Helmholtz smooth k}, we introduce a coarse-fine scale decomposition of its solution space. This decomposition is adapted to the coarse mesh structure; a mesh size of $O(1/k)$ suffices to make this coarse-fine scale decomposition well defined. Moreover, the decomposition is adapted to the coefficients $A, V, \beta$ and the wavenumber $k$. 
    \item Analytically, we show the fine scale part is of $O(H)$ in the energy norm, and it can be computed efficiently by solving the Helmholtz equations locally. Meanwhile, we prove that the space of the coarse scale part is of low complexity, such that there exist local multiscale basis functions that can approximate this part in a nearly exponentially convergent manner.  These serve as the cornerstone of our multiscale numerical method. 
    \item Numerically, we propose a multiscale framework that solves the two parts separately. The nearly exponential rate of convergence in the energy norm and $L^2$ norm is theoretically proved in this paper.
    \item Experimentally, we conduct a number of numerical tests and observe consistently that our multiscale methods give a nearly exponential rate of convergence, even for problems with high-contrast media. Based on these numerical studies, several recommendations for efficient implementations of our methods are provided, especially on how to design the offline and online computation to handle multiple right hand sides efficiently.
\end{enumerate}

{To the best of our knowledge, this multiscale framework is the first one that can be proved rigorously to achieve a nearly exponential rate of convergence in solving \eqref{eqn:Helmholtz smooth k} with rough $A,\beta, V$ and large $k$, especially for $d=2$. It generalizes our previous work on exponential convergence for solving rough elliptic equations \cite{chen2020exponential}, which is motivated by the PUM approach using optimal local approximation spaces for elliptic equations \cite{babuska2011optimal}.}

{Different from the PUM that uses an overlapped domain decomposition, our method relies on non-overlapped domain decomposition and an edge coupling approach to combine local basis functions as in MsFEM. Our coarse-fine scale decomposition of the solution space is built on this non-overlapped edge coupling. For elliptic equations, this decomposition is the same as the orthogonal decomposition in previous work of MsFEM \cite{hou2015optimal, chen2020exponential} and approximate component mode synthesis \cite{hetmaniuk2010special,hetmaniuk2014error}. Under this line of methodology, this paper 
contributes a principled framework for obtaining nearly exponentially convergent basis functions for multiscale Helmholtz equations. }

{There are many differences between the multiscale methods based on PUM and edge coupling. Basically, the support of basis functions in PUMs is usually larger than that of MsFEMs since non-overlapped domain decomposition leads to smaller decomposed domains than its overlapped counterpart. There is no need to introduce additional freedom of partition of unity functions as well.
On the other hand, in 2D, the number of local edges could be twice as many as the number of local domains, leading to more work in constructing the basis functions. Moreover, there will be an increasing design complexity for the non-overlapped edge coupling approach for higher-dimensional problems since the boundaries of high dimensional local domains become more complicated. This is why in this paper, we dedicate specifically to 2D Helmholtz equations \yc{for detailed analysis and numerical experiments. We leave the discussion on generalizations to 3D problems in Section \ref{sec 3d}.}}

{We are not going to dive very deeply about the fundamental comparison between overlapped and non-overlapped decomposition in multiscale methods. The aim of this paper is to demonstrate that one could achieve a nearly exponential convergence rate theoretically using the non-overlapped edge coupling framework in a principled way and show that this method is very competitive numerically. A number of technical difficulties, such as the appropriate approximation space for the edge functions and the spectral analysis of the local restriction operator, are carefully addressed to lay out this framework. We believe this work could help future researchers understand and analyze multiscale methods that are built on different local decomposition and global coupling approaches.} 


Lastly, we remark that in principle, our multiscale algorithm can be applied to general Helmholtz equations numerically, while most of our theoretical results rely on analytical properties of the solution to equation \eqref{eqn:Helmholtz smooth k}, related to the well-posedness, stability and  $C^{\alpha}$ estimates. Therefore, typical conditions (usually very mild) of these analytical properties will be assumed in this paper, in order to get a rigorous theory. We will mention several references to these results in this paper. Some numerical examples in which these assumptions are violated will be also presented to illustrate the effectiveness of our algorithm in a general context. 

\subsection{Organization of the paper} The rest of this paper is organized as follows. In Section \ref{sec: preliminaries}, we review preliminary results for the Helmholtz equation, including the well-posedness, stability, adjoint problems, and H\"older $C^{\alpha}$ estimates. Section \ref{construct} is devoted to analyzing the solution space based on a coarse-fine scale decomposition. Moreover, the computational properties of the coarse and fine parts are rigorously studied in detail. Building upon these properties, in Section \ref{methods} we develop the multiscale computational framework and prove the nearly exponential rate of convergence for  our multiscale methods. The detailed numerical algorithms are discussed and implemented in Section \ref{sec-numeric-experiments} for several Helmholtz equations. To improve the readability of our paper,  some technical proofs of theorems and propositions will be deferred to Section \ref{theory}. Some concluding remarks are made in Section \ref{sec-concluding-remark}.

\section{Preliminaries on the Helmholtz Equation}
\label{sec: preliminaries}
Our multiscale algorithm relies on an in-depth understanding of the solution space of \eqref{eqn:Helmholtz smooth k}. To achieve this, we first present several analytic results for \eqref{eqn:Helmholtz smooth k}, which will serve as preliminaries for our subsequent discussions. We cover the weak formulation, the well-posedness of the equation, the stability estimates of the solution, and H\"older estimates. 
\subsection{Notations} We use $H^1(\Omega)$ to denote the standard complex Sobolev space in $\Omega$, containing $L^2$ functions with $L^2$ first order derivatives. We write $(u,v)_{D}:=\int_D u \bar{v}$ for any domain $D$. We use $C$ as a generic constant, and its value can change from place to place; we will state explicitly the parameters that this constant may or may not depend on.
\subsection{Analytic Results}
\label{sec-analytic-results}
For the model problem \eqref{eqn:Helmholtz smooth k}, we consider the complex Sobolev space $\mathcal{H}(\Omega):=\{u \in H^1(\Omega): u|_{\Gamma_D}=0 \}$ in which functions have zero trace on the Dirichlet boundary. This space is equipped with the norm $\|\cdot\|_{\cH(\Omega)}$ such that
\begin{equation*}
\|u\|_{\mathcal{H}(\Omega)} :=\int_{\Omega} A|\nabla u|^2+k^2V^2|u|^2\,.
\end{equation*}
The dual space of ${\mathcal{H}(\Omega)}$ is denoted by ${\mathcal{H}^{-1}(\Omega)}$ equipped with the norm $\|\cdot \|_{{\mathcal{H}^{-1}(\Omega)}}$; by definition one has
\[\|f\|_{\cH^{-1}(\Omega)}:=\sup_{v \in \cH(\Omega)} \frac{|(f,v)_{\Omega}|}{\|v\|_{\cH(\Omega)}}\, . \]
Now, we present several analytic results pertaining to the Helmholtz equation \eqref{eqn:Helmholtz smooth k}.

\textit{Weak formulation.} The weak formulation of \eqref{eqn:Helmholtz smooth k} is given by
\begin{equation}
\label{eqn: Helmholtz smooth weak form}
a(u, v) :=(A\nabla u, \nabla {v})_{\Omega}-k^{2}(V^2 u,  {v})_{\Omega}  -( T_{k} u,  {v})_{\Gamma_N\cup\Gamma_R} =( f,  {v})_{\Omega} , \quad \forall v \in \cH(\Omega)\, .
\end{equation}

\textit{Continuity estimate.} By the Cauchy-Schwarz and trace inequalities (see Lemma 3.1 of \cite{melenk2010convergence}), the sesquilinear form $a(\cdot,\cdot)$ is bounded on $\cH(\Omega)$ with a constant $C_c$ independent of $k$, i.e., for any $u,v \in \cH(\Omega)$, one has the continuity estimate:
\begin{equation}
\label{eqn: continuity estimate}
|a(u, v)|\leq C_c \|u\|_{\mathcal{H}(\Omega)}\|v\|_{\mathcal{H}(\Omega)}\, .
\end{equation}

\textit{Well-posedness and stability.} If $\Gamma_R$ has positive $d-1$ dimensional measure, then under some mild conditions (see Assumption 2.3 and Theorem 2.4 in \cite{graham2020stability}), problem \eqref{eqn: Helmholtz smooth weak form} admits a unique solution given the right hand side $f \in L^2(\Omega)$. We will assume these conditions. Let the solution operator be $N_k$, so that $u =N_{k} f$. 
Under the same conditions, this operator is stable (Theorem 2.4 in \cite{graham2020stability}) in the sense that
\begin{equation}
\label{eqn: stability}
    C_{\mathrm{stab}}(k) := \sup _{f \in L^{2}(\Omega) \backslash\{0\}} \frac{\left\|N_{k} f\right\|_{\mathcal{H}(\Omega)}}{\|f\|_{L^{2}(\Omega)}} <\infty \, .
\end{equation}
To avoid getting into detailed discussions of these assumptions and for simplicity of presentation, we will base most of our arguments on assuming \eqref{eqn: stability} holds.

The stability constant $C_{\mathrm{stab}}(k)$ will depend on $k$ in general, and obtaining an explicit characterization of this dependence has been a hard task; see \cite{betcke2011condition,brown2017multiscale,graham2019helmholtz,moiola2019acoustic,sauter2018stability}. A prevalent and reasonable assumption on the constant is that of polynomial growth, namely $C_{\mathrm{stab}}(k)\leq C(1+k^\gamma)$ for some constants $\gamma$ and $C$; see for example \cite{lafontaine2019most}. We are not going into detailed discussions on this assumption here, while we mention that the final error estimate of our numerical solution in this paper will depend on $C_{\mathrm{stab}}(k)$ explicitly; thus, those estimates on $C_{\mathrm{stab}}(k)$ in the literature can be readily applied to our context.

In addition, stability for $f \in L^2(\Omega)$ can yield well-posedness and stability for $f \in {\mathcal{H}^{-1}(\Omega)}$. According to Lemma 2.1 in \cite{peterseim2017eliminating} and also \cite{esterhazy2012stability}, one has
\begin{equation}
\label{eqn: wellposed H -1 rhs}
    \sup _{f \in \mathcal{H}^{-1}(\Omega) \backslash\{0\}} \frac{\left\|N_{k} f\right\|_{\mathcal{H}(\Omega)}}{\|f\|_{\mathcal{H}^{-1}(\Omega)}} \leq kC_{\mathrm{stab}}(k)\, .
\end{equation}

\textit{Adjoint problems.} Due to the presence of the Robin boundary condition, $a(\cdot,\cdot)$ is not Hermitian. Its adjoint sesquilinear form is defined as $a^*(u,v)=\overline{a(v,u)}$. The adjoint problem for \eqref{eqn: Helmholtz smooth weak form} is given by $a^*(u,v)=(f,v)_{\Omega}$ for any $v \in \cH(\Omega)$.
It also corresponds to the following PDE: 
\begin{equation*}
\left\{
\begin{aligned}
-\nabla \cdot(A\nabla u)-k^{2}V^2 u&=f, \ \text {on} \ \Omega\\
u&=0, \ \text{in} \ \Gamma_D\\
A\nabla u\cdot\nu&=T_k^*u, \ \text{on} \ \Gamma_N \cup \Gamma_R \, ,
\end{aligned}
\right.
\end{equation*}
where $T^*_ku := \overline{T_k u}=-T_k u$.
The adjoint solution operator is denoted by $N_k^*$. One can readily check that $N^*_k \overline{f}=\overline{N_k{f}}$. Therefore, the adjoint problem admits the same stability constant as the original problem; namely it holds 
\begin{equation*}
C_{\mathrm{stab}}(k) = \sup _{f \in L^{2}(\Omega) \backslash\{0\}} \frac{\left\|N_{k}^{\star}
f\right\|_{\mathcal{H}(\Omega)}}{\|f\|_{L^{2}(\Omega)}} <\infty \, .
\end{equation*}
The adjoint problem will play a valuable role when we analyze the convergence property of our multiscale methods for the Helmholtz equation.

\textit{$C^{\alpha}$ H\"older regularity.} We will need the $C^{\alpha}$ estimates of the solution in order to demonstrate the theoretical properties of our multiscale methods. 
\begin{proposition}
\label{prop: C alpha estimate}
Suppose $d\leq 3$ and \eqref{eqn: stability} holds. If $f \in L^2(\Omega)$, then the solution $u \in C^{\alpha}(\Omega)$ for some $\alpha \in (0,1)$.
\end{proposition}
We defer the proof of this proposition to Subsection \ref{subsec: Proof of Proposition prop: C alpha estimate}.
\begin{remark}
\yc{
The global regularity estimate may depend on the wavenumber $k$. Nevertheless, we only use it to show qualitatively that our solution is continuous, so that the nodal interpolation in Subsection \ref{sec-Localization of Approximation} is mathematically rigorous. Later, when we derive error estimates of our methods, we will only use the local version of the regularity estimate, where the constant is independent of the wavenumber; see Lemma \ref{lemma: bound H 1/2 by energy norm}. }
\end{remark}
We have presented several critical analytic results for the Helmholtz equation. 
Based on these results, we are now ready to study the solution space of \eqref{eqn:Helmholtz smooth k} in the next section. The key is a coarse-fine scale decomposition of the solution space, which will play an essential role in designing our multiscale algorithms.

\section{Coarse-Fine Scale Decomposition}
\label{construct}
In this section, we develop a coarse-fine scale operator-adapted decomposition of the solution space. This decomposition is adaptive to the mesh structure, and a mesh of size $H=O(1/k)$ suffices to make this coarse-fine scale decomposition well defined. We discuss the setting of the mesh structure in Subsection \ref{geometri}, followed by introducing the coarse-fine scale decomposition in Subsection \ref{de comp}. In Subsection  \ref{sec-small-bubble} we show the fine scale part is local and small up to $O(H)$ in the $\cH(\Omega)$ norm. In Subsection \ref{sec-low-complexity-harmonic-part} we show the coarse-scale component can be approximated via local edge basis functions in a nearly exponentially convergent manner.

\subsection{Mesh Structure}
\label{geometri}
We begin by discussing related concepts of the mesh structure. The focus here is on $d=2$ while remarks on generalization to $d\geq 3$ will be discussed in Section \ref{sec-concluding-remark}. In the mesh structure, we discuss two dimensional elements in Subsection \ref{subsec: mesh structure of elements}, one dimensional edges and zero dimensional nodes, and their neighborship in Subsection \ref{sec-nodes-edges}. See also Figure \ref{fig:mesh geometry} for illustrations.
\subsubsection{Elements} 
\label{subsec: mesh structure of elements}
We consider a shape regular and uniform partition of the domain $ \Omega $ into finite elements, such as triangles and quadrilaterals. The collection of elements is denoted by $\cT_H=\{T_1,T_2,...,T_r\}$.  \yc{For simplicity, we assume that each connected component of the domain is at least partitioned into two elements.}

The mesh size is $H$, i.e., $\max_{T \in \cT_H} \operatorname{diam}(T)=H$. 
The uniformity of the mesh implies $\min_{T \in \cT_H} \operatorname{diam}(T)\geq c_0 H$ for some $0<c_0\leq 1$ that is independent of $H$ and $T$. The shape regularity property implies there is a constant $c_1>0$ independent of $H$ and $T$, such that $\max_{T \in \cT_H} \operatorname{diam}(T)^d /{|T|} \leq c_1$, where $|T|$ is the volume of $T$. 

In this mesh, by using a scaling argument, the following Poincar\'e inequality will hold uniformly for $T \in \cT_H$. This inequality will be  used frequently later.
\begin{proposition}[The Poincar\'e inequality]
\label{prop: poincare inequality}
For any $T \in \cT_H$ and a function $v \in H^1(T)$ that vanishes on one of the edges of $T$, it holds that
    \begin{equation}
\label{Poincaree} \|v\|_{L^2(T)} \leq C_P H \|\nabla v\|_{L^2(T)}\, ,
\end{equation}
where $C_P$ depends on $c_0,c_1$ and $d$.
\end{proposition}

\subsubsection{Nodes, Edges and Their Neighbors} 
\label{sec-nodes-edges}
Let $\cN_H=\{x_1,x_2,...,x_p\}$ be the collection of interior nodes, and $\cE_H=\{e_1,e_2,...,e_q\}$ be the collection of edges except those fully on the boundary of $\Omega$. An edge $e \in \cE_H$ is defined such that there exists two different elements $T_i,T_j$ with $e=\overline{T}_i\bigcap\overline{T}_j$ that has co-dimension $1$ in $\bR^d$. We will use $E_H=\bigcup_{e \in \cE_H} e \subset \Omega$ to denote the edges as a whole set. 

We use the symbol $\sim$ to describe the neighbourship between nodes, edges and elements. More precisely, if we consider a node $x \in \cN_H$, an edge $e \in \cE_H$, and an element $T \in \cT_H$, then, (1) $x \sim e$ denotes $x \in e$; (2) $e \sim T$ denotes $e \subset \overline{T}$; (3) $x \sim T$ denotes $x \in \overline{T}$. The relationship $\sim$ is symmetric.

We use $\rN(\cdot,\cdot)$ to describe the union of neighbors as a set. For example, 
$\rN(x,\cE_H)=\bigcup \{e \in \cE_H: e \sim x\} \subset E_H$, $\rN(x,\cT_H)=\bigcup \{T \in \cT_H: T \sim x\} \subset \Omega$, and $\rN(e,\cT_H)=\bigcup \{T \in \cT_H: T \sim e\} \subset \Omega$. 
\begin{figure}[t]
    \centering
\tikzset{every picture/.style={line width=0.75pt}} 

\begin{tikzpicture}[x=0.75pt,y=0.75pt,yscale=-1,xscale=1]

\draw  [draw opacity=0][dash pattern={on 4.5pt off 4.5pt}] (45,47) -- (201,47) -- (201,203) -- (45,203) -- cycle ; \draw  [dash pattern={on 4.5pt off 4.5pt}] (84,47) -- (84,203)(123,47) -- (123,203)(162,47) -- (162,203) ; \draw  [dash pattern={on 4.5pt off 4.5pt}] (45,86) -- (201,86)(45,125) -- (201,125)(45,164) -- (201,164) ; \draw  [dash pattern={on 4.5pt off 4.5pt}] (45,47) -- (201,47) -- (201,203) -- (45,203) -- cycle ;
\draw   (123,86) -- (162,86) -- (162,125) -- (123,125) -- cycle ;
\draw   (123.5,124) .. controls (123.38,128.67) and (125.65,131.06) .. (130.31,131.18) -- (131.68,131.22) .. controls (138.34,131.39) and (141.61,133.81) .. (141.49,138.48) .. controls (141.61,133.81) and (145,131.57) .. (151.67,131.74)(148.67,131.67) -- (154.32,131.81) .. controls (158.99,131.94) and (161.38,129.67) .. (161.5,125) ;
\draw   (126.13,86) .. controls (126.13,84.27) and (124.73,82.88) .. (123,82.88) .. controls (121.27,82.88) and (119.88,84.27) .. (119.88,86) .. controls (119.88,87.73) and (121.27,89.13) .. (123,89.13) .. controls (124.73,89.13) and (126.13,87.73) .. (126.13,86) -- cycle ;
\draw  [draw opacity=0][dash pattern={on 4.5pt off 4.5pt}] (364.5,79) -- (516.5,79) -- (516.5,199) -- (364.5,199) -- cycle ; \draw  [dash pattern={on 4.5pt off 4.5pt}] (378.5,79) -- (378.5,199)(417.5,79) -- (417.5,199)(456.5,79) -- (456.5,199)(495.5,79) -- (495.5,199) ; \draw  [dash pattern={on 4.5pt off 4.5pt}] (364.5,93) -- (516.5,93)(364.5,132) -- (516.5,132)(364.5,171) -- (516.5,171) ; \draw  [dash pattern={on 4.5pt off 4.5pt}]  ;

\draw   (259.13,87) .. controls (259.13,85.27) and (257.73,83.88) .. (256,83.88) .. controls (254.27,83.88) and (252.88,85.27) .. (252.88,87) .. controls (252.88,88.73) and (254.27,90.13) .. (256,90.13) .. controls (257.73,90.13) and (259.13,88.73) .. (259.13,87) -- cycle ;
\draw  [draw opacity=0][dash pattern={on 4.5pt off 4.5pt}] (217,48) -- (295,48) -- (295,126) -- (217,126) -- cycle ; \draw  [dash pattern={on 4.5pt off 4.5pt}] (256,48) -- (256,126) ; \draw  [dash pattern={on 4.5pt off 4.5pt}] (217,87) -- (295,87) ; \draw  [dash pattern={on 4.5pt off 4.5pt}] (217,48) -- (295,48) -- (295,126) -- (217,126) -- cycle ;

\draw   (298.13,176) .. controls (298.13,174.27) and (296.73,172.88) .. (295,172.88) .. controls (293.27,172.88) and (291.88,174.27) .. (291.88,176) .. controls (291.88,177.73) and (293.27,179.13) .. (295,179.13) .. controls (296.73,179.13) and (298.13,177.73) .. (298.13,176) -- cycle ;
\draw  [dash pattern={on 4.5pt off 4.5pt}]  (291.88,176) -- (334,176) ;
\draw  [dash pattern={on 4.5pt off 4.5pt}]  (295,137) -- (295,176) ;
\draw  [dash pattern={on 4.5pt off 4.5pt}]  (256,176) -- (295,176) ;
\draw  [dash pattern={on 4.5pt off 4.5pt}]  (295,176) -- (295,215) ;

\draw [line width=0.75]    (307,87) -- (346,87) ;
\draw  [draw opacity=0][dash pattern={on 4.5pt off 4.5pt}] (307,48) -- (346,48) -- (346,126) -- (307,126) -- cycle ; \draw  [dash pattern={on 4.5pt off 4.5pt}]  ; \draw  [dash pattern={on 4.5pt off 4.5pt}]  ; \draw  [dash pattern={on 4.5pt off 4.5pt}] (307,48) -- (346,48) -- (346,126) -- (307,126) -- cycle ;

\draw (136,97) node [anchor=north west][inner sep=0.75pt]   [align=left] {$\displaystyle T$};
\draw (137,137) node [anchor=north west][inner sep=0.75pt]   [align=left] {$\displaystyle e$};
\draw (110,67) node [anchor=north west][inner sep=0.75pt]   [align=left] {$\displaystyle x$};
\draw (33,223) node [anchor=north west][inner sep=0.75pt]   [align=left] { $\displaystyle \quad x\in \mathcal{N}_{H} ,e\in \mathcal{E}_{H} ,T\in \mathcal{T}_{H}$};
\draw (262,66) node [anchor=north west][inner sep=0.75pt]   [align=left] {$\displaystyle x$};
\draw (279,156) node [anchor=north west][inner sep=0.75pt]   [align=left] {$\displaystyle x$};
\draw (323,68) node [anchor=north west][inner sep=0.75pt]   [align=left] {$\displaystyle e$};
\draw (218,136) node [anchor=north west][inner sep=0.75pt]   [align=left] {$\displaystyle\ \  \rN( x,\mathcal{T}_{H})$};
\draw (301,135) node [anchor=north west][inner sep=0.75pt]   [align=left] {$\displaystyle \ \rN( e,\mathcal{T}_{H})$};
\draw (267,221) node [anchor=north west][inner sep=0.75pt]   [align=left] {$\displaystyle \ \rN( x,\mathcal{E}_{H})$};
\draw (391,219) node [anchor=north west][inner sep=0.75pt]   [align=left] {A fraction of $\displaystyle E_{H}$};

\end{tikzpicture}
    \caption{Geometry of the mesh}
    \label{fig:mesh geometry}
\end{figure}
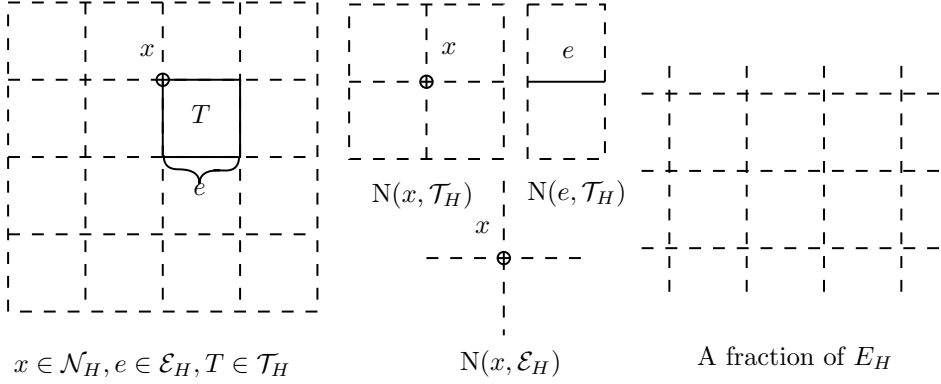

\subsection{Decomposition of Solution Space}
\label{de comp}
With the mesh structure defined, we now discuss the coarse-fine scale decomposition of the solution space. We first discuss decomposition in the local element $T$ in Subsection \ref{sec-Local decomposition} and then the global decomposition in Subsection \ref{sec-Global decomposition}.
 
\subsubsection{Local decomposition} 
\label{sec-Local decomposition}
A crucial requirement for the decomposition to be well defined is that the mesh size is order $O(1/k)$; see Assumption \ref{small mesh}. As we will see later, this bound on $H$ ensures that local Helmholtz problems in each element have properties that are similar to those of elliptic problems; thus, techniques in elliptic equations can then be applied.
\begin{assumption}
\label{small mesh}
The mesh size satisfies $H\leq {A}_{\min}^{1/2}/(\sqrt{2}C_P V_{\max} k)$, where $C_P$ is the constant in Proposition \ref{prop: poincare inequality}.
\end{assumption}

Given Assumption \ref{small mesh}, we decompose\footnote{This decomposition is inspired by that in the elliptic case \cite{chen2020exponential}.} $u$ into two parts $u=u_{T}^{\sfh}+u_T^{\sfb}$ in each element $T \in \cT_H$. The two components satisfy: 
\begin{equation}
\label{eqn:Helmholtz decomposed}
\begin{aligned}
    &\left\{
    \begin{aligned}
    -\nabla \cdot (A \nabla u_T^\sfh )-k^{2}V^2 u^\sfh_T&=0, \ \text{in} \  T\\
    u_T^\sfh&=u, \ \text{on} \  \partial T \setminus (\Gamma_N \cup \Gamma_R)\\
    A\nabla u_T^\sfh\cdot\nu&=T_k u_T^\sfh,\  \text{on} \ \partial T \cap (\Gamma_N \cup \Gamma_R)\, ,
    \end{aligned}
    \right.
    \\
    &\left\{
    \begin{aligned}
    -\nabla \cdot (A \nabla u^\sfb_T )-k^{2}V^2 u^\sfb_T&=f,\  \text{in} \  T\\
    u^\sfb_T&=0, \ \text{on} \  \partial T\setminus (\Gamma_N \cup \Gamma_R)\\
   A\nabla u^\sfb_T\cdot\nu&=T_k u^\sfb_T, \ \text{on} \ \partial T \cap (\Gamma_N \cup \Gamma_R) \, .
    \end{aligned}
    \right.
    \end{aligned}
\end{equation}
In short, the part $u^{\sfh}_T$ incorporates the boundary value of $u$, while $u^{\sfb}_T$ contains information of the right hand side. Both equations in \eqref{eqn:Helmholtz decomposed} should be understood in the standard weak sense using the following local sesquilinear form $a_T(\cdot,\cdot)$ in $T$:
\begin{equation}
a_T(v, w) :=(A\nabla v, \nabla {w})_{T}-k^{2}(V^2 v,  {w})_{T}  -( T_{k}v,  {w})_{\partial T \cap (\Gamma_N\cup\Gamma_R)}\, \text{ for } v,w \in \cH(T) \, ,
\end{equation}
where $\cH(T):=\cH(\Omega)|_T$, the restriction of $\cH(\Omega)$ in the domain $T$. The well-posedness of the two problems is due to the following proposition:
\begin{proposition}
\label{prop-well-definedness-local-prob}
Under Assumption \ref{small mesh}, for $v \in \cH(T)$ that vanishes on one of the edges of $T$, the corresponding sesquilinear form is coercive such that
\[\Re a_T(v,v)\geq \frac{1}{2} \|A^{1/2}\nabla v\|_{L^2(T)}^2\, .  \]
\end{proposition}
\begin{proof}
Using the Poincar\'e inequality \eqref{Poincaree} and Assumption \ref{small mesh}, we get
\begin{equation}
\label{poinmath}
\begin{aligned}
   \Re a_T(v,v)&=\|A^{1/2}\nabla v\|_{L^2(T)}^2- \|kVv\|^2_{L^2(T)} \\
   &\geq (1-C_P^2H^2k^2V_{\max}^2A_{\min}^{-1})\|A^{1/2}\nabla v\|_{L^2(T)}^2 
   \geq \frac{1}{2} \|A^{1/2}\nabla v\|_{L^2(T)}^2 \, .
  \end{aligned}
\end{equation}
\end{proof}
Since both equations in \eqref{eqn:Helmholtz decomposed} contain Dirichlet's boundary condition on at least one of the edges of $T$, the coercivity implied by Proposition \ref{prop-well-definedness-local-prob} suffices for the well-posedness. Consequently, the solutions $u^{\sfh}_T$ and $u^{\sfb}_T$ are well-defined.
\begin{remark}
\label{rmk:orthogonality}
An important property is that $u^{\sfh}_T$ is ``left-orthogonal'' to $u^{\sfb}_T$ in $T$ with respect to the local sesquilinear form $a_T(\cdot,\cdot)$ in $T$, in the sense of $a_T(u^{\sfh}_T,u^{\sfb}_T)=0$, according to the weak form of the equation. Note that we might not have $a_T(u^{\sfb}_T,u^{\sfh}_T)=0$ for $T$ near the boundary (i.e., $\partial T \cap (\Gamma_N\cup\Gamma_R) \neq \emptyset$) due to the fact that $a_T(\cdot,\cdot)$ is not Hermitian here. 
\end{remark}

 \subsubsection{Global decomposition} 
 \label{sec-Global decomposition}
 In this subsection, we define a global decomposition $u=u^{\sfb}+u^{\sfh}$, such that for each $T$, it holds that $u^{\sfh}(x)=u^{\sfh}_T(x)$ and $u^{\sfb}(x)=u^{\sfb}_T(x)$ when $x \in T$. 
 Both $u^\sfh$ and $u^\sfb$ are well-defined and belong to $\mathcal{H}(\Omega)$ due to the continuity across edges. Here, the component $u^{\sfh}_T$ (resp. $u^{\sfh}$) is called the local (resp. global) \textit{Helmholtz-harmonic part} and $u^{\sfb}_T$ (resp. $u^{\sfb}$) is the local (resp. global) \textit{bubble part}, of the solution $u$. 
 
 We further introduce the function space for the Helmholtz-harmonic part 
 \begin{equation}
     \begin{aligned}
         V^\sfh:=\{v \in \cH(\Omega): &-\nabla \cdot (A \nabla v)-k^2V^2v=0 \text{ in each } T\in \cT_H, \\
        &A\nabla v\cdot\nu=T_kv, \text{ on } \Gamma_N \cup \Gamma_R \}\, , 
     \end{aligned}
 \end{equation}
so that $u^\sfh \in V^\sfh$, and the space for the bubble part
  \begin{equation}
     \begin{aligned}
         V^\sfb:=\{v \in \cH(\Omega): v = 0 \text{ on } E_H \} \, ,
     \end{aligned}
 \end{equation}
 such that $u^\sfb \in V^\sfb$. In this way, the solution space of \eqref{eqn:Helmholtz smooth k} can be decomposed to $V^\sfh + V^\sfb$. Furthermore, 
 for any $v \in V^{\sfh} $ and $w \in V^{\sfb}$, it holds that $a(v,w)=0$ by summing up local sesquilinear forms $a_T(\cdot,\cdot)$ and using Remark \ref{rmk:orthogonality}.
 
 We will treat $V^{\sfb}$ as the fine scale or microscopic space, and refer to $V^{\sfh}$ as the coarse scale or macroscopic space. The idea of our multiscale framework is to compute the two parts separately by exploring their own structures.
  
 In the next two subsections, we will study the computational properties of $u^\sfh \in V^\sfh$ and $u^\sfb \in V^\sfb$, respectively. These properties serve as the cornerstone of designing our multiscale algorithm.

\subsection{Local and Small Bubble Part}
\label{sec-small-bubble}
In this subsection, we analyze the bubble part $u^\sfb$. This part depends locally on $f$ in each $T$. Thus, it can be computed efficiently in a parallel manner. Moreover, it is small and can be ignored if the target accuracy is $O(H)$; see Proposition \ref{prop: bubble is small}. 
\begin{proposition}
\label{prop: bubble is small}
Under Assumption \ref{small mesh}, it holds that
\begin{equation}
\label{thm: estimate for bubble part, interior}
\left\|u^{\sfb}\right\|_{\mathcal{H}(\Omega)} \leq  \frac{3C_P}{A_{\min}^{1/2}}H\|f\|_{L^{2}(\Omega)}\, .
\end{equation}
\end{proposition}
\begin{proof}
By definition, inside each patch $T$, it holds that $a_T(u^{\sfb},u^{\sfb})=(f,u^{\sfb})_{T}$. The coercivity estimate in  \eqref{poinmath} implies the inequality $\|kVu^{\sfb}\|^2_{L^2(T)} \leq \frac{1}{2} \|A^{1/2}\nabla u^{\sfb}\|_{L^2(T)}^2$. Using the estimate, we get 
\begin{equation*}
\begin{aligned}
    \Re a_T(u^{\sfb},u^{\sfb}) &=\|A^{1/2}\nabla u^{\sfb}\|_{L^2(T)}^2- \|kVu^{\sfb}\|^2_{L^2(T)}    \\
    &\geq \frac{1}{3} (\|A^{1/2}\nabla u^{\sfb}\|_{L^2(T)}^2+\|kVu^{\sfb}\|^2_{L^2(T)})= \frac{1}{3} \|u^{\sfb}\|^2_{\mathcal{H}(T)}\, .
\end{aligned}
\end{equation*}
Combining the above estimate with the Cauchy-Schwarz inequality, we arrive at
\[\|u^{\sfb}\|^2_{\mathcal{H}(T)}\leq 3\Re a_T(u^{\sfb},u^{\sfb}) = 3(f,u^{\sfb})_{T}\leq 3\|f\|_{L^2(T)}\|u^{\sfb}\|_{L^2(T)}\,.\]
Meanwhile, by the Poincar\'e inequality \eqref{Poincaree}, we get  \[\|u^{\sfb}\|_{L^2(T)}\leq C_P H\|\nabla u^{\sfb}\|_{L^2(T)}\leq \frac{C_P H}{ A_{\min}^{1/2}} \|u^{\sfb}\|_{\mathcal{H}(T)}\, .\] 
Combining all the above inequalities gives $\|u^{\sfb}\|_{\mathcal{H}(T)}\leq 3(C_P H/ A_{\min}^{1/2})\|f\|_{L^2(T)}$ for each element $T$. Summing them up for all elements $T$ yields the desired conclusion.
\end{proof}

\subsection{Low Complexity of the Helmholtz-Harmonic Part}
\label{sec-low-complexity-harmonic-part}
Now, we turn to the study of the Helmholtz-harmonic part $u^\sfh$. The goal is to show that $u^\sfh$ can be approximated via local basis functions in an exponentially efficient manner. 
To achieve this, our approximation framework\footnote{It is similar to that in our previous work for elliptic equations \cite{chen2020exponential}.} contains three steps: (1) reducing the approximation of $u^\sfh$ to that of edge functions in Subsection \ref{sec-Approximation via Edge Functions}; (2) localizing the approximation to every single edge in Subsection \ref{sec-Localization of Approximation}; and (3) realizing local approximation via oversampling and SVD in Subsection \ref{oversampling exp}. Combining all these three steps, we establish the low complexity in approximation of $u^\sfh$ in Subsection \ref{subsec:Low Complexity in Approximation}.

\subsubsection{Approximation via Edge Functions}
\label{sec-Approximation via Edge Functions}
We start with the first step of approximating $u^\sfh$. By definition, $u^\sfh$ belongs to $V^{\sfh}$. A key observation is that any function in $V^{\sfh}$ is determined entirely by its value on the edge set $E_H$. Thus, define  \[\tilde{V}^{\sfh}:=\{\tilde{\psi}: E_H \to \bR, \text{ there exists a function } \psi \in V^{\sfh}, \text{ such that } \tilde{\psi}=\psi|_{E_H} \}\, ;\] 
then \yc{under Assumption \ref{small mesh}}, there is a one to one correspondence $\tilde{\psi} \in \tilde{V}^{\sfh} \leftrightarrow \psi \in V^{\sfh}$. More precisely, in each $T$, it holds that
  \begin{equation}
    \label{eqn: edge bulk correspondence} 
    \left\{
    \begin{aligned}
    -\nabla \cdot (A \nabla \psi)-k^2V^2\psi&=0, \quad \text{in} \  T\\
    \psi&=\tilde{\psi}, \quad \text{on} \  \partial T\setminus (\Gamma_N \cup \Gamma_R)\\
    A\nabla \psi\cdot\nu&=T_k \psi, \ \text{on} \ \partial T \cap (\Gamma_N \cup \Gamma_R)\, .
    \end{aligned}
    \right.
    \end{equation}
    Indeed, we have $\tilde{V}^{\sfh} = H^{1/2}(E_H)$ by the trace theory {since the local equation is elliptic}. Using the above identification, approximating $u^\sfh$ corresponds to approximating $\tilde{u}^\sfh$, which is a function defined on edges and of lower complexity. {We need to pay attention to the norm we use when approximating $\tilde{u}^\sfh$ so that we can use the error bound of the approximation to control the  error of $u^\sfh$ in the energy norm. This will be the focus of the next section.}
    \begin{remark}
    In the remaining part of the article, we will frequently use the correspondence between $V^\sfh$ and $\tilde{V}^\sfh$. Conventionally, when we write a tilde on the top of a function in $V^\sfh$, it refers to its corresponding part in $\tilde{V}^\sfh$.
    \end{remark}
\subsubsection{Localization of Approximation}
\label{sec-Localization of Approximation} 
We discuss how to approximate the edge function $\tilde{u}^\sfh$, whose domain is $E_H$, which is nonlocal. Since it is often preferable to have localized basis functions for approximation and numerical algorithms, our second step is to localize the task of approximating $\tilde{u}^\sfh$ to every single edge.

To achieve localization, we study the geometry of the edge set $E_H$ first. Observing that different edges only communicate with each other along their shared nodes, we can use nodal interpolation to localize the approximation. More precisely, we proceed with the following steps:
\begin{enumerate}
    \item Interpolation: for each node $x_i\in \cN_H$, choose $\tilde{\psi}_i$ to be the piecewise linear tent function on $E_H$, satisfying $\tilde{\psi}_i(x_j)=\delta_{ij}$ for each $x_j \in \cN_H$. This defines an interpolation operator for $v\in V^{\sfh} \cap C(\overline{\Omega})$:  \[I_H {v}:=\sum_{x_i \in \cN_H} {v}(x_i){\psi}_i(x)\, .\]
    Note that ${\psi}_i(x)$ is the same as the basis function constructed via the multiscale finite element method (MsFEM \cite{hou_multiscale_1997}). The interpolation residual $v-I_H v$ vanishes on each $x_i \in \cN_H$. Set\footnote{Note that we can apply $I_H$ to $\tilde{u}^\sfh$ due to the $C^{\alpha}$ estimate of $u$ in Proposition \ref{prop: C alpha estimate}.} $v = \tilde{u}^\sfh$ and let $I_H\tilde{u}^{\sfh}$ be one part of the approximation for $\tilde{u}^\sfh$.  Then, it remains to approximate the residue $\tilde{u}^{\sfh}-I_H\tilde{u}^{\sfh}$. 
    \item Localization: we wish to explore the fact that $\tilde{u}^{\sfh}-I_H\tilde{u}^{\sfh}$ vanishes on nodes to localize the subsequent approximation task. To achieve so, define  $R_e\tilde{u}^{\sfh}=P_e(\tilde{u}^{\sfh}-I_H\tilde{u}^{\sfh}):=(\tilde{u}^{\sfh}-I_H\tilde{u}^{\sfh})|_e$. The goal is to find some basis functions on each $e$ to approximate $R_e\tilde{u}^{\sfh}$. To make this problem  precise, we need to specify the function space of $R_e\tilde{u}^{\sfh}$, and the norm for approximation. 

    It turns out that the natural function space $R_e\tilde{u}^{\sfh}$ is the Lions-Magenes space; see the following Proposition \ref{prop: C alpha interpolation residue}. 
\end{enumerate}

\begin{proposition}
    \label{prop: C alpha interpolation residue}
    Let $d=2$. Suppose  $f \in L^2(\Omega)$ and \eqref{eqn: stability} holds. For each $e \in \cE_H$, it holds that $R_e\tilde{u}^{\sfh} \in H_{00}^{1/2}(e)$, the Lions-Magenes space which contains functions $v\in H^{1/2}(e)$ such that 
      \[\frac{v(x)}{\dist(x,\partial e)} \in L^2(e)\, . \]
      Here $\dist(x,\partial e)$ is the Euclidean distance from $x$ to the boundary of $e$.
    \end{proposition}
    It might seem unclear at this stage why we should consider such a complicated function space. In fact, this is related to the zero extension of functions. According to Chapter 33 of \cite{tartar2007introduction}, $H_{00}^{1/2}(e)$ can also be characterized as the space of functions in $H^{1/2}(e)$, such that their zero extensions to $E_H$ is still in $H^{1/2}(E_H)$. This is the key and in fact the only property that we will use for $H_{00}^{1/2}(e)$. The zero extension allows us to connect local approximation and global approximation. In the following we will not distinguish $\tilde{\psi}\in H_{00}^{1/2}(e)$ and its zero extension to $E_H$ that belongs to $H^{1/2}(E_H)$.
    
    For any function in $ H_{00}^{1/2}(e)$, we define a norm to measure approximation accuracy.
\begin{definition}
    \label{def: H 1/2 norm}
      Let $d=2$. The $\cH^{1/2}(e)$ norm of a function $\tilde{\psi}\in H_{00}^{1/2}(e)$ is defined as:
    \begin{equation}
    \label{eqn-def-H-edge-norm}
    \|\tilde{\psi}\|_{\cH^{1/2}(e)}^2:=\int_{\Omega}A|\nabla \psi |^2+k^2|V\psi|^2\, ,
    \end{equation}
    where we have used the one to one correspondence $\tilde{\psi} \in \tilde{V}^{\sfh} \leftrightarrow \psi \in V^{\sfh}$. \yc{Here we identify $\tilde{\psi}$ as the zero extension of its value on the edge $e$ to $E_H$}. 
    \end{definition}
The $\cH^{1/2}(e)$ norm in Definition \ref{def: H 1/2 norm} is the natural one to consider here since eventually, we aim for approximation accuracy in the energy norm. 

The following theorem is the cornerstone for the above localization strategy. It states that a local accuracy guarantee can be seamlessly coupled to form a global accuracy guarantee.
    \begin{theorem}[Global error estimate]
        \label{thm: Edge coupling error estimate}
        Let $d=2$. Suppose for each edge $e$, there exists an edge function $\tilde{v}_e \in H_{00}^{1/2}(e)$ that satisfies 
        \begin{equation}
        \label{eqn: local err estimate}
            \|R_e\tilde{u}^{\sfh}-\tilde{v}_e\|_{\cH^{1/2}(e)} \leq \epsilon_{e}\, .
        \end{equation}
      Let $v_e \in V^\sfh$ be the corresponding part of $\tilde{v}_e \in \tilde{V}^\sfh$. Then, it holds that
        \begin{equation}
        \|u^{\sfh}-I_Hu^{\sfh}-\sum_{e \in \cE_H} v_e\|^2_{\cH(\Omega)}\leq C_{\mathrm{mesh}}\sum_{e \in \cE_H} \epsilon_{e}^2\, ,    
        \end{equation}
        where $C_{\mathrm{mesh}}$ is a constant depending on the number of edges for the elements only, e.g., for quadrilateral mesh $C_{\mathrm{mesh}}=4$.
    \end{theorem}
    
    Given this theorem, to approximate $u^{\sfh}$ it suffices to find local edge basis functions that satisfy \eqref{eqn: local err estimate} for some desired $\epsilon_e$. This is a localized task for each $e$.
    
    The proofs for Propositions \ref{prop: C alpha interpolation residue} and Theorem \ref{thm: Edge coupling error estimate} are similar to that in the setting of elliptic equations \cite{chen2020exponential}. However, for completeness, we will also present them here in Subsections \ref{sec: Proof of Proposition prop: C alpha interpolation residue} and \ref{sec:Proof of Theorem thm: Edge coupling error estimate}.
\subsubsection{Local Approximation via Oversampling}
     \label{oversampling exp}
     The last step of approximation is to find local edge basis functions for each $e$ so that \eqref{eqn: local err estimate} is satisfied.
     In this subsection, we discuss how to achieve this via oversampling and SVD, which can yield exponentially decaying $\epsilon_e$.  The general idea is to explore the fact that for a coarse scale function, its behavior on $e$ can be controlled very well by that in an oversampling domain due to the compactness property of the restriction operator.
     
   More precisely, for a given edge $e$, consider an oversampling domain $\omega_e$ associated with the edge. In general, any domain containing $e$ in the interior can serve as a candidate. Here, for simplicity of presentation and as an illustrative example, we set 
\begin{equation}
    \label{eqn: os domain 1 layer}
        \omega_e=\overline{\bigcup \{T\in \cT_H: \overline{T} \cap e \neq \emptyset\}}\, .
    \end{equation} 
   For interior edges and edges connected to the boundary, an illustration of this choice \eqref{eqn: os domain 1 layer} for a quadrilateral mesh is given in Figure \ref{fig:os domain}. 
         \begin{figure}[ht]
        \centering
\tikzset{every picture/.style={line width=0.75pt}} 
\begin{tikzpicture}[x=0.75pt,y=0.75pt,yscale=-1,xscale=1]

\draw  [draw opacity=0][dash pattern={on 4.5pt off 4.5pt}] (60,64) -- (201.5,64) -- (201.5,163) -- (60,163) -- cycle ; \draw  [dash pattern={on 4.5pt off 4.5pt}] (70,64) -- (70,163)(109,64) -- (109,163)(148,64) -- (148,163)(187,64) -- (187,163) ; \draw  [dash pattern={on 4.5pt off 4.5pt}] (60,74) -- (201.5,74)(60,113) -- (201.5,113)(60,152) -- (201.5,152) ; \draw  [dash pattern={on 4.5pt off 4.5pt}]  ;
\draw    (109,113) -- (148,113) ;
\draw    (70,74) -- (187,74) ;
\draw    (187,74) -- (187,152) -- (70,152) -- (70,74) ;
\draw  [draw opacity=0][dash pattern={on 4.5pt off 4.5pt}] (274,71) -- (399.5,71) -- (399.5,162) -- (274,162) -- cycle ; \draw  [dash pattern={on 4.5pt off 4.5pt}] (274,71) -- (274,162)(313,71) -- (313,162)(352,71) -- (352,162)(391,71) -- (391,162) ; \draw  [dash pattern={on 4.5pt off 4.5pt}] (274,71) -- (399.5,71)(274,110) -- (399.5,110)(274,149) -- (399.5,149) ; \draw  [dash pattern={on 4.5pt off 4.5pt}]  ;
\draw    (313,71) -- (313,110) ;
\draw    (274,71) -- (274,149) -- (352,149) -- (352,71) -- cycle ;

\draw (124,102) node [anchor=north west][inner sep=0.75pt]   [align=left] {$\displaystyle e$};
\draw (166,76) node [anchor=north west][inner sep=0.75pt]   [align=left] {$\displaystyle \omega _{e}$};
\draw (88,171) node [anchor=north west][inner sep=0.75pt]   [align=left] {interior edge};
\draw (253,171) node [anchor=north west][inner sep=0.75pt]   [align=left] {edge connected to boundary};
\draw (314,89) node [anchor=north west][inner sep=0.75pt]   [align=left] {$\displaystyle e$};
\draw (276,130) node [anchor=north west][inner sep=0.75pt]   [align=left] {$\displaystyle \omega _{e}$};
\end{tikzpicture}
        \caption{Illustration of oversampling domains. On the right, we use an edge connected to the upper boundary as an illustrating example.}
        \label{fig:os domain}
    \end{figure}
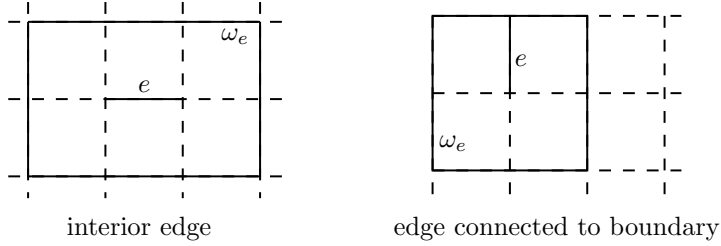

The key idea is to treat the residue $R_e\tilde{u}^{\sfh}$ as a restriction of a coarse scale function in $\omega_e$ and explore the compactness property of such restriction operators. {By an abuse of notation via the correspondence of $V^{\sfh}$ and $\tilde{V}^{\sfh}$, for any $\cH(\Omega)$ in $V$, we identify $R_e v$ as $R_e \tilde{v}^{\sfh}$}. As a first step, we write
   \begin{equation}
  \label{eqn: residue os decompose}
    R_e\tilde{u}^{\sfh}=R_e u=R_eu_{\omega_e}^{\sfh}+ R_eu_{\omega_e}^{\sfb}\, ,  
  \end{equation}
 where we decompose $u$ in $\omega_e$ into its coarse and fine scale components, via \eqref{eqn:Helmholtz decomposed} with $T$ replaced by $\omega_e$, and we shall use $u_{\omega_e}^\sfh$ and $u_{\omega_e}^\sfb$ to denote the corresponding local Helmholtz-harmonic and bubble part respectively.    
 Then, to approximate $R_e\tilde{u}^{\sfh}$, we could approximate the two terms in \eqref{eqn: residue os decompose} separately. We will show that the first term can be approximated in an exponentially efficient manner due to a compactness property, and the second term can be computed locally and is very small.
 
   \begin{remark}
   One may ask whether the decomposition \eqref{eqn: residue os decompose} in the oversampling domain is still well-defined. Indeed, similar to \eqref{Poincaree}, we have a uniform Poincar\'e inequality for every $\omega_e$: for any edge $e$ and $H^1(\omega_e)$ function $v$ vanishing on any one of the edge boundaries of $\omega_e$, it holds that
    \begin{equation}
\label{Poincaree1} 
\|v\|_{L^2(\omega_e)} \leq C'_{P} H \|\nabla v\|_{L^2(\omega_e)}\, ,
\end{equation}
where $C'_P$ is a constant that only depends on $c_0,c_1, d$ and our choice of oversampling domain. For the particular choice \eqref{eqn: os domain 1 layer}, $C'_P$ is a constant multiple of $C_P$; without loss of generality we assume $C'_P\geq C_P$. Based on this observation, we will choose a small $H$ so that Assumption \ref{small mesh1} holds, which guarantees that local Helmholtz operators in the oversampling domain behave in a manner similar to that of elliptic case; this is similar to Proposition \ref{prop-well-definedness-local-prob}.
\begin{assumption}
\label{small mesh1}
The mesh size satisfies $H\leq{A}_{\min}^{1/2}/(\sqrt{2}C'_P V_{\max} k)$, where $C'_P$ is the constant in \eqref{Poincaree1}.
 \end{assumption}
\end{remark}
Note that Assumption \ref{small mesh1} implies Assumption \ref{small mesh}.  Now, we discuss in detail how to deal with the two terms in \eqref{eqn: residue os decompose}.
  \begin{enumerate}
      \item For the first term, we consider the following function space in $\omega_e$:
    \begin{equation}
    \label{eqn-Helmholtz-harmonic-space}
    \begin{aligned}
        U(\omega_e):= \{v \in \cH(\omega_e): &-\nabla \cdot (A \nabla v)-k^2V^2v=0, \text{ in } \omega_e\\
        &A\nabla v\cdot\nu=T_kv, \text{ on } (\Gamma_N \cup \Gamma_R) \cap \partial \omega_e\} 
        \, .
    \end{aligned}
    \end{equation}
    {Functions in this space are fully determined by their trace on $\partial \omega_e\backslash (\Gamma_N \cup \Gamma_R)$.} By definition, $u_{\omega_e}^{\sfh}$ belongs to $U(\omega_e)$.  Under Assumption \ref{small mesh1}, $(U(\omega_e),\|\cdot\|_{\cH(\omega_e)})$ is a Hilbert space, since the Helmholtz operator in $\omega_e$ is elliptic. Then, by abuse of notation, consider the operator 
\[R_e: (U(\omega_e),\|\cdot\|_{\cH(\omega_e)}) \to (H_{00}^{1/2}(e),\|\cdot\|_{\cH^{1/2}(e)})\, ,\]
    such that $R_e v=P_e(v-I_Hv)$ for $v\in U(\omega_e)$. A critical property is that the singular values of $R_e$ decay nearly exponentially fast; see Theorem \ref{thm: svd exponential decay of Re}. Its proof is deferred to Subsection \ref{subsec: Proof of thm: svd exponential decay of Re}.
    \begin{theorem}
     \label{thm: svd exponential decay of Re}
     {Let $d$=2.} Under Assumption \ref{small mesh1}, the operator $R_{e}$ is compact for each $e \in \cE_H$. Denote the pairs of its left singular vectors and singular values by $\{\tilde{v}_{m,e},\lambda_{m,e} \}_{m \in \bN}$, where $\tilde{v}_{m,e} \in H_{00}^{1/2}(e)$ and the sequence $\{\lambda_{m,e}\}_{m \in \bN}$ is in a descending order. Then, for any $\epsilon>0$, it holds that
     \begin{equation}
     \label{eqn: upper bound of singular value}
         \lambda_{m,e}\leq C_{\epsilon}\exp\left(-m^{(\frac{1}{d+1}-\epsilon)}\right)\, ,
     \end{equation}
     where $C_{\epsilon}$ is a constant that is independent of $k, H$ and may depend on $\epsilon, d$ and the mesh parameters $c_0,c_1$.    \end{theorem}\yc{\begin{remark}
     As we can see from the proof, we actually show that \eqref{eqn: upper bound of singular value} still holds by setting $C_\epsilon$ to be $1$ and requiring for $m > N_\epsilon$ with $N_\epsilon$ depending on $k$ and $H$.  But we can also make the above inequality hold for all $m$ by introducing the constant $C_\epsilon$.     \end{remark}}
     We discuss the implication of this theorem. By definition of singular values, if we set $W_{m,e}=\mathrm{span}~\{\tilde{v}_{j,e}\}_{j=1}^{m-1}$, then Theorem \ref{thm: svd exponential decay of Re} implies that
    \begin{equation}
    \label{eqn-approximation-property-singular-values}
        \min_{\tilde{v}_{e} \in W_{m,e}} \|R_e v-\tilde{v}_{e}\|_{\cH^{1/2}(e)}\leq C_{\epsilon}\exp\left(-m^{(\frac{1}{d+1}-\epsilon)}\right)\|v\|_{\cH(\omega_e)}\, .
    \end{equation} Applying this result to $v=u^\sfh_{\omega_e} \in U(\omega_e)$ leads to
    \begin{equation}
    \label{eqn: os harmonic part approximation}
    \min_{\tilde{v}_e \in W_{m,e}} \|R_e u^\sfh_{\omega_e}-\tilde{v}_e\|_{\cH^{1/2}(e)}\leq C_{\epsilon}\exp\left(-m^{(\frac{1}{d+1}-\epsilon)}\right)\|u_{\omega_e}^{\sfh}\|_{\cH(\omega_e)}\, .    
    \end{equation}
    Thus, there is a nearly exponential efficiency in approximating the first term $R_e u^\sfh_{\omega_e}$.
    \item For the second term in \eqref{eqn: residue os decompose}, the oversampling bubble part $u^{\sfb}_{\omega_e}$ can be efficiently computed by solving local Helmholtz problems. Moreover, under Assumption \ref{small mesh1} this term is small in the $\cH(\Omega)$ norm as shown in the following proposition.
\begin{proposition}
\label{prop: small os bubble}
Under Assumption \ref{small mesh1}, for each $e \in \cE_H$ the following estimate holds for the oversampling bubble part:
    \label{prop: os bubble small}
    \[\|R_e u^\sfb_{\omega_e}\|_{\cH^{1/2}(e)}\leq CH\|f\|_{L^2(\omega_e)}\, ,\]
    where $C$ is a constant independent of $k$ and $H$.
    \end{proposition}
    The proof is deferred to Subsection \ref{subsec: Proof of Proposition prop: small os bubble}. 
    
       \yc{We further define a special Helmholtz-harmonic function $u^\sfs\in V^\sfh$, such that that its restriction on each edge $e\in E_H$ equals $R_eu_{\omega_e}^{\sfb}$. Namely this special Helmholtz-harmonic function accounts for the second term in \eqref{eqn: residue os decompose} for each edge. By the previous proposition, we immediately have the estimate:
       \begin{equation*}
         \|u^\sfs\|_{\cH(\Omega)}\leq CH\|f\|_{L^2(\Omega)}\, ,
       \end{equation*} where $C$ is a constant independent of $k$ and $H$. Along with Proposition \ref{prop: bubble is small}, we conclude that there is a constant $C_s$ independent of $k$ and $H$ such that 
       \begin{equation}
           \label{bbh}
         \|u^\sfs\|_{\cH(\Omega)}+\|u^\sfb\|_{\cH(\Omega)}\leq C_s H\|f\|_{L^2(\Omega)}\, .
       \end{equation}} 
    \end{enumerate}
   Now consider the following space of basis functions:
  \[\tilde{V}_{H,m,e}^{(1)}:=\yc{W_{m,e}} \, .\]
    In practice, this space can be computed locally by an SVD of $R_e$. Due to \eqref{eqn: residue os decompose} and \eqref{eqn: os harmonic part approximation}, we have the following error estimate on each $e$:
    \begin{equation}
    \label{eqn: harmonic part approximation on each e}
        \min_{\tilde{v}_e \in \tilde{V}_{H,m,e}^{(1)}} \|R_e u^\sfh-u^\sfs-\tilde{v}_e\|_{\cH^{1/2}(e)}\leq C_{\epsilon}\exp\left(-m^{(\frac{1}{d+1}-\epsilon)}\right)\|u_{\omega_e}^{\sfh}\|_{\cH(\omega_e)}\, .
    \end{equation}
    {\begin{remark}
    The operator $R_e$ involves nodal interpolation, which is in general not stable for $H^1$ functions if the dimension is greater than $1$. However, in Theorem \ref{thm: svd exponential decay of Re}, we take the domain of the operator to be $U(\omega_e)$, which contains Helmholtz-harmonic functions that are  H\"older continuous, due to the standard $C^{\alpha}$ estimates for elliptic equations. More specifically, Lemma \ref{lemma: bound H 1/2 by energy norm} implies the stability of $R_e$ in this space.
    \end{remark}}
    
    \begin{remark}
    \label{rmki}
    If we follow the proof of Lemma 3.13 in \cite{ma2021novel}, it is be possible to remove the small parameter $\epsilon$ in Theorem \ref{thm: svd exponential decay of Re} to get a better asymptotic bound $O(\exp(-m^{\frac{1}{d+1}}))$.
    \end{remark}
    \subsubsection{Low Complexity in Approximation}
    \label{subsec:Low Complexity in Approximation}
    Finally, define the collection of edge basis functions \[\tilde{V}_{H,m}^{(1)}=\text{span}~\{\bigcup_e \tilde{V}^{(1)}_{H,m,e}\}\, ,\] and denote by $\tilde{V}_H^{(0)}$ the span of the nodal interpolation basis used earlier, i.e. $\tilde{V}_H^{(0)}:=\text{span}~\{\tilde{\psi}_i\}$. Define the overall edge approximation $\tilde{V}_{H,m}=\text{span}~\{\tilde{V}_H^{(0)} \bigcup \tilde{V}_{H,m}^{(1)}\}$.
    Let $V_{H,m} \subset V^\sfh$ be the corresponding part of  $\tilde{V}_{H,m} \subset \tilde{V}^\sfh$, via \eqref{eqn: edge bulk correspondence}. Then, using \eqref{eqn: harmonic part approximation on each e} and Theorem \ref{thm: Edge coupling error estimate}, we get a nearly exponentially decaying error estimate for approximating $u^\sfh$; see Theorem \ref{thm: exponential error estimate}. 
    \begin{theorem}
    \label{thm: exponential error estimate} Let $d=2$. Under Assumption \ref{small mesh1} and \eqref{eqn: stability}, it holds that
    \[\min_{v \in V_{H,m}} \|u^\sfh - u^\sfs - v\|_{\cH(\Omega)} \leq C_d(C_{\mathrm{stab}}(k)+H)\exp\left(-m^{(\frac{1}{d+1}-\epsilon)}\right)\|f\|_{L^2(\Omega)}\, , \]
    where $C_d$ is a generic constant independent of $k, m, H$.
    \end{theorem}
\begin{proof}
    By Theorem \ref{thm: svd exponential decay of Re} and the global error estimate in Theorem \ref{thm: Edge coupling error estimate}, we get
     \begin{equation}
     \label{eqn-global-bounded-by-local}
         \min_{v \in V_{H,m}} \|u^\sfh -u^\sfs -  v\|^2_{\cH(\Omega)} \leq  C_{\mathrm{mesh}}C_{\epsilon}^2\exp\left(-2m^{(\frac{1}{d+1}-\epsilon)}\right)\sum_{e \in \cE_H} \|u_{\omega_e}^{\sfh}\|^2_{\cH(\omega_e)}\, .
     \end{equation}
     Due to Assumption \ref{small mesh1}, we have the elliptic estimate for the oversampling bubble part 
     \begin{equation}
     \label{eqn-small-os-bubble}
         \|u^{\sfb}_{\omega_e}\|_{\mathcal{H}(\omega_e)} \leq  \frac{3C'_P}{A_{\min}^{1/2}}H\|f\|_{L^{2}(\omega_e)}\, .
     \end{equation}
     This is similar to Proposition \ref{prop: bubble is small}, which is a consequence of Assumption \ref{small mesh}.
     Then,  using $u_{\omega_e}^{\sfh}=u-u^{\sfb}_{\omega_e}$, it follows that 
     \begin{equation}
     \label{eqn-local-bounded-by-rhs}
         \|u_{\omega_e}^{\sfh}\|^2_{\cH(\omega_e)}\leq 2(\|u\|^2_{\cH(\omega_e)}+\|u^{\sfb}_{\omega_e}\|^2_{\mathcal{H}(\omega_e)})\leq\frac{18C'^{2}_P}{A_{\min}}H^2\|f\|^2_{L^{2}(\omega_e)}+2\|u\|^2_{\cH(\omega_e)}\,. 
     \end{equation}
     Note that by our choice of oversampling domains, every element $T$ can only be covered by $\{\omega_e\}_{e \in \cE_H}$ at most $C_1$ times for a fixed $C_1$. Therefore it holds that
     \begin{equation}
     \label{eqn-sum-of-local-rhs-bounded}
        \sum_{e \in \cE_H} \|f\|^2_{L^{2}(\omega_e)}\leq C_1\|f\|^2_{L^{2}(\Omega)}\, , 
     \end{equation}
     as well as
     \begin{equation}
     \label{eqn-sum-of-u-norm-bounded}
        \sum_{e \in \cE_H} \|u\|^2_{\cH(\omega_e)}\leq C_1\|u\|^2_{\cH(\Omega)}\leq C_1C^2_{\mathrm{stab}}(k)\|f\|^2_{L^{2}(\Omega)}\, , 
     \end{equation}
      where the last inequality is due to the \textit{a priori} estimate \eqref{eqn: stability}. Combining \eqref{eqn-global-bounded-by-local}, \eqref{eqn-local-bounded-by-rhs}, \eqref{eqn-sum-of-local-rhs-bounded} and \eqref{eqn-sum-of-u-norm-bounded} completes the proof.
    \end{proof}
    
    Clearly, Theorem \ref{thm: exponential error estimate} implies the low complexity property of the part $u^\sfh-u^\sfs$. Each edge contains at most $m$ basis functions, so the space $V_{H,m}$ is of dimension $O(m/H^d)$, while the approximation accuracy is of order $\exp\left(-m^{(\frac{1}{d+1}-\epsilon)}\right)$. We will use the space $V_{H,m}$ in our multiscale framework for approximating $u^\sfh-u^\sfs$.

{\begin{remark}
\label{crucial}
\yc{$V_{H,m}$ does not depend on the right hand side $f$ or the solution $u$. Therefore, we can use the same $V_{H,m}$ for different right-hand sides.}
\end{remark}}
\section{The Multiscale Methods}
\label{methods}
In this section, we discuss the multiscale methods for solving \eqref{eqn:Helmholtz smooth k}, based on the coarse-fine scale decomposition established in the last section. 

\yc{By the nature of a multiscale algorithm, we will handle the ``coarse part'' $u^\sfh-u^\sfs$ and the ``fine part'' $u^\sfb+u^\sfs$ separately. Conceptually, the locality and small magnitude of $u^\sfb+u^\sfs$ imply that it can be computed efficiently or ignored without affecting the accuracy much, and the low complexity of $u^\sfh-u^\sfs$ indicates that we can use a Galerkin method with a small number of basis functions to approximate it accurately.}

In Subsection \ref{subsec: framework}, we outline our general multiscale computational framework. Depending on how the trial and test spaces in the Galerkin method are selected, we get two categories of algorithms, namely the Ritz-Galerkin approach and Petrov-Galerkin approach that we will make precise in Subsections \ref{subsec: Ritz-Galerkin Method} and \ref{subsec: Petrov-Galerkin Method}, respectively. 

\subsection{The Multiscale Framework}  
\label{subsec: framework}
\yc{The bubble part $u^\sfb$ and the special function $u^\sfs$ are first computed locally. Given these parts, we form an effective equation for $u^\sfh-u^\sfs$ as
\begin{equation}
\label{eqn: effective eqn for u h}
    a(u^\sfh-u^\sfs,v)=(f,v)_{\Omega}-a(u^\sfb+u^\sfs,v)\, ,
\end{equation}}
for any $v \in \cH(\Omega)$. 
\begin{remark}
The right hand side in \eqref{eqn: effective eqn for u h} can be seen as a bounded linear functional on $v \in \cH(\Omega)$. By the estimate in \eqref{eqn: wellposed H -1 rhs}, this equation for $u^\sfh-u^\sfs$ (given fixed $u^\sfb+u^\sfs$) is well-posed.
\end{remark}

Numerically, we solve the equation \eqref{eqn: effective eqn for u h} for $u^\sfh-u^\sfs$ using a Galerkin method. That is, we choose a trial space $S$ and a test space $S_{\mathrm{test}}$ to find a numerical solution $u_S \in S$ that satisfies
\begin{equation}
    a(u_S,v)=(f,v)_{\Omega}-a(u^\sfb+u^\sfs,v)\, ,
\end{equation}
for any $v \in S_{\mathrm{test}}$. If $S_{\mathrm{test}}=S$, then it is called a Ritz-Galerkin method, otherwise it is a Petrov-Galerkin method. Here since the equation is formulated in the complex domain, we specifically refer to the choice $S_{\mathrm{test}}=\overline{S}$ as the Petrov-Galerkin method. 

In Subsection \ref{subsec: Ritz-Galerkin Method}, we formulate our Ritz-Galerkin method and present theories for the well-posedness of the discrete problem, as well as the error estimate in both the energy norm and the $L^2$ norm. In Subsection \ref{subsec: Petrov-Galerkin Method}, we discuss the Petrov-Galerkin method, which is more straightforward and appears more convenient in practical computation. 
\subsection{The Ritz-Galerkin Method}
\label{subsec: Ritz-Galerkin Method}
First, we establish a general strategy for analyzing the Ritz-Galerkin method in solving \eqref{eqn: effective eqn for u h}. We start with a definition of the approximation accuracy of $S$.
\begin{definition}  For $S\subset V^\sfh$, the approximation accuracy of $S$ is defined as  
    \begin{equation}
    \label{apro proxy}
        \eta(S):=\sup _{f \in L^{2}(\Omega) \backslash\{0\}} \inf _{v \in S} \frac{\left\|u-v\right\|_{\mathcal{H}(\Omega)}}{\|f\|_{L^{2}(\Omega)}}\, ,
    \end{equation}
    where $u$ and $f$ are related via the Helmholtz equation in \eqref{eqn:Helmholtz smooth k}.
\end{definition}
For the Ritz-Galerkin method, it turns out that $\eta(S)$ is critical in analyzing the solution errors of $u_S$.
\begin{theorem}
\label{eqn:approximation property}
Suppose \eqref{eqn: stability} holds and $k \eta(S)  \leq 1/{(4C_cV_{\max})}$ as well as $\overline{S}=S$.
Then, the following statements hold for the Ritz-Galerkin method:
\begin{enumerate}
    \item The Galerkin solution $u_S$ is a quasi-optimal approximation in the sense that
 \begin{align*}
     &\|u^\sfh-u^\sfs-u_{S}\|_{\mathcal{H}(\Omega)} \leq 2 C_{c} \inf _{v \in S}\|u^\sfh-u^\sfs-v\|_{\mathcal{H}(\Omega)}\, ,\\
     &\|u^\sfh-u^\sfs-u_{S}\|_{L^{2}(\Omega)} \leq C_{c} \eta(S)\|u^\sfh-u^\sfs-u_{S}\|_{\mathcal{H}(\Omega)}\, .
 \end{align*}
     \item \yc{If we further assume $Hk\leq 1/(8C_sC_cV_{\max})$}, for constant $C_s$ defined in \eqref{bbh}, the discrete problem satisfies the discrete inf-sup stability condition:
 \begin{equation*}
\inf _{v \in S} \sup _{v' \in S \backslash\{0\}} \frac{|a(v, v')|}{\|v\|_{\mathcal{H}(\Omega)}\|v'\|_{\mathcal{H}(\Omega)}} \geq \frac{1}{4+3C_{c}^{-1}+8 kV_{\max} C_{\mathrm{stab}}(k)} \, .
\end{equation*}
\end{enumerate}
\end{theorem}
{The proof of this theorem is deferred to Subsection \ref{Gard}. It is inspired by the standard G\aa rding-type inequality for a posteriori estimate; see for example \cite{melenk2010convergence}. However, our proofs are slightly different since 
\yc{only the part $u^\sfh-u^\sfs$ is approximated via the basis functions.}}

The above theorem implies that once $\eta(S)$ is small, the discrete problem is well-posed, and the Galerkin solution approximates the exact solution accurately.

\yc{Given Theorem \ref{eqn:approximation property}, we can choose $S=V_{H,m}+\overline{V_{H,m}}$ where $V_{H,m}$ is defined in Theorem \ref{thm: exponential error estimate} independent of the right hand side. 
For the quantity $\eta(S)$, we have the following estimate using its subspace $V_{H,m}$:
\begin{equation}
\label{generic}
    \eta(S)\leq \eta(V_{H,m})\leq \max(C_d,C_s)\left((C_{\mathrm{stab}}(k)+H)\exp\left(-m^{(\frac{1}{d+1}-\epsilon)}\right)+H\right). 
\end{equation}
 Here we have used \eqref{bbh} for the small parts $u^\sfb$ and $u^\sfs$ of size $O(H)$, and Theorem \ref{thm: exponential error estimate} for the approximation error for $u^\sfh-u^\sfs$. 
}
{
Invoking Theorems \ref{eqn:approximation property} and  \ref{thm: exponential error estimate}, we get the following error analysis for the Galerkin solution:
\begin{theorem}
\label{thm: err estimate Ritz-Galerkin}
Let $d=2$. Suppose Assumption \ref{small mesh1} and \eqref{eqn: stability} hold, and \[ \max(C_d,C_s)k\left((C_{\mathrm{stab}}(k)+H)\exp\left(-m^{(\frac{1}{d+1}-\epsilon)}\right)+H\right) \leq 1/{(4C_cV_{\max})}\, ,\] where \yc{$C_s$, $C_d$ are generic constants defined in \eqref{bbh} and Theorem \ref{thm: exponential error estimate} respectively. }  Then using $S=V_{H,m}+\overline{V_{H,m}}$ in the Ritz-Galerkin method leads to a solution $u_S$ that satisfies:
\begin{equation}
\label{eqn: ritz-Galerkin, final error estimate}
     \|u^\sfh-u^\sfs-u_{S}\|_{\mathcal{H}(\Omega)} \leq 2 C_{c}C_d(C_{\mathrm{stab}}(k)+H)\exp\left(-m^{(\frac{1}{d+1}-\epsilon)}\right)\|f\|_{L^2(\Omega)}\, .
\end{equation}
\end{theorem}}

{For the $\epsilon$ that satisfies $\frac{1}{d+1}-\epsilon=\frac{1}{d+2}$, we can take $m\sim O(\log^{d+2}(kC_{\mathrm{stab}}(k))$. Then the condition in Theorem \ref{thm: err estimate Ritz-Galerkin} holds, provided that the mesh size $H$ satisfies the following Assumption \ref{small mesh2}:
 \begin{assumption}
\label{small mesh2}
The mesh size satisfies $H\leq 1/(8\max(C_d,C_s)C_c V_{\max} k)$.
 \end{assumption}} Furthermore, if $C_{\mathrm{stab}}(k)\leq C(1+k^\gamma)$ for some constants $\gamma$ and $C$, then the condition $m\sim O(\log^{d+2}(kC_{\mathrm{stab}}(k))$ reduces to $m \sim \log^{d+2}(k)$. This implies that once $m$ is moderately large, i.e., logarithmic in $k$, the nearly exponential convergence of the Galerkin solution shown in Theorem \ref{thm: err estimate Ritz-Galerkin} will become effective. As in Remark \ref{rmki}, we can improve the index $d+2$ to $d+1$.

We provide several additional remarks of the Ritz-Galerkin method below.
\begin{remark}
In the Ritz-Galerkin method, the trial and test spaces are $S=V_{H,m}+\overline{V_{H,m}}$. One can intuitively understand that $V_{H,m}$ is needed to represent the desired solution, and $\overline{V_{H,m}}$ is used for the approximation of the adjoint problem, which is required in the numerical analysis of the Helmholtz equation. There can be a lot of overlap between $V_{H,m}$ and $\overline{V_{H,m}}$: on each interior edge, since the singular vectors of $R_e$ are real, these edge basis functions are real-valued. Thus, $V_{H,m}$ and $\overline{V_{H,m}}$ can only differ on the edges connected to the boundary, where the presence of the Robin boundary condition makes the operator non-Hermitian. 
\end{remark}

\begin{remark}
\yc{Combining \eqref{eqn: ritz-Galerkin, final error estimate} with the local computation of the fine parts will yield the overall error estimate for $u$, which is nearly exponentially convergent. }
\end{remark}


\subsection{The Petrov-Galerkin Method} \label{subsec: Petrov-Galerkin Method}
In this subsection, we introduce the Petrov-Galerkin method. We choose $S=V_{H,m}$ and $S_{\text{test}}=\overline{V_{H,m}}$. {We give the following remarks on this method}.
\begin{remark}
The trial and test spaces in the Petrov-Galerkin method often have smaller dimensions than their Ritz-Galerkin counterpart, since we do not put the complex conjugate $\overline{V_{H,m}}$ in $S$. This can save computational efforts.
\end{remark}

\begin{remark}
Our current theory does not address the stability of the discrete system and the $\cH(\Omega)$ error estimate for the Petrov-Galerkin method. This is left for our future work. We note that our numerical experiments in the next section imply that these properties also hold for the Petrov-Galerkin method.
\end{remark}
\section{Numerical Experiments} 
\label{sec-numeric-experiments}
In this section, we will outline and discuss our numerical algorithms in detail based on the established theoretical analysis. Several Helmholtz equations are solved using our algorithm, which confirm our theoretical results. We also consider some examples in which our theoretical assumptions are not satisfied. Even for these examples, our methods still give a nearly exponential rate of convergence. This provides further evidence for the robustness of our methods.

\subsection{Set-up}

We consider the domain $\Omega=[0,1]\times [0,1]$ and discretize it by a uniform two-level quadrilateral mesh; see a fraction of this mesh in Figure \ref{fig:mesh1}, where we also show an edge $e$ and its oversampling domain $\omega_e$ in solid lines. 
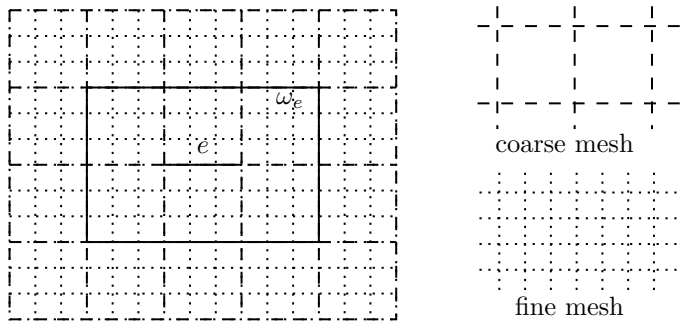
\begin{figure}[htbp]
\centering

\tikzset{every picture/.style={line width=0.75pt}} 

\begin{tikzpicture}[x=0.75pt,y=0.75pt,yscale=-1,xscale=1]

\draw  [draw opacity=0][dash pattern={on 4.5pt off 4.5pt}] (170.5,75) -- (365.5,75) -- (365.5,231) -- (170.5,231) -- cycle ; \draw  [color={rgb, 255:red, 0; green, 0; blue, 0 }  ,draw opacity=1 ][dash pattern={on 4.5pt off 4.5pt}] (209.5,75) -- (209.5,231)(248.5,75) -- (248.5,231)(287.5,75) -- (287.5,231)(326.5,75) -- (326.5,231) ; \draw  [color={rgb, 255:red, 0; green, 0; blue, 0 }  ,draw opacity=1 ][dash pattern={on 4.5pt off 4.5pt}] (170.5,114) -- (365.5,114)(170.5,153) -- (365.5,153)(170.5,192) -- (365.5,192) ; \draw  [color={rgb, 255:red, 0; green, 0; blue, 0 }  ,draw opacity=1 ][dash pattern={on 4.5pt off 4.5pt}] (170.5,75) -- (365.5,75) -- (365.5,231) -- (170.5,231) -- cycle ;
\draw  [draw opacity=0][dash pattern={on 0.84pt off 2.51pt}] (170.5,75) -- (365.5,75) -- (365.5,231) -- (170.5,231) -- cycle ; \draw  [color={rgb, 255:red, 0; green, 0; blue, 0 }  ,draw opacity=1 ][dash pattern={on 0.84pt off 2.51pt}] (183.5,75) -- (183.5,231)(196.5,75) -- (196.5,231)(209.5,75) -- (209.5,231)(222.5,75) -- (222.5,231)(235.5,75) -- (235.5,231)(248.5,75) -- (248.5,231)(261.5,75) -- (261.5,231)(274.5,75) -- (274.5,231)(287.5,75) -- (287.5,231)(300.5,75) -- (300.5,231)(313.5,75) -- (313.5,231)(326.5,75) -- (326.5,231)(339.5,75) -- (339.5,231)(352.5,75) -- (352.5,231) ; \draw  [color={rgb, 255:red, 0; green, 0; blue, 0 }  ,draw opacity=1 ][dash pattern={on 0.84pt off 2.51pt}] (170.5,88) -- (365.5,88)(170.5,101) -- (365.5,101)(170.5,114) -- (365.5,114)(170.5,127) -- (365.5,127)(170.5,140) -- (365.5,140)(170.5,153) -- (365.5,153)(170.5,166) -- (365.5,166)(170.5,179) -- (365.5,179)(170.5,192) -- (365.5,192)(170.5,205) -- (365.5,205)(170.5,218) -- (365.5,218) ; \draw  [color={rgb, 255:red, 0; green, 0; blue, 0 }  ,draw opacity=1 ][dash pattern={on 0.84pt off 2.51pt}] (170.5,75) -- (365.5,75) -- (365.5,231) -- (170.5,231) -- cycle ;
\draw    (209.5,114) -- (209.5,192) -- (326.5,192) -- (326.5,114) -- cycle ;

\draw [color={rgb, 255:red, 0; green, 0; blue, 0 }  ,draw opacity=1 ][line width=0.75]    (248.5,153) -- (287.5,153) ;
\draw  [draw opacity=0][dash pattern={on 0.84pt off 2.51pt}] (407.5,157) -- (507.5,157) -- (507.5,218) -- (407.5,218) -- cycle ; \draw  [color={rgb, 255:red, 0; green, 0; blue, 0 }  ,draw opacity=1 ][dash pattern={on 0.84pt off 2.51pt}] (417.5,157) -- (417.5,218)(430.5,157) -- (430.5,218)(443.5,157) -- (443.5,218)(456.5,157) -- (456.5,218)(469.5,157) -- (469.5,218)(482.5,157) -- (482.5,218)(495.5,157) -- (495.5,218) ; \draw  [color={rgb, 255:red, 0; green, 0; blue, 0 }  ,draw opacity=1 ][dash pattern={on 0.84pt off 2.51pt}] (407.5,167) -- (507.5,167)(407.5,180) -- (507.5,180)(407.5,193) -- (507.5,193)(407.5,206) -- (507.5,206) ; \draw  [color={rgb, 255:red, 0; green, 0; blue, 0 }  ,draw opacity=1 ][dash pattern={on 0.84pt off 2.51pt}]  ;
\draw  [draw opacity=0][dash pattern={on 4.5pt off 4.5pt}] (406.5,73) -- (508.5,73) -- (508.5,135) -- (406.5,135) -- cycle ; \draw  [color={rgb, 255:red, 0; green, 0; blue, 0 }  ,draw opacity=1 ][dash pattern={on 4.5pt off 4.5pt}] (416.5,73) -- (416.5,135)(455.5,73) -- (455.5,135)(494.5,73) -- (494.5,135) ; \draw  [color={rgb, 255:red, 0; green, 0; blue, 0 }  ,draw opacity=1 ][dash pattern={on 4.5pt off 4.5pt}] (406.5,83) -- (508.5,83)(406.5,122) -- (508.5,122) ; \draw  [color={rgb, 255:red, 0; green, 0; blue, 0 }  ,draw opacity=1 ][dash pattern={on 4.5pt off 4.5pt}]  ;

\draw (263.5,140) node [anchor=north west][inner sep=0.75pt]   [align=left] {$\displaystyle e$};
\draw (303.5,115) node [anchor=north west][inner sep=0.75pt]   [align=left] {$\displaystyle \omega _{e}$};
\draw (414,136) node [anchor=north west][inner sep=0.75pt]   [align=left] {coarse mesh};
\draw (424,218) node [anchor=north west][inner sep=0.75pt]   [align=left] {fine mesh};

\end{tikzpicture}

\caption{Two level mesh: a fraction}
\label{fig:mesh1}
\end{figure}
    The coarse and fine mesh sizes are denoted by $H$ and $h$, respectively. 
    
    For a given Helmholtz equation,  we compute the reference solution $u_{\text{ref}}$ using the classical FEM on the fine mesh; with a sufficiently small $h$, it is reasonable to treat $u_{\text{ref}}$ as the ground truth $u$. {We remark that via a posteriori estimates, we can check that the fine mesh indeed resolve the corresponding problems; thus the associated fine mesh solutions could serve as good reference solutions, for all of our numerical examples. To be precise, we check that the relative error between the solutions using fine mesh of size $h$ and $h/2$ are small, such that it is of order $10^{-2}$ in energy norm and $10^{-4}$ in $L^2$ norm.}
    
    The accuracy of a numerical solution $u_{\mathrm{sol}}$ is computed by comparing it with the reference solution $u_{\text{ref}}$ on the fine mesh. The accuracy will be measured both in the $L^2$ norm and energy norm: 
    \begin{equation}
\label{rel_error}
\begin{aligned}
e_{L^{2}}&=\frac{\|u_{\text{ref}}-u_{\mathrm{sol}}\|_{L^{2}(\Omega)}}{\|u_{\text{ref}}\|_{L^{2}(\Omega)}}\, ,\\
e_{\cH}&=\frac{\|u_{\text{ref}}-u_{\mathrm{sol}}\|_{\cH(\Omega)}}{\|u_{\text{ref}}\|_{\cH(\Omega)}}\, . 
\end{aligned}
\end{equation}
\subsection{Multiscale Algorithms}
We outline our numerical algorithms for obtaining $u_{\mathrm{sol}}$. There are offline and online stages, depending on whether the steps involve the information of the right hand side.
\subsubsection{Offline Stage}
For each edge $e \in \cE_H$ and its associated oversampling domain $\omega_e$, the key step in the offline stage is to construct the discretized version of the operator
\[R_e: (U(\omega_e),\|\cdot\|_{\cH(\omega_e)}) \to (H_{00}^{1/2}(e),\|\cdot\|_{\cH^{1/2}(e)})\, ,\]
which is defined by $R_e v = (v-I_H v)|_e$. Here $U(\omega_e)$ is defined in \eqref{eqn-Helmholtz-harmonic-space}, $\|\cdot\|_{\cH(\omega_e)}$ is the energy norm in $\omega_e$,  while $H_{00}^{1/2}(e)$ is the Lions-Magenes space, and $\|\cdot\|_{\cH^{1/2}(e)}$ is defined in \eqref{eqn-def-H-edge-norm}. 

We note that functions in $U(\omega_e)$ are fully determined by their traces on $\partial \omega_e \backslash (\Gamma_N\cup \Gamma_R)$. Thus, we can take the discretized matrix version of $R_e$ as a linear mapping from Dirichlet's data on  $\partial \omega_e \backslash (\Gamma_N\cup \Gamma_R)$ to the image of $R_e$, which contains functions on the edge $e$. The discretization of the $\|\cdot\|_{\cH(\omega_e)}$ and $\|\cdot\|_{\cH^{1/2}(e)}$ norms leads to positive definite matrices on the discretized domains $\partial \omega_e \backslash (\Gamma_N\cup \Gamma_R)$ and $e$. To obtain these positive definite matrices, we construct the Helmholtz-harmonic extension operators both on $e$ and $\partial \omega_e \backslash (\Gamma_N\cup \Gamma_R)$, which maps boundary data to the Hemholtz-harmonic function in the domain. Based on this operator, we can calculate the energy norms of the extended Hemholtz-harmonic function. This leads to the required norms as well as the positive definite matrices defining these norms\footnote{See also the implementation in Subsection 4.2 of \cite{chen2020exponential} on how these matrices are constructed for elliptic problems.}.

With the discretized matrices constructed, the next step is to compute the top $m$ left singular vectors of $R_e$ for some selected $m \in \bN$. This SVD problem turns out to be a generalized eigenvalue problem for these discrete matrices. For each $e$, denote the singular vectors by $\tilde{v}_{1,e}, ...,\tilde{v}_{m,e} \in H^{1/2}_{00}(e)$. Their Helmholtz-harmonic extensions to the domain are denoted by $v_{1,e},...,v_{m,e} \in \cH(\Omega)$, obtained via the correspondence \eqref{eqn: edge bulk correspondence}. The basis function space formed by the collection of all $v_{j,e}, 1\leq j \leq m$ and $e \in \cE_H$, together with the interpolation part $\{\psi_i\}_{x_i \in \cN_H}$, are denoted by \yw{$V_{H,m}$ and will constitute the Galerkin basis as defined in Subsection \ref{subsec:Low Complexity in Approximation}}. Note that here $\{\psi_i\}_{x_i \in \cN_H}$ are the same as the basis functions in the MsFEM.

\yw{We are now in a position to construct our Galerkin basis and the associated stiffness matrix. The construction depends on how to choose the trial and test spaces in the Galerkin method. We will outline two possible choices below:
\begin{itemize}
    \item  Ritz-Galerkin: $S = V_{H,m} + \overline{V_{H,m}}$ and $S_{\mathrm{test}}=S$.
    \item  Petrov-Galerkin: $S = V_{H,m}$ and $S_{\mathrm{test}}=\overline{V_{H,m}}$.
\end{itemize}}

\subsubsection{Online Stage}
\label{sec: online stage}
In the online stage, we solve \yw{the coarse and fine scales separately. Firstly we solve for $u^\sfb$ and $u^\sfs$, and then we use the effective equation \eqref{eqn: effective eqn for u h} to solve for $u^\sfh-u^\sfs$.

For the bubble part $u^\sfb$, we solve the local Hemholtz problem in each element $T\in\cT_H$, which leads to $u^{\sfb}_T$ defined in \eqref{eqn:Helmholtz decomposed}. Gluing them together leads to $u^\sfb$.

For $u^\sfs$, on each $e \in \cE_H$ and $\omega_e$, we construct the oversampling bubble part $u^{\sfb}_{\omega_e}$ via solving a local Helmholtz equation. Then, we get an edge function $R_e u^{\sfb}_{\omega_e}$ for each edge. We solve locally the Helmholtz-harmonic extension of these edge functions and add them together to obtain $u^\sfs$.

Now we can form the right-hand side vector in our effective equation \eqref{eqn: effective eqn for u h}, and use the offline-assembled stiffness matrix to obtain the Galerkin solution for the part $u^\sfh-u^\sfs$.}

This construction yields a practical numerical algorithm that efficiently handles multiple right-hand sides.

We note that all the above algorithms consider a uniform number of basis functions, namely $m$, for each edge $e \in \cE_H$. It is also possible to make this number vary with edges, thus fully adaptive to the problem's local properties such as the approach in \cite{hou2015optimal}. Consequently, this will lead to an adaptive algorithm where the truncated singular values serve as local error indicators. We do not pursue this in detail here and will leave this to our future work.

In the following, we will test our algorithms for different model problems. Our general set-up is to fix a reasonable coarse scale $H$ and then study how the errors behave as $m$ changes, for the two choices outlined above.

\begin{remark}
Our numerical experience implies that in the Ritz-Galerkin method, one does not need to add the conjugate space $\overline{V_{H,m}}$  into $S$ while still obtaining an exponential rate of convergence. 
\end{remark}

\subsection{A High Wavenumber Example: Planar Wave} 
\label{subsec: High Wavelength Example}
We start with an example of planar wave where the coefficients are constant and the wavenumber is high. More precisely, we set $A=V=\beta=1$ and $f=0$. The wavenumber $k=2^7$. We take the exact solution to be\[u(x_1,x_2)= \exp(-ik(0.6x_1+0.8x_2))\, .\] Using this solution, we are able to specify the Robin boundary condition on $\partial \Omega$. Note that this is an inhomogeneous boundary condition, so it is beyond our previous discussion. In this case, the inhomogeneous data are incorporated to the equation of the bubble part $u^\sfb$, while the treatment for the Helmholtz-harmonic part remains the same as that in the homogeneous case. \yc{To be specific, now our decomposition on each element $T$ is $u=u_T^\sfh+u_T^\sfb+u_T^\sfp$ where $u_T^\sfp$ stands for a particular solution. The part $u_T^\sfb+u_T^\sfp$ satisfies
$$\begin{aligned}
    -\nabla \cdot (A \nabla (u_T^\sfb+u_T^\sfp) )-k^{2}V^2 (u_T^\sfb+u_T^\sfp)&=f,\  \text{in} \  T\\
   u_T^\sfb+u_T^\sfp&=0, \ \text{on} \  \partial T\setminus (\Gamma_N \cup \Gamma_R)\\
   A\nabla (u_T^\sfb+u_T^\sfp)\cdot\nu&=T_k (u_T^\sfb+u_T^\sfp)+g, \ \text{on} \ \partial T \cap (\Gamma_N \cup \Gamma_R) \, .
    \end{aligned}$$
    We will use $u^\sfb+u^\sfp$ to replace $u^\sfb$ on the right-hand side of the effective equation for Galerkin solution \eqref{eqn: effective eqn for u h}. Similarly, when we compute the special Helmholtz-harmonic function $u^\sfs$ to account for the oversampling bubble part, its restriction on each edge equals $R_e(u_{\omega_e}^\sfb+u_{\omega_e}^\sfp)$ instead of $R_e u_{\omega_e}^\sfb$. In this way we can take care of the boundary data via local particular problems and still obtain the desired accuracy. The error analysis in such case remains the same once we replace $u_T^\sfb$ in the homogeneous data case by $u_T^\sfb+u_T^\sfp$;  in the bound we will also have the norm of $g$.}

We set the fine mesh $h=2^{-10}$, coarse mesh $H=2^{-5}$. We vary the number of edge basis functions in each $e \in \cE_H$, choosing $m=1,2,...,7$ and implementing the two algorithms outlined in Subsection \ref{sec: online stage}. The results are shown in Figure \ref{fig:eg1}.
\begin{figure}[ht]
    \centering
    \includegraphics[width=6cm]{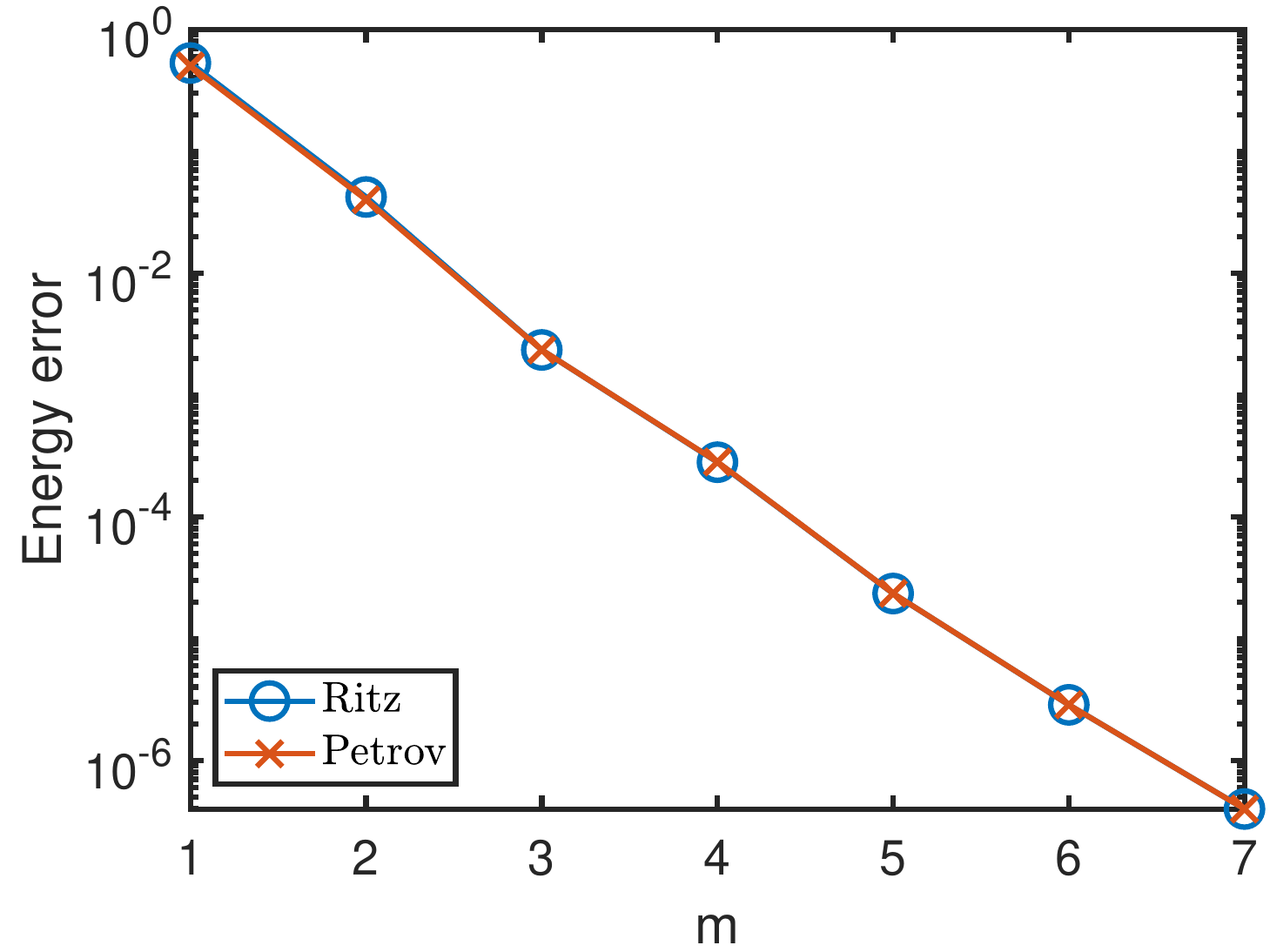}
    \includegraphics[width=6cm]{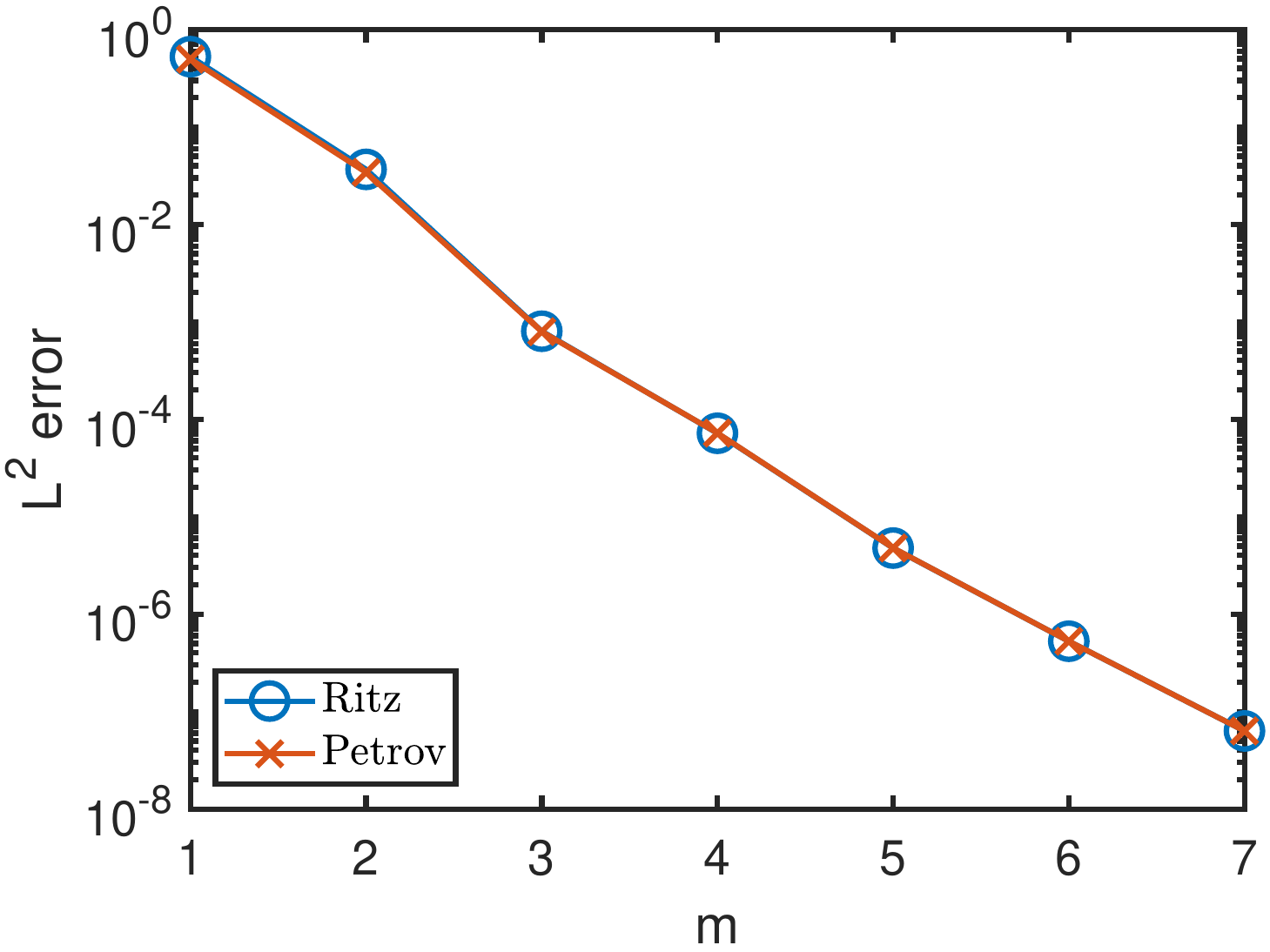}
    \caption{Numerical results for the high wavenumber example. Left: $e_{\cH}$ versus $m$; right: $e_{L^2}$ versus $m$.}
    \label{fig:eg1}
\end{figure}
We observe that the online basis approaches achieve nearly exponential decaying errors with respect to $m$. The difference between the Ritz-Galerkin and Petrov-Galerkin approaches is almost negligible.  We can see that a few basis per edge suffice for very high accuracy.

\yc{
Furthermore, we make some comparison between our edge coupling approach (the Ritz-Galerkin version) and the PUM approach reported in \cite{ma2021wavenumber}. We adopt the same setting there with $k=100$, $H=1/20$, $h=1/1000$ and vary the number of edge basis functions in each $e \in \cE_H$, choosing $m=2,3,...,7$. We present the results in Figure \ref{fig:eg11}.
\begin{figure}[ht]
    \centering
    \includegraphics[width=6cm]{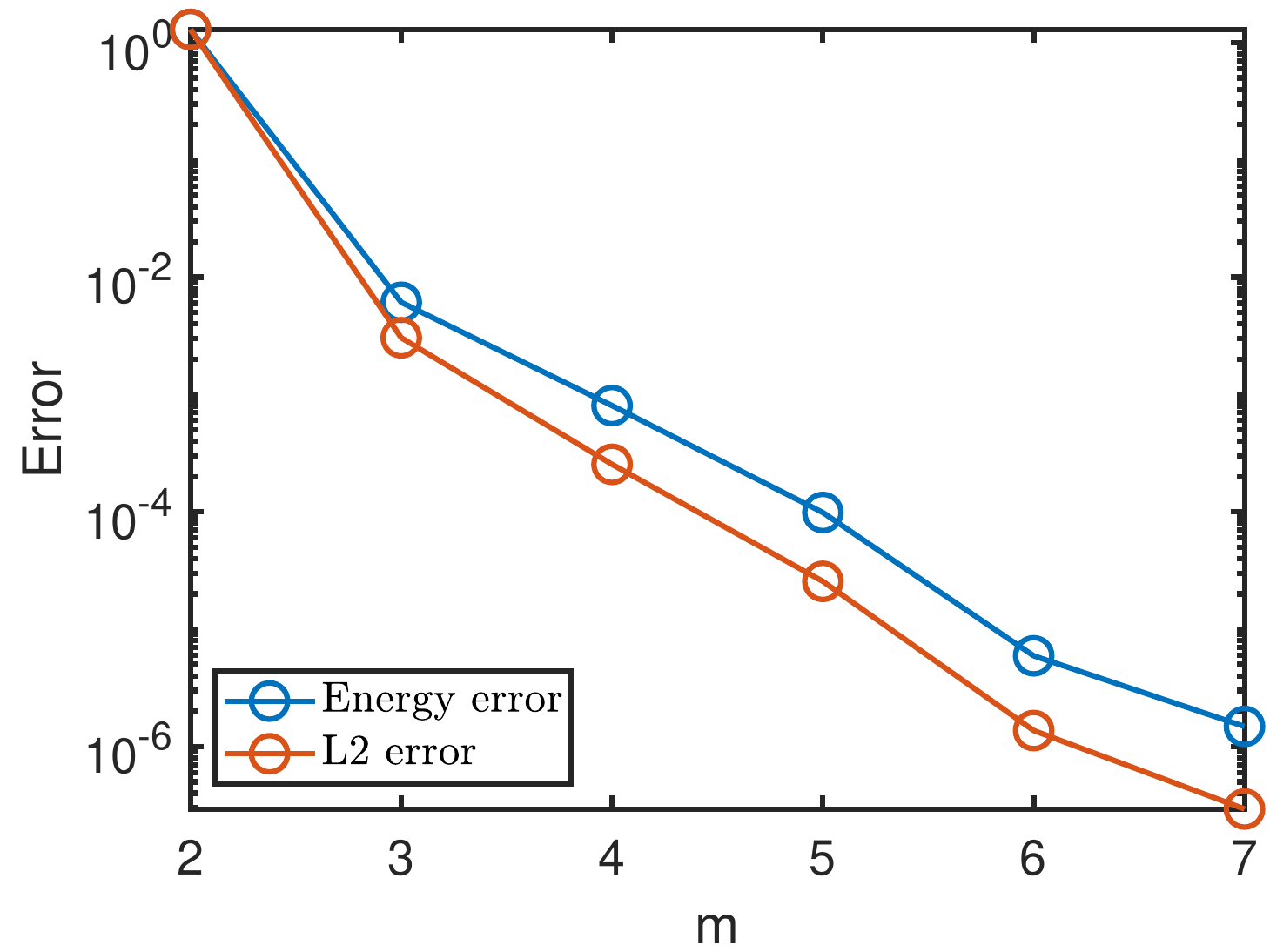}
    \caption{Numerical results for the high wavenumber example with $k=100$, $H=1/20$, $h=1/1000$.}
    \label{fig:eg11}
\end{figure}
We see that both errors decay very fast, and in particular,  the error in our method for $m=7$ is smaller than the error in \cite{ma2021wavenumber} with oversampling ratio $H_*/H=2$ and $35$ local basis per patch. With the same wavenumber and number of coarse patches, our method uses a slightly larger oversampling domain, while reducing the number of multiscale basis by a factor of around $35/(2\times7) = 2.5$. Here, we have used the fact that the number of edges is twice as many as domains in 2D. Nevertheless, the support of basis functions in our approach and PUM approach could be different by a factor of $2$, and the size of the overlapped domain decomposition in the PUM approach could also influence the result, leading to additional complexities for comparison. More detailed numerical study of the two approaches could be of future interests.
}

\subsection{A High Contrast Example: Mie resonances}
In this example, we consider an $A(x)$ with high contrast channels. More precisely, define the domain \begin{equation}
 \Omega_{\varepsilon}=(0.25,0.75)^{2} \cap \bigcup_{j \in \mathbb{Z}^{2}} \varepsilon\left(j+(0.25,0.75)^{2}\right)\, ,
\end{equation}
and the coefficient is defined as
\begin{equation*}
    A(x)=\left\{ 
    \begin{aligned} 1, \quad &x\notin \Omega_{\varepsilon}\\
    \varepsilon^2, \quad &x\in \Omega_{\varepsilon}\, .
    \end{aligned}
    \right.
\end{equation*}
Here, $\varepsilon$ is a parameter controlling the contrast. We choose $\varepsilon=2^{-4}$ and visualize $\log_{10} A(x)$ in the left plot of Figure \ref{fig:contour_A}.
\begin{figure}[ht]
    \centering
    \includegraphics[width=6cm]{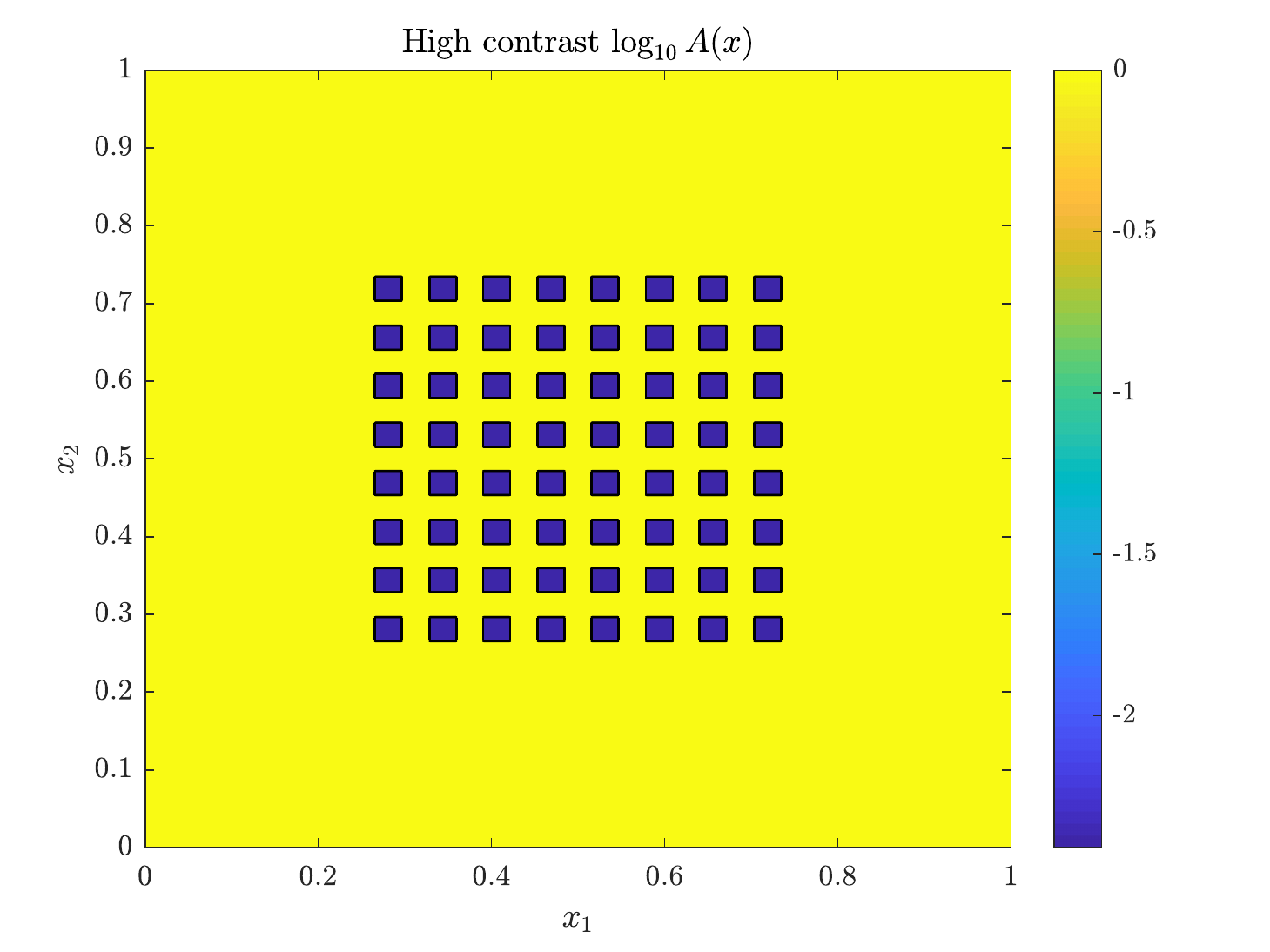}
    \includegraphics[width=6cm]{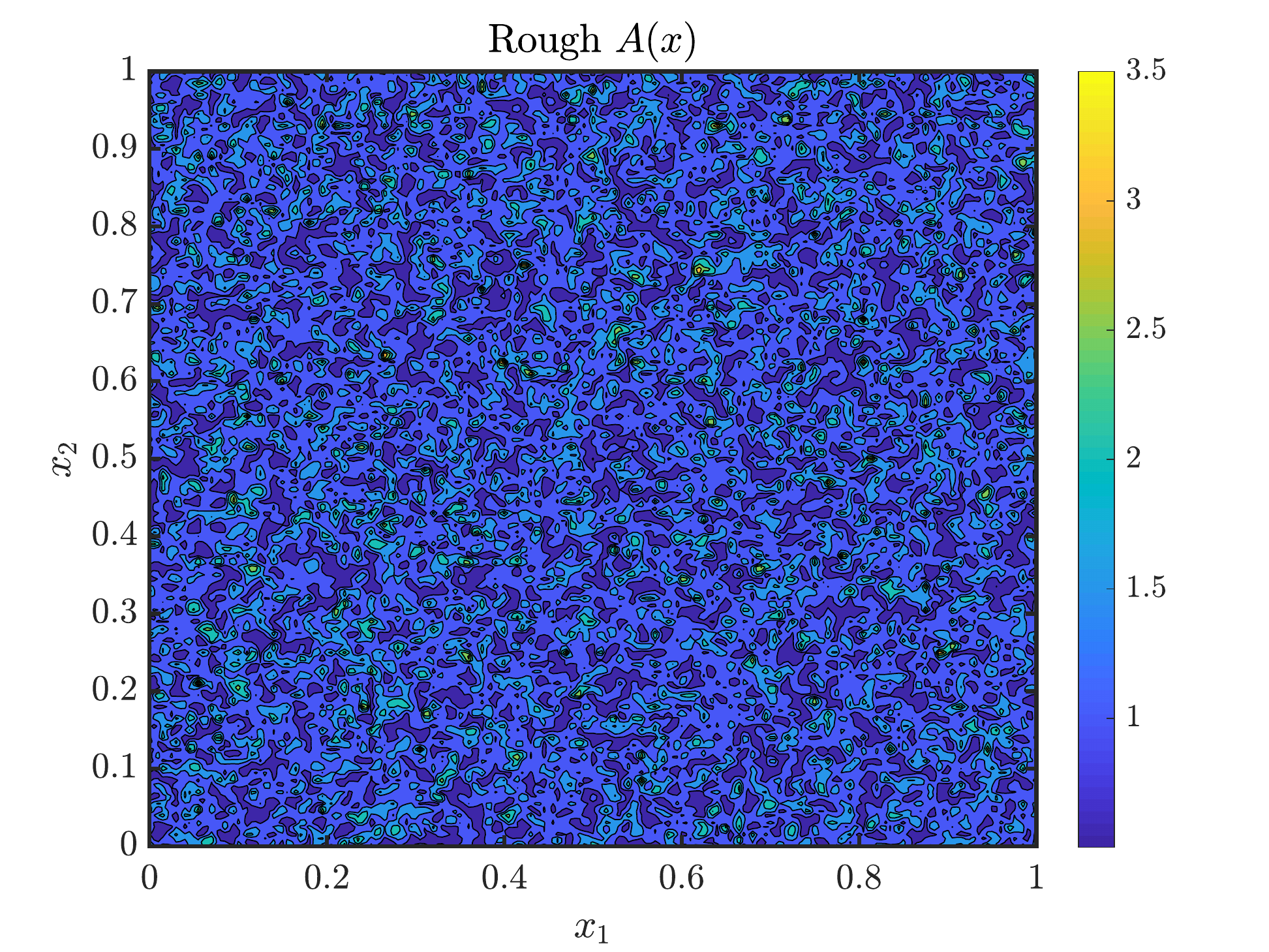}
    \caption{Left: the contour of $\log_{10} A$ for the high contrast example; right: the contour of $A$ for the rough media example.}
    \label{fig:contour_A}
\end{figure}

{We take $\beta=1, V=1, k=9$. For such a choice of $k$, the model exhibits an unusual behavior induced by Mie resonances in the small inclusions; see \cite{ohlberger2018new,peterseim2020computational}. An accurate numerical solution for this model would be hard to compute and it serves as a proper benchmark for our method.}  The right hand side is
\begin{equation*}
    f(x_1,x_2)=\left\{
    \begin{aligned}10000\exp(-\frac{1}{1-400\times \mathrm{dist}(x,z)^2}), &\ \mathrm{dist}(x,z)^2<\frac{1}{400}\\
    0, &\ \text{otherwise}\, , \end{aligned}\right.
\end{equation*}
where $z=(0.125,0.5)$ and $\mathrm{dist}(x,z)^2 = (x_1-0.125)^2+(x_2-0.5)^2$. 
We impose the homogeneous Robin boundary condition on $\partial \Omega$.
We take the fine mesh $h=2^{-9}$ and the coarse mesh $H=2^{-5}$.  As before we take $m=1,2,...,7$ and the numerical results are shown in Figure \ref{fig:eg3}. {A nearly exponential rate of convergence is observed consistently, and in this particular example, the Ritz method \yw{slightly outperforms} the Petrov method.}
\begin{figure}[ht]
    \centering
    \includegraphics[width=6cm]{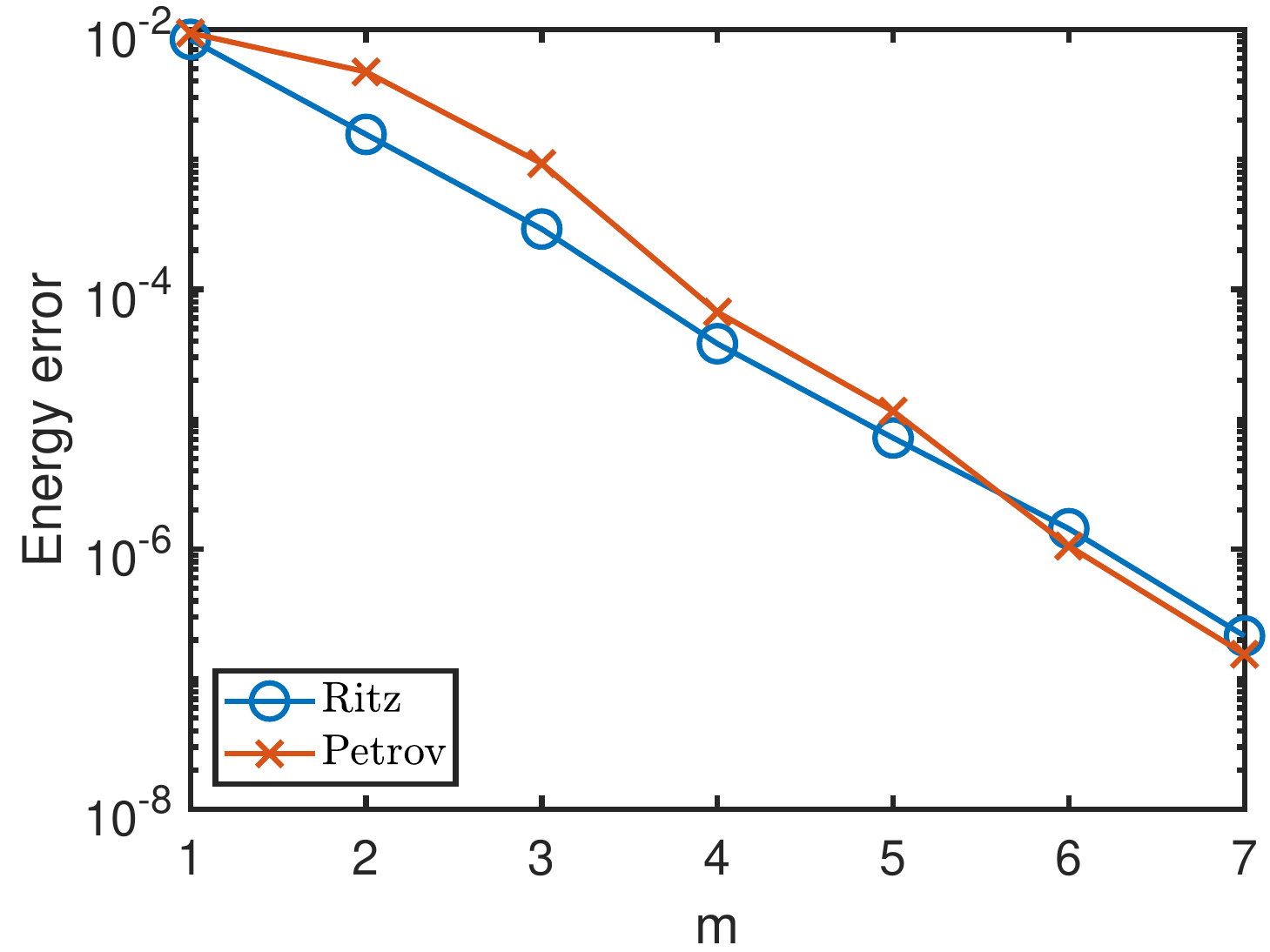}
    \includegraphics[width=6cm]{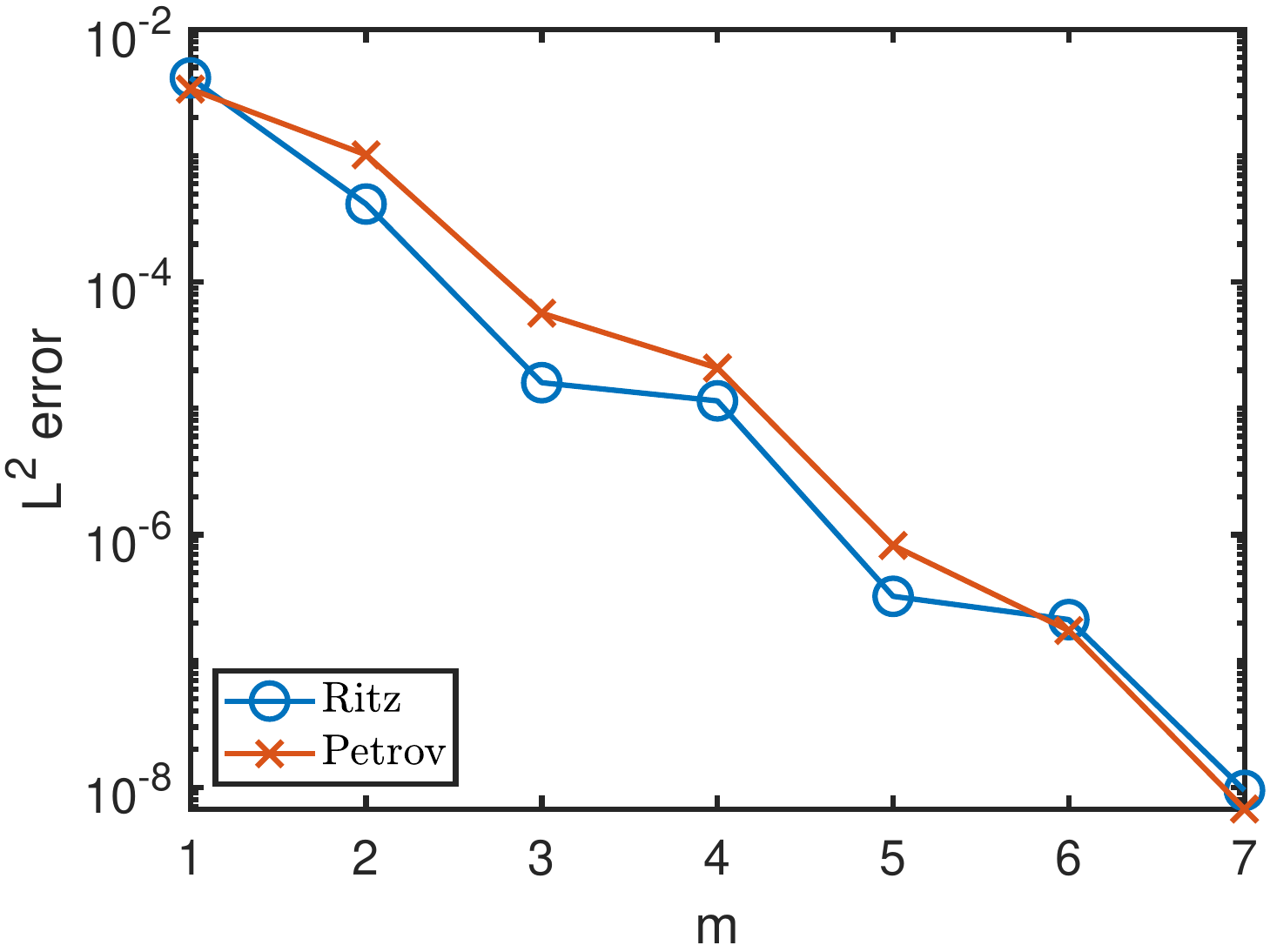}
    \caption{Numerical results for the high contrast example. Left: $e_{\cH}$ versus $m$; right: $e_{L^2}$ versus $m$.}
    \label{fig:eg3}
\end{figure}
\subsection{An Numerical Example with Mixed Boundary and Rough Field}
In the last example, we consider a mixed boundary problem. We impose the
homogeneous Dirichlet boundary condition on $(x_1,0), x_1 \in [0,1]$, the homogeneous Neumann boundary condition on $(x_1,1), x_1 \in [0,1]$, and the homogeneous Robin boundary condition on the other two parts of $\partial \Omega$. 
We choose $A(x)$ to be a realization of some random field; more precisely,
\begin{equation}
A(x)=|\xi(x)|+0.5\, ,
\end{equation}
        where the field $\xi(x)$ satisfies \[\xi(x)=a_{11}\xi_{i,j}+a_{21}\xi_{i+1,j}+a_{12}\xi_{i,j+1}+a_{22}\xi_{i+1,j+1}, \ \text{if}\ x \in [\frac{i}{2^{7}},\frac{i+1}{2^{7}})\times [\frac{j}{2^{7}},\frac{j+1}{2^{7}})\, .\]
        Here, $\{\xi_{i,j}, 0\leq i,j \leq 2^7 \}$ are i.i.d. unit Gaussian random variables. In addition, $a_{11}=(i+1-2^7x_1)(j+1-2^7x_2)$, $a_{21}=(2^7x_1-i)(j+1-2^7x_2)$, $a_{12}=(i+1-2^7x_1)(2^7x_2-j)$, $a_{22}=(2^7x_1-i)(2^7x_2-j)$ are interpolating coefficients to make $\xi(x)$ piecewise linear. A sample from this field is displayed in the right plot of Figure \ref{fig:contour_A}.

Moreover, we also take $V(x)$ and $\beta(x)$ as independent samples drawn from this random field. We choose the wavenumber $k=2^5$, the right hand side 
{$f(x_1,x_2)=x_1^4-x_2^3+1$}, the fine mesh $h=2^{-10}$ and the coarse mesh $H=2^{-5}$. Again we take $m=1,2,...,7$ and present the numerical results in Figure \ref{fig:eg4}. 
\begin{figure}[ht]
    \centering
    \includegraphics[width=6cm]{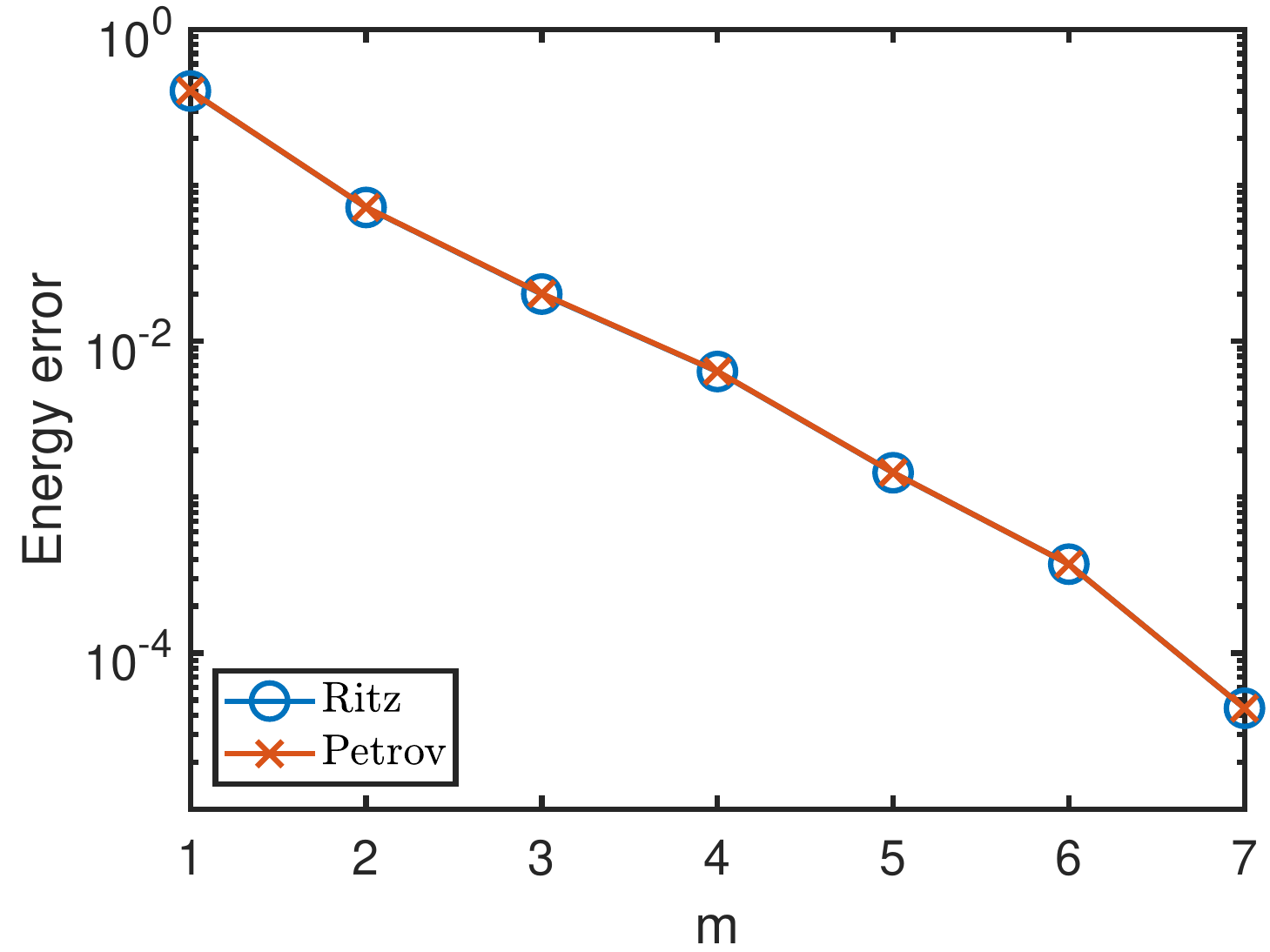}
    \includegraphics[width=6cm]{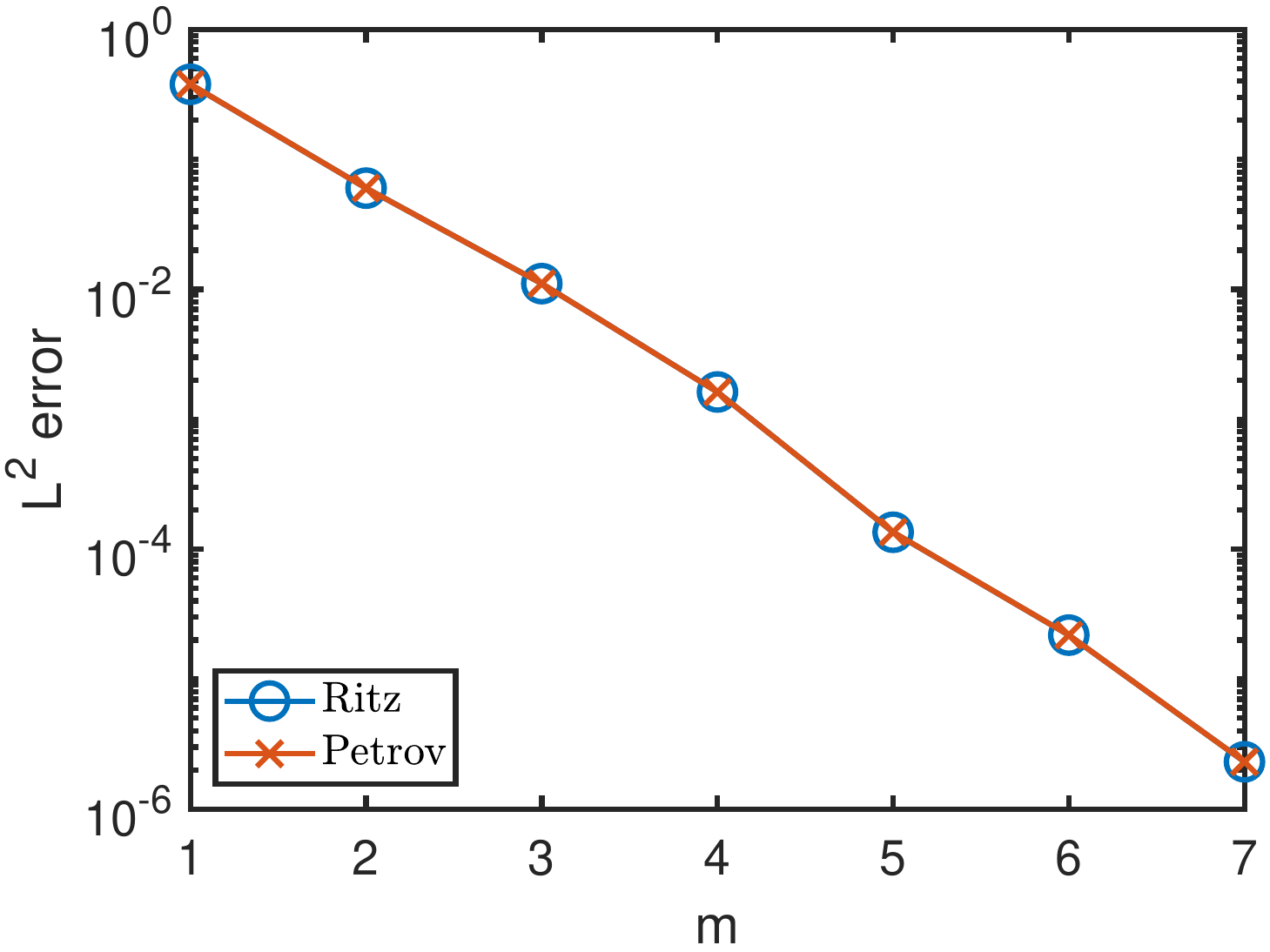}
    \caption{Numerical results for the mixed boundary and rough field example. Left: $e_{\cH}$ versus $m$; right: $e_{L^2}$ versus $m$.}
    \label{fig:eg4}
\end{figure}
A nearly exponential rate of convergence is still observed for this challenging example. The differences between the Ritz-Galerkin method and Petrov-Galerkin method is very mild. 

{It is worth noting that this example is constructed artificially, mixing different kinds of boundary conditions and rough coefficients, without taking into account the analytical properties of this combination. Thus, the numerical results for this example demonstrate the effectiveness of our multiscale methods in a more general setting. Moreover, our right hand side $f$ is global, so most oversampling bubble parts would be non-zero.}
\subsection{Summary}
We summarize what we have observed in these numerical examples. \yw{Both algorithms lead to a nearly exponential rate of convergence with respect to $m$, and we are able to use the offline-computed Galerkin basis to solve for multiple right-hand sides.} 

{Moreover, it is observed that the difference between the Ritz-Galerkin and the Petrov-Galerkin approaches is very mild in most cases, but sometimes Ritz-Galerkin method can have better performances.
Therefore, we recommend using the Ritz-Galerkin approach in practice.}
\section{Proofs}
\label{theory}
This section presents the theoretical proofs in this paper. Some proofs are similar to those in the elliptic case. We will  refer these proofs to the corresponding proofs in the elliptic case \cite{chen2020exponential}, while we will make relevant remarks on possible changes and modifications.
\subsection{Proof of Proposition \ref{prop: C alpha estimate}}
\label{subsec: Proof of Proposition prop: C alpha estimate}
In this subsection, we provide the proof of the qualitative version of $C^{\alpha}$ estimate. It is a direct application of related results for elliptic equations. 
\begin{proof}
We note that the Helmholtz PDE \eqref{eqn:Helmholtz smooth k} is equivalent to 
\begin{equation}
\label{reformulated eqn}
\begin{cases}
-\nabla \cdot(A\nabla u)=f+k^{2}V^2 u, \ \text{in} \ \Omega\\
u=0, \ \text{on} \ \Gamma_D\\
A\nabla u\cdot\nu=T_ku, \ \text{on} \  \Gamma_N \cup \Gamma_R \, .
\end{cases}
\end{equation}
Since $f\in L^2(\Omega)$, we know by the a priori estimate of the Helmholtz equation that $u\in H^1(\Omega)$. Therefore  we can regard \eqref{reformulated eqn} as an elliptic PDE with $k^{2}V^2 u$ known as a part of the right hand side. This PDE has its right hand side in $L^2(\Omega)$ and has $u$ as its solution. We can invoke the result in Remark 6.5 of \cite{griepentrog2001linear}, which concludes that $u$ lies in some H\"older space $C^{\alpha}(\Omega)$ such that
\begin{equation*}
    \|u\|_{C^{\alpha}(\Omega)}\leq C(\|f\|_{L^2(\Omega)}+k^2\|u\|_{L^2(\Omega)})\, ,
\end{equation*}
for some H\"older exponent $\alpha \in (0,1)$ and $C$.
\end{proof}

\subsection{Proof of Proposition \ref{prop: C alpha interpolation residue}}
\label{sec: Proof of Proposition prop: C alpha interpolation residue}
The proof relies on the fact that any function $v$ on $e$ belonging to $H^{1/2}(e)\bigcap C^{\alpha}(e)$ and vanishing at $\partial e$ will be in the space $H_{00}^{1/2}(e)$; see Proposition 2.1 in \cite{chen2020exponential} for detailed arguments of this fact. Then, $R_e\tilde{u}^{\sfh} \in H^{1/2}(e)\bigcap C^{\alpha}(e)$ and vanishes at $\partial e$, so it belongs to $H_{00}^{1/2}(e)$.

\subsection{Proof of Theorem \ref{thm: Edge coupling error estimate}}
\label{sec:Proof of Theorem thm: Edge coupling error estimate}
We decompose the energy norm into the contribution from each element $T \in \cT_H$:
    \[\|u^{\sfh}-I_Hu^{\sfh}-\sum_{e \in \cE_H} v_e\|^2_{\cH(\Omega)}=\sum_{T\in\cT_H}\|u^{\sfh}-I_Hu^{\sfh}-\sum_{e \sim T} v_e\|^2_{\cH(T)}\, , \]
    where we have used the fact that $v_e = 0$ in $T$ if $e$ and $T$ are not neighbors.
    
   Let us fix an element $T$. For each $e \sim T$, the trace of the function $u^{\sfh}-I_Hu^{\sfh}-\sum_{e \in T} v_e$ on $e$ is $\tilde{u}^{\sfh}-I_H\tilde{u}^{\sfh}-\tilde{v}_e\in H^{1/2}_{00}(e)$. We can extend this trace to $\partial T\backslash e$ by $0$ to get an $H^{1/2}(\partial T)$ boundary data. Then, this boundary data can be used to define a Helmholtz-harmonic function in $T$, via the correspondence \eqref{eqn: edge bulk correspondence}. Using the triangle inequality and the Cauchy-Schwarz inequality, we get
     \[ \|u^{\sfh}-I_Hu^{\sfh}-\sum_{e \sim T} v_e\|^2_{\cH(T)}\leq C_{\text{mesh}}\sum_{e \sim T} \|P_e(\tilde{u}^{\sfh}-I_H\tilde{u}^{\sfh})-\tilde{v}_e\|^2_{\cH_T^{1/2}(e)}\, ,\]
     where the $\cH_T^{1/2}(e)$ norm of a function $\tilde{\psi} \in H_{00}^{1/2}(e)$ is defined as 
    \begin{equation}
    \|\tilde{\psi}\|_{\cH_T^{1/2}(e)}^2:=\int_{T}A|\nabla \psi |^2+k^2|V\psi|^2\, .
    \end{equation}
    The constant $C_{\text{mesh}}$ depends on the mesh type only; for example $C_{\text{mesh}}=4$ for the quadrilateral mesh and $C_{\text{mesh}}=3$ for the triangular mesh.
    Then, we sum the above inequality over all $T \in \cT_H$, which yields
    \begin{equation} \label{eq1}
\begin{split}
 \|u^{\sfh}-I_Hu^{\sfh}-\sum_{e \in \cE_H} v_e\|^2_{\cH(\Omega)}&\leq C_{\mathrm{mesh}}\sum_{T\in\cT_H}\sum_{e \sim T} \|P_e(\tilde{u}^{\sfh}-I_H\tilde{u}^{\sfh})-\tilde{v}_e\|^2_{\cH_T^{1/2}(e)} \\
 & =C_{\mathrm{mesh}}\sum_{e \in \cE_H}\|P_e(\tilde{u}^{\sfh}-I_H\tilde{u}^{\sfh})-\tilde{v}_e\|^2_{\cH^{1/2}(e)} \\
 & \leq C_{\mathrm{mesh}}\sum_{e \in \cE_H} \epsilon_{e}^2\, .
\end{split}
\end{equation}
The proof is completed.

\subsection{Proof of Theorem \ref{thm: svd exponential decay of Re}} 
\label{subsec: Proof of thm: svd exponential decay of Re}
This is the key theorem underlying the exponential convergence for approximating $u^\sfh$.
To prove it, we need to analyze the spectrum of the operator $R_e$ for each edge $e$. The treatments for interior edges and edges connected to the boundary are slightly different, due to the different boundary conditions involved. We will explain the proof for interior edges in detail and comment on the changes needed to be made for edges connected to the boundary.

Since this theorem is stated for all edges, we start by discussing some geometric relations that hold uniformly for all interior edges. 
\subsubsection{Geometric Relation}
Suppose $e$ is an interior edge, so that $e$ lies strictly in the interior domain of $\omega_e$; see Figure \ref{fig:os domain}. We describe some geometric relation\footnote{It is similar to that in Subsection 3.3.1 of \cite{chen2020exponential}.} between $e$ and $\omega_e$ that will be needed in our analysis.
Figure \ref{fig:omega omega star} illustrates our ideas for a uniform quadrilateral mesh. For each interior edge $e$, there exists two concentric rectangles $\omega \subset \omega^*$ with center being the midpoint $m_e$ of $e$, such that $e \subset \omega \subset \omega^* \subset \omega_e$; the center $m_e$ is the center of gravity of $\omega$ and $\omega^*$. We require $\omega^* \cap \partial \Omega = \emptyset$. 
    Moreover, one side of $\omega$ and $\omega^*$ should be parallel to $e$. We introduce three parameters $l_1,l_2,l_3$ to specify and describe the geometry:
    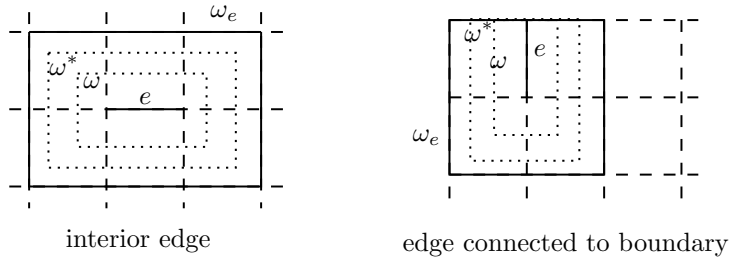
\begin{figure}[ht]
        \centering
\tikzset{every picture/.style={line width=0.75pt}} 

\begin{tikzpicture}[x=0.75pt,y=0.75pt,yscale=-1,xscale=1]

\draw  [draw opacity=0][dash pattern={on 4.5pt off 4.5pt}] (60,64) -- (201.5,64) -- (201.5,163) -- (60,163) -- cycle ; \draw  [dash pattern={on 4.5pt off 4.5pt}] (70,64) -- (70,163)(109,64) -- (109,163)(148,64) -- (148,163)(187,64) -- (187,163) ; \draw  [dash pattern={on 4.5pt off 4.5pt}] (60,74) -- (201.5,74)(60,113) -- (201.5,113)(60,152) -- (201.5,152) ; \draw  [dash pattern={on 4.5pt off 4.5pt}]  ;
\draw    (109,113) -- (148,113) ;
\draw    (70,74) -- (187,74) ;
\draw    (187,74) -- (187,152) -- (70,152) -- (70,74) ;
\draw  [draw opacity=0][dash pattern={on 4.5pt off 4.5pt}] (282,68) -- (407.5,68) -- (407.5,159) -- (282,159) -- cycle ; \draw  [dash pattern={on 4.5pt off 4.5pt}] (282,68) -- (282,159)(321,68) -- (321,159)(360,68) -- (360,159)(399,68) -- (399,159) ; \draw  [dash pattern={on 4.5pt off 4.5pt}] (282,68) -- (407.5,68)(282,107) -- (407.5,107)(282,146) -- (407.5,146) ; \draw  [dash pattern={on 4.5pt off 4.5pt}]  ;
\draw    (321,68) -- (321,107) ;
\draw    (282,68) -- (282,146) -- (360,146) -- (360,68) -- cycle ;
\draw  [dash pattern={on 0.84pt off 2.51pt}] (94.5,95) -- (159.5,95) -- (159.5,132) -- (94.5,132) -- cycle ;
\draw  [dash pattern={on 0.84pt off 2.51pt}] (79.75,84.5) -- (174.25,84.5) -- (174.25,142.5) -- (79.75,142.5) -- cycle ;
\draw  [dash pattern={on 0.84pt off 2.51pt}] (304.5,68) -- (336.5,68) -- (336.5,126) -- (304.5,126) -- cycle ;
\draw  [dash pattern={on 0.84pt off 2.51pt}] (292.5,67.5) -- (347.5,67.5) -- (347.5,139) -- (292.5,139) -- cycle ;

\draw (124,103) node [anchor=north west][inner sep=0.75pt]   [align=left] {$\displaystyle e$};
\draw (160,59) node [anchor=north west][inner sep=0.75pt]   [align=left] {$\displaystyle \omega _{e}$};
\draw (87,172) node [anchor=north west][inner sep=0.75pt]   [align=left] {interior edge};
\draw (257,174) node [anchor=north west][inner sep=0.75pt]   [align=left] {edge connected to boundary};
\draw (323,79) node [anchor=north west][inner sep=0.75pt]   [align=left] {$\displaystyle e$};
\draw (262,122) node [anchor=north west][inner sep=0.75pt]   [align=left] {$\displaystyle \omega _{e}$};
\draw (96,95) node [anchor=north west][inner sep=0.75pt]   [align=left] {$\displaystyle \omega $};
\draw (78.75,84.5) node [anchor=north west][inner sep=0.75pt]   [align=left] {$\displaystyle \omega ^{*}$ };
\draw (288,67) node [anchor=north west][inner sep=0.75pt]   [align=left] {$\displaystyle \omega ^{*}$ };
\draw (301,85) node [anchor=north west][inner sep=0.75pt]   [align=left] {$\displaystyle \omega $};

\end{tikzpicture}
        \caption{Geometric relation $e\subset\omega\subset \omega^*\subset \omega_e$}
        \label{fig:omega omega star}
    \end{figure}
    \begin{enumerate}
        \item With respect to the center $m_e$, the two rectangles $\omega$ and $\omega^*$ are scaling equivalent, such that there exists $l_1>1$, $\omega^*-m_e=l_1\cdot (\omega-m_e)$. Here we use the notation that $t\cdot X :=\{tx:x\in X\}$ for a set $X$ and a scalar $t$. For our choice of $\omega_e$, the parameter $l_1$ can be selected to only depend on $c_0$ and $c_1$ in Subsection \ref{subsec: mesh structure of elements}. 
        \item The ratio of $\omega$'s larger side length over the smaller side length is bounded by a uniform constant $l_2>1$ that depends on $c_0$ and $c_1$ only.
        \item There is a constant $l_3 > 1$ depending on $c_0$ and $c_1$ only such that $l_3\cdot e \subset \omega$. 
    \end{enumerate}
      We note that $l_1,l_2,l_3$ are universal constants for all interior edges. All three parameters depend on $c_0,c_1$ only. We introduce these parameters in order to get a uniform treatment for every interior edge. Indeed, several constants in our estimates depend on $l_1,l_2,l_3$, but not on $k$ and $H$, uniformly for all interior edges.
      
    \subsubsection{Main Idea of the Proof}
    \label{sec-main-idea-of-proof}
   In the following, we explain the main ideas of our proof. Recall the target is to show the left singular values of $R_e$ decays nearly exponentially fast. Similar to the rationale behind \eqref{eqn-approximation-property-singular-values}, it suffices to show there exists an $m-1$ dimensional space $W_{m,e} \subset H_{00}^{1/2}(e)$ such that
   \begin{equation}
   \label{eqn-to-show}
        \min_{\tilde{v}_{e} \in W_{m,e}} \|R_e v-\tilde{v}_{e}\|_{\cH^{1/2}(e)}\leq C_{\epsilon}\exp\left(-m^{(\frac{1}{d+1}-\epsilon)}\right)\|v\|_{\cH(\omega_e)}\, ,
    \end{equation}
    for any $v \in U(\omega_e)$. We also use $U(\omega')$ to denote the function space in $\omega'$ defined via \eqref{eqn-Helmholtz-harmonic-space} with $\omega_e$ replaced by any $\omega'$. Our proof contains two main steps, summarized in the following two lemmas.
\begin{lemma}
\label{lemma-compactness-restriction-of-harmonic}
 For $d>0$ and any $v \in U(\omega^*)$, there exists an $m-1$ dimensional space $\Phi_{m,e} \subset U(\omega)$ such that 
\begin{equation}
      \min_{\chi \in \Phi_{m,e}}\| v-\chi\|_{\cH(\omega)}\leq C_{\epsilon}\exp\left(-m^{(\frac{1}{d+1}-\epsilon)}\right) \|v\|_{\cH(\omega^*)}\, ,
    \end{equation}
    for some $C_{\epsilon}$ independent of $k$ and $H$.
\end{lemma}
\begin{lemma}
\label{lemma: bound H 1/2 by energy norm}
 For $d=2$ and any $v \in H^1(\omega)$ and $\nabla \cdot (A\nabla v) \in L^2(\omega)$, it holds that
\begin{equation}
\left\|R_e v\right\|_{\cH^{1/2}(e)}\leq C 
\left(\|v\|_{\cH(\omega)} + H\|\nabla \cdot (A\nabla v)+k^2V^2v\|_{L^2(\omega)}\right) \, ,
\end{equation}
for some $C$ independent of $k$ and $H$.
\end{lemma}
\yc{
\begin{remark}
Here in Lemma \ref{lemma: bound H 1/2 by energy norm}, the constant is independent of $k$ because we are using a local version of $C^{\alpha}$ regularity estimate, where in the local domain the operator behaves similarly to an elliptic operator; see the discussions preceding Assumption \ref{small mesh1}. This is different from the global $C^{\alpha}$ regularity estimate in the proof of Proposition \ref{prop: C alpha estimate}. See Subsection \ref{subsec: Proof of Lemma bound H 1/2 by energy norm} for details, where we only use $C^{\alpha}$ estimate of an elliptic equation.
\end{remark}}
We will defer the proofs of the two lemmas to Subsections \ref{sec-proof-lemma-compactness-restriction-of-harmonic} and \ref{subsec: Proof of Lemma bound H 1/2 by energy norm}, and describe how to prove Theorem \ref{thm: svd exponential decay of Re} using them here.
\begin{proof}[Proof of Theorem \ref{thm: svd exponential decay of Re}]
From the above discussion, it remains to show \eqref{eqn-to-show}. For $v \in \cH(\omega_e)$, we have $v \in \cH(\omega^*)$ and $\|v\|_{\cH(\omega^*)} \leq \|v\|_{\cH(\omega_e)}$. By Lemma \ref{lemma-compactness-restriction-of-harmonic}, we get
\[\min_{\chi \in \Phi_{m,e}}\| v-\chi\|_{\cH(\omega)}\leq C_{\epsilon}\exp\left(-m^{(\frac{1}{d+1}-\epsilon)}\right) \|v\|_{\cH(\omega_e)}\, . \]
Now since $v - \chi$ satisfies the condition in Lemma \ref{lemma: bound H 1/2 by energy norm} and $v$ and  $\chi$ both vanish under the operator $v \to \nabla \cdot (A\nabla v)+k^2V^2v$, we obtain
\[\min_{\chi \in \Phi_{m,e}}\| R_e v-R_e\chi\|_{\cH^{1/2}(e)}\leq CC_{\epsilon}\exp\left(-m^{(\frac{1}{d+1}-\epsilon)}\right) \|v\|_{\cH(\omega_e)}\, . \]
Thus, taking $W_{m,e}=R_e\Phi_{m,e}$ completes the proof.
\end{proof}
\subsubsection{Proof of Lemma \ref{lemma-compactness-restriction-of-harmonic}}
\label{sec-proof-lemma-compactness-restriction-of-harmonic}
The proof of this lemma is inspired by Theorem 3.3 in \cite{babuska2011optimal}, which states a similar result but for elliptic equations only. We generalize it here for the Helmholtz equation. 

First, by our geometric construction, $\omega^*-m_e=l_1\cdot (\omega-m_e)$. We denote a sequence of domains $\omega=\omega_0\subset \omega_1 \subset \cdots \subset \omega_{N-1} \subset \omega_N = \omega^*$ such that they are concentric and that $\omega_{j}-m_e=(1+t)(\omega_{j-1}-m_e)$ for $j=1,2,\cdots,N$. Here $t=l_1^{1/N}-1$. Then, there are two important lemmas, whose proofs are presented in Subsections \ref{sec-proof-lemma-L2-by-H} and \ref{sec-proof-lemma-H-by-L2}.
\begin{lemma}
\label{lemma-L2-by-H}
For each $0\leq j \leq N$ and any $n \in \mathbb{N}$, there is an $n$-dimensional space $W_{n}(\omega_j)\subset U(\omega_j)$, such that for all $v \in U(\omega_j)$, it holds that
  \begin{equation}
  \label{eqn-L2-by-H}
\inf _{w \in W_{n}\left(\omega_j\right)}\|v-w\|_{L^{2}\left(\omega_j\right)} \leq CHn^{-1/d}\|v\|_{\mathcal{H}\left(\omega_j\right)}\, ,
  \end{equation}
  where $C$ is a generic constant independent of $k,H,t$ and $n$.
\end{lemma}
\begin{lemma}
\label{lemma-H-by-L2}
For each $1\leq j \leq N$ and every $v \in U(\omega_j)$, it holds that
\begin{equation}
    \label{eqn-H-by-L2}
    \|v\|_{\mathcal{H}(\omega_{j-1})} \leq C/( tH)\|v\|_{L^{2}(\omega_j)}\, ,
\end{equation}
where $C$ is a generic constant independent of $k,H$ and $t$.
\end{lemma}
With the two lemmas, we are ready to prove Lemma \ref{lemma-compactness-restriction-of-harmonic}.
\begin{proof}[Proof of Lemma \ref{lemma-compactness-restriction-of-harmonic}]
Choose $n =\left\lfloor m/N\right\rfloor$. The proof relies on an iteration argument. We start from $j=N$. By \eqref{eqn-L2-by-H} and \eqref{eqn-H-by-L2}, we get an $n$ dimensional space $W_n(\omega_N) \subset U(\omega_N)$ and a function $w_N \in W_n(\omega_N)$ such that 
\[\|v-w_N\|_{\cH(\omega_{N-1})}\leq C/ (tH) \|v-w_N\|_{L^{2}(\omega_{N})} \leq Ct^{-1}n^{-1/d}\|v\|_{\mathcal{H}(\omega_N)}\, , \]
where we have used the fact that the infimum in \eqref{eqn-L2-by-H} is attained since it is a finite dimensional optimization problem. Here by abuse of notation the value of the constant $C$ varies in different places. It is a generic constant independent of $k,H,t$ and $n$.

Now, we iterate the above process. The function $v-w_N \in U(\omega_{N-1})$, so again by \eqref{eqn-L2-by-H} and \eqref{eqn-H-by-L2}, we get an $n$ dimensional space $W_n(\omega_{N-1}) \subset U(\omega_{N-1})$ and a function $w_{N-1} \in W_n(\omega_{N-1})$ such that 
\[\|v-w_N-w_{N-1}\|_{\cH(\omega_{N-2})}\leq  Ct^{-1}n^{-1/d}\|v-w_N\|_{\mathcal{H}(\omega_{N-1})} \leq  (Ct^{-1}n^{-1/d})^2\|v\|_{\mathcal{H}(\omega_N)}\, . \] 
Repeating the above procedure, we get
\[\|v-\sum_{j=1}^N w_j\|_{\cH(\omega)} \leq  (Ct^{-1}n^{-1/d})^N\|v\|_{\mathcal{H}(\omega^*)}\, , \]
where each $w_j \in U(\omega_j) \subset U(\omega_0) = U(\omega)$. Therefore, there exists an $nN \leq m$ dimensional space $\Phi_{m,e} \subset U(\omega)$ such that
\[\inf_{w \in\Phi_{m,e}} \|v-w\|_{\cH(\omega)} \leq (Ct^{-1}n^{-1/d})^N\|v\|_{\mathcal{H}(\omega^*)}\, .  \]
{For a paramter $q$ to be determined later}, choose $N = \left\lfloor m^{\frac{q}{q+1}} \right\rfloor$, then we obtain
\begin{equation}
\label{eqn-bound-on-error}
\begin{aligned}
    (Ct^{-1}n^{-1/d})^N \leq \left(Ct^{-1}(\frac{m}{N})^{-1/d}\right)^N = \exp(N\left(\frac{1}{d}
    \log(\frac{N}{m})+\log C - \log t\right))\, .
\end{aligned}
\end{equation}
Using $N \leq m^{\frac{q}{q+1}}$ and $t =l_1^{1/N}-1=\exp(\frac{1}{N}\log l_1)-1\geq \frac{1}{N}\log l_1 \geq  m^{-\frac{q}{q+1}}\log l_1$, we can bound the right hand side of \eqref{eqn-bound-on-error} as
\begin{equation}
\label{eqn: exp bound}
\begin{aligned}
    (Ct^{-1}n^{-1/d})^N &\leq \exp\left(-m^{\frac{q}{q+1}} \left((\frac{1}{d}-q)\frac{1}{q+1}\log m -\log C + \log\log l_1\right) \right)\\
    &\leq C_q \exp\left(-m^{\frac{q}{q+1}}\right)\, ,
\end{aligned}
\end{equation}
for some constant $C_q$ that depends on $q, d,C,l_1$, if $q < 1/d$. {Here in the last inequality, we used the fact that when $q < 1/d$, there exists an $M_q$ such that if $m \geq M_q$ then  \[(\frac{1}{d}-q)\frac{1}{q+1}\log m -\log C + \log\log l_1 \geq 1 \, ,\]
and thus $(Ct^{-1}n^{-1/d})^N \leq \exp\left(-m^{\frac{q}{q+1}}\right)$ for $m\geq M_q$. By choosing 
\[C_q = \max_{1\leq m < M_q}  \exp\left(-m^{\frac{q}{q+1}} \left((\frac{1}{d}-q)\frac{1}{q+1}\log m -\log C + \log\log l_1\right) \right) \exp(m^{\frac{q}{q+1}}) + 1 \]
we can prove that  \eqref{eqn: exp bound} is valid.}

{Now, we choose $q < 1/d$ and denote $\frac{q}{q+1} = \frac{1}{d+1}-\epsilon$ for some $\epsilon>0$. There is a one-to-one correspondence between $q$ and  small positive $\epsilon$, so we can also write the error estimate in \eqref{eqn: exp bound} in terms of $\epsilon$ as
\[ (Ct^{-1}n^{-1/d})^N \leq C_{\epsilon}\exp\left(-m^{\frac{1}{d+1}-\epsilon}\right)\, . \]
This completes the proof.}
\end{proof}
\subsubsection{Proof of Lemma \ref{lemma-L2-by-H}}
\label{sec-proof-lemma-L2-by-H}
First, using the spectrum of the Laplacian operator with Neumann's boundary condition, there exists an $n$ dimensional space $S_n \subset H^1(\omega_j)$ such that for any $v \in H^1(\omega_j)$, 
\begin{equation}
\label{eqn-laplacian-eigen}
\inf _{w \in S_{n}}\|v-w\|_{L^{2}\left(\omega_j\right)} \leq CHn^{-1/d}\|v\|_{H^1\left(\omega_j\right)} \leq CHn^{-1/d}\|v\|_{\mathcal{H}\left(\omega_j\right)}\, ,
  \end{equation}
  where $C$ is a generic constant independent of $k,H,t$ and $n$. Equivalently, this implies the identity embedding operator $Q: (\cH(\omega_j), \|\cdot\|_{\cH(\omega_j)}) \to (L^2(\omega_j),\|\cdot\|_{L^2(\omega_j)})$ such that $Qv=v$ is compact and the its $n$-th largest left singular value  $\mu_n \leq CHn^{-1/d}$.
  
  Now, since $U(\omega_j)$ is a closed subspace of $(\cH(\omega_j), \|\cdot\|_{\cH(\omega_j)})$, we can view $Q$ as an operator from $(U(\omega_j), \|\cdot\|_{\cH(\omega_j)})$ to $(L^2(\omega_j),\|\cdot\|_{L^2(\omega_j)})$. Denote its singular values in a non-increasing order by $\{\mu_n'\}$. Using the max-min theorem for singular values, we obtain
  \begin{equation*}
  \begin{aligned}
   \mu_n'&=\max_{S_n \subset U(\omega_j), \mathrm{dim}(S_n)=n} ~\min_{v\in S_n, \|v\|_{\cH(\omega_j)}=1} \|Qv\|_{L^2(\omega_j)}\\
   &\leq \max_{S_n \subset \cH(\omega_j), \mathrm{dim}(S_n)=n} ~\min_{v\in S_n, \|v\|_{\cH(\omega_j)}=1} \|Qv\|_{L^2(\omega_j)}= \mu_n\, .
  \end{aligned}
  \end{equation*}
  Thus, $\mu_n' \leq CHn^{-1/d}$. Therefore, there is an $n$-dimensional space $W_{n}(\omega_j)\subset U(\omega_j)$, such that for all $v \in U(\omega_j)$, it holds that
  \begin{equation*}
\inf _{w \in S_{n}}\|v-w\|_{L^{2}\left(\omega_j\right)} \leq CHn^{-1/d}\|v\|_{H^1\left(\omega_j\right)} \leq CHn^{-1/d}\|v\|_{\mathcal{H}\left(\omega_j\right)}\, .
  \end{equation*}
  The proof is completed.
\subsubsection{Proof of Lemma \ref{lemma-H-by-L2}}
\label{sec-proof-lemma-H-by-L2}
We introduce a cutoff function 
$\eta \in C^{1}(\omega_j)$ such that 
$0\leq \eta\leq 1$, and $\eta=1$ in $\omega_{j-1}$, as well as
$|\nabla \eta(x)| \leq C/(tH)$ for some constant $C$ independent of $k,H$ and $t$.

For any $v \in U(\omega_j)$, we use the test function $\eta^{2} v$ and the weak form to get
\begin{equation}
\label{eqn-Helmholtz-harmonic-weak-form}
(A \nabla v, \nabla (\eta^{2} v))_{\omega_j} -k^2(V v, V \eta^{2}v)_{\omega_j}=0\, ,
\end{equation}
where we have used the definition of $U(\omega_j)$ (see the beginning of Subsection \ref{sec-main-idea-of-proof}), and the property of our construction that  $\partial \omega_j \cap (\Gamma_N \cup \Gamma_R) = \emptyset$.

Using the relation $\|A^{1 / 2}\eta  \nabla v\|_{L^2(\omega_j)}^2 = (A \nabla v, \eta^{2}\nabla  v)_{\omega_j}$ and the above formula, we obtain
\begin{equation}
\begin{aligned}
\|A^{1 / 2}\eta  \nabla v\|_{L^2(\omega_j)}^2 &=- 2(A^{1 / 2}\eta  \nabla v, A^{1 / 2}v\nabla \eta )_{\omega_j} +k^2(V v,V\eta^{2} v)_{\omega_j} \, ,\\
& \leq \frac{1}{2}\|A^{1 / 2}\eta  \nabla v\|_{L^2(\omega_j)}^2 + 2\|A^{1/2}v\nabla \eta\|_{L^2(\omega_j)}^2 + k^2V_{\max}^2\|v\|_{L^2(\omega_j)}^2\, ,
\end{aligned}
\end{equation}
which leads to $\|A^{1 / 2}\eta  \nabla v\|_{L^2(\omega_j)}^2 \leq 4\|A^{1/2}v\nabla \eta\|_{L^2(\omega_j)}^2 + 2k^2V_{\max}^2\|v\|_{L^2(\omega_j)}^2$. Therefore, using the fact that $\eta = 1$ in $\omega_{j-1}$, we have
\begin{equation}
    \begin{aligned}
    \|v\|_{\cH(\omega_{j-1})}^2 &\leq \|A^{1 / 2}\eta  \nabla v\|_{L^2(\omega_j)}^2 + k^2V_{\max}^2\|v\|_{L^2(\omega_j)}^2\\
    &\leq 4\|A^{1/2}v\nabla \eta\|_{L^2(\omega_j)}^2 + 3k^2V_{\max}^2\|v\|_{L^2(\omega_j)}^2\\
    & \leq \left(\frac{4C^2}{(tH)^2}+3k^2V_{\max}^2\right)\|v\|_{L^2(\omega_j)}^2\\
    & \leq \frac{C'^2}{(tH)^2}\|v\|_{L^2(\omega_j)}^2\, ,
    \end{aligned}
\end{equation}
for some $C'$ independent of $k,H$ and $t$, where we have used Assumption \ref{small mesh} such that $kV_{\max}H\leq C''$ for $C''= A_{\min}^{1/2}/(\sqrt{2}C_P)$. This completes the proof.
    
\subsubsection{Proof of Lemma \ref{lemma: bound H 1/2 by energy norm}}
\label{subsec: Proof of Lemma bound H 1/2 by energy norm}
We use Lemma 3.9 of \cite{chen2020exponential}, which implies that
\begin{equation}
\left\|R_e v\right\|_{\cH^{1/2}(e)}\leq C 
\left(\|A^{1/2}\nabla v\|_{L^2(\omega)} + H\|\nabla \cdot (A\nabla v)\|_{L^2(\omega)}\right) \, ,
\end{equation}
for some $C$ independent of $k, H$.
By a triangular inequality, we have
\begin{equation*}
    \begin{aligned}
    H\|\nabla \cdot (A\nabla v)\|_{L^2(\omega)} &\leq H\|k^2V^2v\|_{L^2(\omega)}+H\|\nabla \cdot (A\nabla v)+k^2V^2v\|_{L^2(\omega)} \\
    &\leq C'\|kVv\|_{L^2(\omega)}+H\|\nabla \cdot (A\nabla v)+k^2V^2v\|_{L^2(\omega)}\, ,
    \end{aligned}
\end{equation*}
where we have used Assumption \ref{small mesh} such that $kV_{\max}H\leq C'$ for $C'= A_{\min}^{1/2}/(\sqrt{2}C_P)$. Now, using the definition of the $\cH(\omega)$ norm, we have
\[\|A^{1/2}\nabla v\|_{L^2(\omega)} + C'\|kVv\|_{L^2(\omega)} \leq C''\|v\|_{\cH(\omega)}\, ,  \]
for some generic constant $C''$ that does not depend on anything else. Combining the above inequalities concludes the proof.
\subsubsection{For Edges Connected to the Boundary}
The above proofs are for interior edges. For edges connected to the boundary, we need a different geometric relation, as depicted in the right of Figure \ref{fig:omega omega star}. The quantitative characterization of this geometric relation is the same as that in Subsection 3.3.2 of \cite{chen2020exponential}, which introduces three other parameters $l_4,l_5,l_6$ to describe the geometry associated with edges, similar to $l_1,l_2,l_3$ for interior edges.

The main idea of the proof for this case is the same as that for the interior edges. We need to prove Lemmas \ref{lemma-compactness-restriction-of-harmonic} and \ref{lemma: bound H 1/2 by energy norm} for edges connected to the boundary. The proof of Lemma  \ref{lemma: bound H 1/2 by energy norm} remains the same since its statement holds for all edges. To prove Lemma \ref{lemma-compactness-restriction-of-harmonic}, we again use the same strategy in Subsection \ref{sec-proof-lemma-compactness-restriction-of-harmonic}, by establishing Lemmas \ref{lemma-L2-by-H} and \ref{lemma-H-by-L2} and then using an iteration argument. The iteration argument and the proof for Lemma \ref{lemma-L2-by-H} remain unchanged. For Lemma \ref{lemma-H-by-L2}, the only slight change is \eqref{eqn-Helmholtz-harmonic-weak-form}, which becomes
\begin{equation}
(A \nabla v, \nabla (\eta^{2} v))_{\omega_j} -k^2(V v, V \eta^{2}v)_{\omega_j}=(T_k v, \eta^2 v)_{\partial \omega_j \cap (\Gamma_N \cup \Gamma_R)}\, ,
\end{equation}
due to the boundary conditions involved. However, since $\Re (T_k v, \eta^2 v)_{\partial \omega_j \cap (\Gamma_N \cup \Gamma_R)} \leq 0$, the conclusion of Lemma \ref{lemma-H-by-L2} still holds. 

Therefore, the result also holds for edges connected to the boundary.

\subsection{Proof of Proposition \ref{prop: small os bubble}}
\label{subsec: Proof of Proposition prop: small os bubble}
First we have the bound on the oversampling bubble part
in \eqref{eqn-small-os-bubble}:
\begin{equation}
         \|u^{\sfb}_{\omega_e}\|_{\mathcal{H}(\omega_e)} \leq  \frac{3C'_P}{A_{\min}^{1/2}}H\|f\|_{L^{2}(\omega_e)}\, .
     \end{equation}
Applying Lemma \ref{lemma: bound H 1/2 by energy norm} and the definition of $u^{\sfb}_{\omega_e}$ leads to
\begin{equation}
\begin{aligned}
\|R_eu^{\sfb}_{\omega_e}\|_{\cH^{1/2}(e)}&\leq C\left(\|u^{\sfb}_{\omega_e}\|_{\mathcal{H}(\omega)}+ H\|\nabla \cdot (A\nabla u^{\sfb}_{\omega_e})+k^2V^2u^{\sfb}_{\omega_e}\|_{L^2(\omega)}\right)\\
&\leq C'H\|f\|_{L^2(\omega_e)}\, ,
\end{aligned}
\end{equation}
where $C'$ is a constant independent of $k$ and $H$.

\subsection{Proof of Theorem \ref{eqn:approximation property}}
\label{Gard}
\begin{proof}
\yw{
Define $e_S=u^\sfh-u^\sfs-u_S \in V^\sfh$. Take $\psi=N_{k}^{\star}(e_S)$. It holds that \[\|e_S\|_{L^{2}(\Omega)}^{2} =a(e_S, \psi) = a(e_S, \psi-v)\, , \]
for any $v \in S$, due to the property of the Galerkin solution. Thus, using the boundedness of $a(\cdot,\cdot)$, we obtain that 
\begin{equation}
\label{eqn-error-estimate-L2-err}
    \|e_S\|_{L^{2}(\Omega)}^{2}\leq C_c\|e_S\|_{\mathcal{H}(\Omega)}\|\psi-v\|_{\mathcal{H}(\Omega)}=C_c\|e_S\|_{\mathcal{H}(\Omega)}\|\overline{\psi}-\overline{v}\|_{\mathcal{H}(\Omega)}\, .
\end{equation}
As $\overline{\psi} = N_k\overline{e_S}$ according to the definition of the adjoint problem in Subsection \ref{sec-analytic-results}, we can take infimum of $v$ over $S$, using the fact that $S = \overline{S}$, the definition \eqref{apro proxy}, the inequality \eqref{eqn-error-estimate-L2-err}, to get 
\[ \|e_S\|_{L^{2}(\Omega)}^{2} \leq C_c\|e_S\|_{\cH(\Omega)} \cdot \eta(S)\|\overline{e_S}\|_{L^2(\Omega)}\, ,\]
which leads to the desired $L^2(\Omega)$ error estimate: $\|e_S\|_{L^{2}(\Omega)} \leq C_c\eta(S)\|e_S\|_{\cH(\Omega)}$.

For the $\cH(\Omega)$ error, the property of Galerkin's solution implies that for any $v \in S$, we have
\begin{equation}
\label{4.5}
\begin{aligned}\|e_S\|_{\mathcal{H}(\Omega)}^{2} &
=\Re a(e_S, e_S)+\{\|e_S\|_{\mathcal{H}(\Omega)}^{2}-\Re a(e_S, e_S)\} \\ 
&=\Re a(e_S, u^\sfh-u^\sfs-v)+2 \|kV(x)e_S\|_{L^{2}(\Omega)}^{2}+\Re (T_{k} e_S, e_S)_{{\Gamma_N\cup\Gamma_R}}\\
&\leq C_{c}\|e_S\|_{\mathcal{H}(\Omega)}\|u^\sfh-u^\sfs-v\|_{\mathcal{H}(\Omega)}+2 (kV_{\max} C_{c} \eta(S))^2\|e_S\|_{\mathcal{H}(\Omega)}^{2} \, ,
\end{aligned}
\end{equation}
where we have used the fact that $\Re (T_{k} e_S, e_S)_{{\Gamma_N\cup\Gamma_R}} \leq 0$ and  the $L^2(\Omega)$ error estimate that we established earlier.

By the assumption $k \eta^\sfh(S)  \leq 1/{(2C_cV_{\max})}$, the last term in \eqref{4.5} is bounded by  $\frac{1}{2}\|e_S\|_{\mathcal{H}(\Omega)}^2$. Thus due to the  arbitrariness of $v$, we arrive at 
\[\left\|e_S\right\|_{\mathcal{H}(\Omega)} \leq 2 C_{c} \inf _{v \in S}\|u^\sfh-v\|_{\mathcal{H}(\Omega)}\, . \]
This completes the proof for the first part.
Next, we move to the proof for the discrete inf-sup stability.
 For any $v\in S$, set $z=2N_{k}^{\star} (k^2V^2v) \in \cH(\Omega)$ so that $a(v,z)=2k^{2}(V^2 v,  {v})_{\Omega} $. Plugging $v$ and $v+z$ into the sesquilinear form yields:
\begin{equation*}
\begin{aligned} 
a(v, v+z) & = a(v,v) + a(v,z)\\
& = (A\nabla v, \nabla {v})_{\Omega}-k^{2}(V^2 v,  {v})_{\Omega}  -( T_{k} v,  {v})_{{\Gamma_N\cup\Gamma_R}}+2k^{2}(V^2 v,  {v})_{\Omega}\\
 &=\|v\|_{\mathcal{H}(\Omega)}^{2}  -( T_{k} v,  {v})_{{\Gamma_N\cup\Gamma_R}}\, . \end{aligned}
\end{equation*}
By the definition of $T_k$, $\Re( T_{k} v,  {v})_{{\Gamma_N\cup\Gamma_R}} \leq 0$, so it holds that
 \begin{equation*}
 \Re a(v, v+z)\geq 
\|v\|_{\mathcal{H}(\Omega)}^{2}\, .
\end{equation*}
Now, by the definition of the adjoint problem, we have $\overline{z}=2N_k(k^2V^2\overline{v})$. 
Let $z_S \in S$ achieve the best approximation in $\eqref{apro proxy}$ for $f = 2k^2V^2\overline{v}$, so that
\begin{equation}
\label{eqn-best-approximation-zh}
    \|\overline{z}^\sfh -\overline{z}^\sfs- z_S\|_{\cH(\Omega)} \leq \eta(S)\|2k^2V^2\overline{v}\|_{L^2(\Omega)}\leq 2kV_{\max} \eta(S)\|v\|_{\cH(\Omega)}\, .
\end{equation}

We can choose $v'=v+\overline{z_S} \in S$ to compute
\begin{equation*}
\Re a(v, v+\overline{z_S})=\Re a(v, v+z)-\Re a(v, z-\overline{z_S})\geq \|v\|_{\mathcal{H}(\Omega)}^{2}-
C_{c}\|v\|_{\mathcal{H}(\Omega)}\|\overline{z}-z_{S}\|_{\mathcal{H}(\Omega)}\, .
\end{equation*}
We use the bound in \eqref{eqn-best-approximation-zh} and the triangle inequality to get
\begin{equation*}
    |a(v, v+\overline{z_S})|\geq \|v\|_{\mathcal{H}(\Omega)}^{2}(1-2 C_{c}kV_{\max}\eta(S)  )-C_c\|v\|_{\mathcal{H}(\Omega)}(\|z^\sfs\|_{\mathcal{H}(\Omega)}+\|z^\sfb\|_{\mathcal{H}(\Omega)})\, .
\end{equation*}
Meanwhile, by a triangle inequality, we get
\begin{equation*}
\begin{aligned}
\|v+\overline{z_S}\|_{\mathcal{H}(\Omega)}  \leq\|v\|_{\mathcal{H}(\Omega)}+\|z^\sfh-z^\sfs-\overline{z_S}\|_{\mathcal{H}(\Omega)}+\|z\|_{\mathcal{H}(\Omega)}+\|z^\sfs\|_{\mathcal{H}(\Omega)}+\|z^\sfb\|_{\mathcal{H}(\Omega)}\,.
\end{aligned}
\end{equation*}

Finally we are left to estimate the energy norm of $z$ and its fine scale parts. By the stability estimate in \eqref{eqn: stability}, we have
\[\|z\|_{\mathcal{H}(\Omega)}\leq C_{\mathrm{stab}}(k)\|2k^2V^2v\|_{L^2(\Omega)}\leq 2C_{\mathrm{stab}}(k) k V_{\max}\|v\|_{\mathcal{H}(\Omega)}\, ,\]
and by the bound on the fine part as given by \eqref{bbh}, it holds that
\[\|z^\sfs\|_{\mathcal{H}(\Omega)}+\|z^\sfb\|_{\mathcal{H}(\Omega)} \leq  C_s H\|2k^2V^2\overline{v}\|_{L^{2}(\Omega)}\leq2C_s H k V_{\max} \|v\|_{\mathcal{H}(\Omega)}\, .\]
Therefore, we obtain
\begin{align*}
    \sup _{v' \in S \backslash\{0\}}& \frac{|a(v, v')|}{\|v\|_{\mathcal{H}(\Omega)}\|v'\|_{\mathcal{H}(\Omega)}} \geq \frac{|a(v, v+\overline{z_S})|}{\|v\|_{\mathcal{H}(\Omega)}\|v+\overline{z_S}\|_{\mathcal{H}(\Omega)}} \\
    & \geq  \frac{(1-2 \eta(S)C_{c}kV_{\max} -2C_cC_sHkV_{\max} )\|v\|_{\mathcal{H}(\Omega)}^{2}}{(1+2 \eta(S)kV_{\max} +2 C_{\mathrm{stab}}(k)kV_{\max}+2C_sHkV_{\max} )\|v\|^2_{\mathcal{H}(\Omega)}}\, .
\end{align*}
Using the assumptions that $\eta(S)kV_{\max}   \leq 1/(4C_c)$ and $C_sHkV_{\max}\leq 1/(8C_c)$, we obtain the desired estimate.}
\end{proof}
\section{Generalization to 3D problems}
\label{sec 3d}
\yw{We discuss how to generalize the 2D edge basis framework to 3D problems.
Now, in the mesh structure, the cubes will be the elements $\cT_H=\{T_1,T_2,...,T_r\}$. As in the 2D case, we have the collection of nodes $\cN_H=\{x_1,x_2,...,x_p\}$ and the collection of edges $\cE_H=\{e_1,e_2,...,e_q\}$. We also have the collection of faces $\cF_H=\{f_1,f_2,...,f_s\}$ and the face set is defined via $F_H:=\bigcup_{f \in \cF_H} f$. 

As before we perform the local and global harmonic-bubble decomposition $u=u^\sfh+u^\sfb$ for a mesh size $H=O(1/k)$. We want to approximate the Helmholtz-harmonic part $u^\sfh$, which is uniquely determined by its traces on the face set $F_H$. To achieve so, we will construct local face basis functions. 

In order to isolate the approximation task to each face (similar to 2D, where we isolate the task to each edge via nodal interpolation), we first introduce nodal and edge basis functions. Similar to 2D, these basis functions are constructed via linear interpolations from nodal or edge values to facial values. We detail the construction below. 
\begin{enumerate}
    \item \textbf{Nodal basis:} We define a nodal basis associated with each $x_i$ that satisfies $\phi_i(x_j)=\delta_{ij}$, with its value on the faces as linear interpolations between the values of its boundary four nodes. The Helmholtz-harmonic extension of the facial values leads to the basis functions.
    \item \textbf{Edge function and edge basis:} For each edge $e_i$, we consider edge functions $\phi^{(1)}_i$ supported on $e_i$ and vanishing on the endpoint nodes: namely $\phi^{(1)}_i(x)=0$ for any $x\in e_j$ when $j\neq i$. We extend its value to the faces via linear interpolation between the values on the edges. Then, its value in the domain is identified via Helmholtz-harmonic extension. This defines an edge function.
    
     We know that the restriction of any edge function on the faces is supported in the $4$ faces containing $e_i$ as an edge. The associated Helmholtz-harmonic function has its support in the $4$ cubes containing $e_i$ as an edge. Notice that we could have an infinite number of edge functions on every edge. We will construct a finite number of edge basis later for approximation. This is different from the 2D case and the nodal basis functions.
\end{enumerate} 
With nodal basis and edge functions defined, we can define similarly to Section \ref{sec-Localization of Approximation} the interpolation of face functions. Recall that similar to the 2D case, face functions are functions defined on the face set and whose Helmholtz-harmonic extension leads to functions defined in the domain.
We pick up the notation there and denote the nodal interpolation of a face function $v$ as $I_H v$. Now, $v-I_H v$ restricted to the edge set is localized onto each edge, and we can use edge functions to account for the edge values. We denote the edge interpolation as $J_H v$ such that $v-I_H v-J_H v$ vanishes on each edge. Here, $J_Hv$ equals $v-I_Hv$ on each edge and is linear in each face. 

The operators $I_H$, $J_H$ are linear, with the image of $I_H, J_H$ spanned by the nodal and edge functions, respectively. We only need to find edge basis functions to approximate $J_H v$ and face basis functions to approximate $F_Hv:=v-I_H v-J_H v$. 

Notice that the interpolation operators $J_H$ and $F_H$ are locally well defined for each edge $e$
and face $f$ respectively: $J_H v$ restricted on $e$ is only dependent on the values of $v$ on $\overline{e}$ and $F_H v$ restricted on $f$ is only dependent on the values of $v$ on $\overline{f}$. We use the notation $J_e$ and $F_f$ to denote their restrictions on $e$ and $f$ so that $J_H = \sum_{e} J_e$ and $F_H = \sum_{f} F_f$.

We use oversampling to achieve exponential approximation accuracy for both edges and faces. To be specific, we consider an oversampling element $\omega_e$ associated with each edge $e$ such that $e$ is contained in the interior; for example, similar to \eqref{eqn: os domain 1 layer} we take\[
        \omega_e=\overline{\bigcup \{T\in \cT_H: \overline{T} \cap e \neq \emptyset\}}\, .
\]  Similarly, we can take an oversampling element $\omega_f$ associated with a face $f$ to be \[
        \omega_f=\overline{\bigcup \{T\in \cT_H: \overline{T} \cap f \neq \emptyset\}}\, .
\] Now, we perform harmonic-splitting on each oversampling domain and define the two special harmonic functions $u_{(1)}^\sfs$ and $u_{(2)}^\sfs$ to account for the effect of oversampling bubble part for edges and faces, similar as the 2D framework. The function $u_{(1)}^\sfs$ lies in the span of edge functions, with its restriction on each edge $e$ equal to $J_e u_{\omega_e}^\sfb$. The restriction of $u_{(2)}^\sfs$ on each face $f$ equals $F_f u_{\omega_f}^\sfb$. Both special functions are tractable by solving local problems. Then, we have the localization of the approximation \[\begin{aligned}
u-u^\sfb-u_{(1)}^\sfs-u_{(2)}^\sfs&=I_H u^\sfh+J_H u^\sfh-u_{(1)}^\sfs+F_H u^\sfh-u_{(2)}^\sfs\\
&=I_H u + \sum_{e}J_e u_{\omega_e}^\sfh + \sum_{f}F_f u_{\omega_f}^\sfh\, .
\end{aligned}\]

We only need to show that for each edge $e$ and each face $f$, the parts $J_e u_{\omega_e}^\sfh$ and $F_f u_{\omega_f}^\sfh$ can be approximated with exponential efficiency. Similar to the 2D case, by exponential decaying eigenvalues of concentric domains in Lemma \ref{lemma-compactness-restriction-of-harmonic},  we only need to show the boundedness of the restrictions $J_e$ and $F_f$, similar to Lemma \ref{lemma: bound H 1/2 by energy norm}; see the two propositions below. 
\begin{proposition}

\label{pro-compactness-restriction-of-edge}
  For $d=3$ and an oversampling domain $\omega$ such that $l_1 e\in\omega$ for a uniform constant $l_1>1$. If $v \in H^1(\omega)$ and $\nabla \cdot (A\nabla v) \in L^2(\omega)$, it holds that
\begin{equation}
\left\|J_e v\right\|_{\cH^{1/2}(e)}\leq C 
\left(\|v\|_{\cH(\omega)} + H\|\nabla \cdot (A\nabla v)+k^2V^2v\|_{L^2(\omega)}\right) \, ,
\end{equation}
for some $C$ independent of $k$ and $H$.
\end{proposition}
\begin{proposition}
\label{pro-compactness-restriction-of-face}
 For $d=3$ and an oversampling domain $\omega$ such that $l_1 f\in\omega$ for a uniform constant $l_1>1$. If $v \in H^1(\omega)$ and $\nabla \cdot (A\nabla v) \in L^2(\omega)$, it holds that
\begin{equation}
\left\|F_f v\right\|_{\cH^{1/2}(f)}\leq C 
\left(\|v\|_{\cH(\omega)} + H\|\nabla \cdot (A\nabla v)+k^2V^2v\|_{L^2(\omega)}\right) \, ,
\end{equation}
for some $C$ independent of $k$ and $H$.

\end{proposition}
The norms $\cH^{1/2}(e)$ and $\cH^{1/2}(f)$ are defined via the Helmholtz-harmonic extension of the edge functions and face functions, similar to Definition \ref{def: H 1/2 norm}. They are equivalent to the $H^{1/2}_{00}$ norm on the faces and are bounded by the local $C^\alpha$ norm.
The proofs of the propositions will be analogous to Lemma \ref{lemma: bound H 1/2 by energy norm} in the 2D case, where we use local $C^\alpha$ estimates to conclude the proofs.

Combing all the above, we obtain a framework in 3D, using nodal basis, edge basis, and face basis, together with the bubble part to get nearly exponentially convergent accuracy.
}
\section{Concluding Remarks}
\label{sec-concluding-remark}
In this paper, we have developed a multiscale framework for solving the Helmholtz equation in heterogeneous media and high frequency regimes. The coarse-fine scale decomposition of the solution space is motivated by the MsFEM. In our algorithm, the coarse scale Helmholtz-harmonic part and the fine scale bubble part are computed separately. Their own structures are carefully explored, such as the low complexity of the coarse part and the locality of the fine part. A nearly exponential rate of convergence is proved rigorously and is confirmed numerically for a wide range of the Helmholtz equations with rough coefficients, high contrast, and mixed boundary conditions.

Perhaps surprisingly, our framework implies that designing an accurate multiscale method for the Helmholtz equation is not much more different from that for the elliptic equation. Many techniques in the elliptic case can be successfully adapted once the mesh size satisfies $H=O(1/k)$, a condition that does not suffer from the pollution effect.
This work also demonstrates the broad applicability of our exponentially convergent multiscale framework proposed originally in \cite{chen2020exponential}.

Most discussions in this paper are concerned with dimension $d=2$. \yw{We provide a generalization incorporating face basis with dimension $d=3$ in Section \ref{sec 3d}, where a similar idea of localizing the basis via a non-overlapped decomposition is exploited.}

It is of future interest to extend this methodology systematically 
to other equations such as the Schrodinger equation, where the problem is time-dependent and the potential function could introduce indefiniteness into the system. On the other hand, developing a better theoretical understanding of the behavior of the multiscale framework with respect to high contrast in the media is also an exciting direction for further exploration.


\begin{thebibliography}{10}

\bibitem{aziz1988two}
{\sc A.~K. Aziz, R.~B. Kellogg, and A.~B. Stephens}, {\em A two point boundary
  value problem with a rapidly oscillating solution}, Numerische Mathematik, 53
  (1988), pp.~107--121.

\bibitem{babuska2011optimal}
{\sc I.~Babu{\v{s}}ka and R.~Lipton}, {\em Optimal local approximation spaces
  for generalized finite element methods with application to multiscale
  problems}, Multiscale Modeling \& Simulation, 9 (2011), pp.~373--406.

\bibitem{babuvska2020multiscale}
{\sc I.~Babu{\v{s}}ka, R.~Lipton, P.~Sinz, and M.~Stuebner}, {\em
  Multiscale-spectral {GFEM} and optimal oversampling}, Computer Methods in
  Applied Mechanics and Engineering, 364 (2020), p.~112960.

\bibitem{babuvska2000can}
{\sc I.~Babu{\v{s}}ka and J.~Osborn}, {\em Can a finite element method perform
  arbitrarily badly?}, Mathematics of Computation, 69 (2000), pp.~443--462.

\bibitem{babuska1997pollution}
{\sc I.~M. Babuska and S.~A. Sauter}, {\em Is the pollution effect of the fem
  avoidable for the {H}elmholtz equation considering high wave numbers?}, SIAM
  Journal on numerical analysis, 34 (1997), pp.~2392--2423.

\bibitem{bernkopf2022wavenumber}
{\sc M.~Bernkopf, T.~Chaumont-Frelet, and J.~M. Melenk}, {\em
  Wavenumber-explicit stability and convergence analysis of hp finite element
  discretizations of helmholtz problems in piecewise smooth media}, arXiv
  preprint arXiv:2209.03601,  (2022).

\bibitem{betcke2011condition}
{\sc T.~Betcke, S.~N. Chandler-Wilde, I.~G. Graham, S.~Langdon, and
  M.~Lindner}, {\em Condition number estimates for combined potential integral
  operators in acoustics and their boundary element discretisation}, Numerical
  Methods for Partial Differential Equations, 27 (2011), pp.~31--69.

\bibitem{brown2017multiscale}
{\sc D.~L. Brown, D.~Gallistl, and D.~Peterseim}, {\em Multiscale
  {Petrov-Galerkin} method for high-frequency heterogeneous {H}elmholtz
  equations}, in Meshfree methods for partial differential equations VIII,
  Springer, 2017, pp.~85--115.

\bibitem{buhr2018randomized}
{\sc A.~Buhr and K.~Smetana}, {\em Randomized local model order reduction},
  SIAM journal on scientific computing, 40 (2018), pp.~A2120--A2151.

\bibitem{chen2020randomized}
{\sc K.~Chen, Q.~Li, J.~Lu, and S.~J. Wright}, {\em Randomized sampling for
  basis function construction in generalized finite element methods},
  Multiscale Modeling \& Simulation, 18 (2020), pp.~1153--1177.

\bibitem{chen2019subsampled}
{\sc Y.~Chen and T.~Y. Hou}, {\em Function approximation via the subsampled
  poincar{\'e} inequality}, Discrete and Continuous Dynamical Systems-Series A,
  41 (2021), pp.~169--199.

\bibitem{chen2020multiscale}
{\sc Y.~Chen and T.~Y. Hou}, {\em Multiscale elliptic pde upscaling and
  function approximation via subsampled data}, Multiscale Modeling \&
  Simulation, 20 (2022), pp.~188--219.

\bibitem{chen2020exponential}
{\sc Y.~Chen, T.~Y. Hou, and Y.~Wang}, {\em Exponential convergence for
  multiscale linear elliptic pdes via adaptive edge basis functions},
  Multiscale Modeling \& Simulation, 19 (2021), pp.~980--1010.

\bibitem{chung2018constraint}
{\sc E.~T. Chung, Y.~Efendiev, and W.~T. Leung}, {\em Constraint energy
  minimizing generalized multiscale finite element method}, Computer Methods in
  Applied Mechanics and Engineering, 339 (2018), pp.~298--319.

\bibitem{efendiev_generalized_2013}
{\sc Y.~Efendiev, J.~Galvis, and T.~Y. Hou}, {\em Generalized multiscale finite
  element methods ({GMsFEM})}, Journal of Computational Physics, 251 (2013),
  pp.~116 -- 135.

\bibitem{engquist2011sweeping}
{\sc B.~Engquist and L.~Ying}, {\em Sweeping preconditioner for the {H}elmholtz
  equation: hierarchical matrix representation}, Communications on pure and
  applied mathematics, 64 (2011), pp.~697--735.

\bibitem{engquist2012sweeping}
{\sc B.~Engquist and L.~Ying}, {\em Sweeping preconditioner for the {H}elmholtz
  equation: moving perfectly matched layers}, Multiscale Modeling \&
  Simulation, 9 (2011), pp.~686--710.

\bibitem{engquist2018approximate}
{\sc B.~Engquist and H.~Zhao}, {\em Approximate separability of the green's
  function of the helmholtz equation in the high frequency limit},
  Communications on Pure and Applied Mathematics, 71 (2018), pp.~2220--2274.

\bibitem{esterhazy2012stability}
{\sc S.~Esterhazy and J.~M. Melenk}, {\em On stability of discretizations of
  the {H}elmholtz equation}, in Numerical analysis of multiscale problems,
  Springer, 2012, pp.~285--324.

\bibitem{freese2021super}
{\sc P.~Freese, M.~Hauck, and D.~Peterseim}, {\em Super-localized orthogonal
  decomposition for high-frequency helmholtz problems}, arXiv preprint
  arXiv:2112.11368,  (2021).

\bibitem{fu2019edge}
{\sc S.~Fu, E.~Chung, and G.~Li}, {\em Edge multiscale methods for elliptic
  problems with heterogeneous coefficients}, Journal of Computational Physics,
  396 (2019), pp.~228--242.

\bibitem{fu2017fast}
{\sc S.~Fu and K.~Gao}, {\em A fast solver for the {H}elmholtz equation based
  on the generalized multiscale finite-element method}, Geophysical Journal
  International, 211 (2017), pp.~797--813.

\bibitem{fu2019wavelet}
{\sc S.~Fu, G.~Li, R.~Craster, and S.~Guenneau}, {\em Wavelet-based edge
  multiscale finite element method for {H}elmholtz problems in perforated
  domains}, arXiv preprint arXiv:1906.08453,  (2019).

\bibitem{gallistl2015stable}
{\sc D.~Gallistl and D.~Peterseim}, {\em Stable multiscale {P}etrov--{G}alerkin
  finite element method for high frequency acoustic scattering}, Computer
  Methods in Applied Mechanics and Engineering, 295 (2015), pp.~1--17.

\bibitem{graham2020stability}
{\sc I.~Graham and S.~Sauter}, {\em Stability and finite element error analysis
  for the {H}elmholtz equation with variable coefficients}, Mathematics of
  Computation, 89 (2020), pp.~105--138.

\bibitem{graham2019helmholtz}
{\sc I.~G. Graham, O.~R. Pembery, and E.~A. Spence}, {\em The {H}elmholtz
  equation in heterogeneous media: a priori bounds, well-posedness, and
  resonances}, Journal of Differential Equations, 266 (2019), pp.~2869--2923.

\bibitem{griepentrog2001linear}
{\sc J.~A. Griepentrog and L.~Recke}, {\em Linear elliptic boundary value
  problems with non--smooth data: normal solvability on {S}obolev--{C}ampanato
  spaces}, Mathematische Nachrichten, 225 (2001), pp.~39--74.

\bibitem{hauck2021multi}
{\sc M.~Hauck and D.~Peterseim}, {\em Multi-resolution localized orthogonal
  decomposition for {H}elmholtz problems}, arXiv preprint arXiv:2104.11190,
  (2021).

\bibitem{henning2013oversampling}
{\sc P.~Henning and D.~Peterseim}, {\em Oversampling for the multiscale finite
  element method}, Multiscale Modeling \& Simulation, 11 (2013),
  pp.~1149--1175.

\bibitem{hetmaniuk2014error}
{\sc U.~Hetmaniuk and A.~Klawonn}, {\em Error estimates for a two-dimensional
  special finite element method based on component mode synthesis}, Electron.
  Trans. Numer. Anal, 41 (2014), pp.~109--132.

\bibitem{hetmaniuk2010special}
{\sc U.~L. Hetmaniuk and R.~B. Lehoucq}, {\em A special finite element method
  based on component mode synthesis}, ESAIM: Mathematical Modelling and
  Numerical Analysis, 44 (2010), pp.~401--420.

\bibitem{hou2015optimal}
{\sc T.~Y. Hou and P.~Liu}, {\em Optimal local multi-scale basis functions for
  linear elliptic equations with rough coefficient}, Discrete and Continuous
  Dynamical Systems, 36 (2016), pp.~4451--4476.

\bibitem{hou_multiscale_1997}
{\sc T.~Y. Hou and X.-H. Wu}, {\em A multiscale finite element method for
  elliptic problems in composite materials and porous media}, Journal of
  Computational Physics, 134 (1997), pp.~169 -- 189.

\bibitem{lafontaine2019most}
{\sc D.~Lafontaine, E.~A. Spence, and J.~Wunsch}, {\em For most frequencies,
  strong trapping has a weak effect in frequency-domain scattering}, arXiv
  preprint arXiv:1903.12172,  (2019).

\bibitem{lafontaine2022wavenumber}
{\sc D.~Lafontaine, E.~A. Spence, and J.~Wunsch}, {\em Wavenumber-explicit
  convergence of the hp-fem for the full-space heterogeneous helmholtz equation
  with smooth coefficients}, Computers \& Mathematics with Applications, 113
  (2022), pp.~59--69.

\bibitem{ma2021wavenumber}
{\sc C.~Ma, C.~Alber, and R.~Scheichl}, {\em Wavenumber explicit convergence of
  a multiscale gfem for heterogeneous helmholtz problems}, arXiv preprint
  arXiv:2112.10544,  (2021).

\bibitem{ma2021error}
{\sc C.~Ma and R.~Scheichl}, {\em Error estimates for fully discrete
  generalized fems with locally optimal spectral approximations}, arXiv
  preprint arXiv:2107.09988,  (2021).

\bibitem{ma2021novel}
{\sc C.~Ma, R.~Scheichl, and T.~Dodwell}, {\em Novel design and analysis of
  generalized fe methods based on locally optimal spectral approximations},
  arXiv preprint arXiv:2103.09545,  (2021).

\bibitem{malqvist_localization_2014}
{\sc A.~M{\aa}lqvist and D.~Peterseim}, {\em Localization of elliptic
  multiscale problems}, Mathematics of Computation, 83 (2014), pp.~2583--2603.

\bibitem{melenk2010convergence}
{\sc J.~Melenk and S.~Sauter}, {\em Convergence analysis for finite element
  discretizations of the {H}elmholtz equation with {Dirichlet-to-Neumann}
  boundary conditions}, Mathematics of Computation, 79 (2010), pp.~1871--1914.

\bibitem{melenk2011wavenumber}
{\sc J.~M. Melenk and S.~Sauter}, {\em Wavenumber explicit convergence analysis
  for {G}alerkin discretizations of the {H}elmholtz equation}, SIAM Journal on
  Numerical Analysis, 49 (2011), pp.~1210--1243.

\bibitem{moiola2019acoustic}
{\sc A.~Moiola and E.~A. Spence}, {\em Acoustic transmission problems:
  wavenumber-explicit bounds and resonance-free regions}, Mathematical Models
  and Methods in Applied Sciences, 29 (2019), pp.~317--354.

\bibitem{oberai1998multiscale}
{\sc A.~A. Oberai and P.~M. Pinsky}, {\em A multiscale finite element method
  for the {H}elmholtz equation}, Computer Methods in Applied Mechanics and
  Engineering, 154 (1998), pp.~281--297.

\bibitem{ohlberger2018new}
{\sc M.~Ohlberger and B.~Verfurth}, {\em A new heterogeneous multiscale method
  for the {H}elmholtz equation with high contrast}, Multiscale Modeling \&
  Simulation, 16 (2018), pp.~385--411.

\bibitem{owhadi_multigrid_2017}
{\sc H.~Owhadi}, {\em Multigrid with rough coefficients and multiresolution
  operator decomposition from hierarchical information games}, SIAM Review, 59
  (2017), pp.~99--149.

\bibitem{owhadi2019operator}
{\sc H.~Owhadi and C.~Scovel}, {\em Operator-Adapted Wavelets, Fast Solvers,
  and Numerical Homogenization: From a Game Theoretic Approach to Numerical
  Approximation and Algorithm Design}, vol.~35, Cambridge University Press,
  2019.

\bibitem{owhadi2014polyharmonic}
{\sc H.~Owhadi, L.~Zhang, and L.~Berlyand}, {\em Polyharmonic homogenization,
  rough polyharmonic splines and sparse super-localization}, ESAIM:
  Mathematical Modelling and Numerical Analysis, 48 (2014), pp.~517--552.

\bibitem{peterseim2017eliminating}
{\sc D.~Peterseim}, {\em Eliminating the pollution effect in {H}elmholtz
  problems by local subscale correction}, Mathematics of Computation, 86
  (2017), pp.~1005--1036.

\bibitem{peterseim2020computational}
{\sc D.~Peterseim and B.~Verf{\"u}rth}, {\em Computational high frequency
  scattering from high-contrast heterogeneous media}, Mathematics of
  Computation, 89 (2020), pp.~2649--2674.

\bibitem{poulson2013parallel}
{\sc J.~Poulson, B.~Engquist, S.~Li, and L.~Ying}, {\em A parallel sweeping
  preconditioner for heterogeneous 3{D} {H}elmholtz equations}, SIAM Journal on
  Scientific Computing, 35 (2013), pp.~C194--C212.

\bibitem{sauter2018stability}
{\sc S.~Sauter and C.~Torres}, {\em Stability estimate for the {H}elmholtz
  equation with rapidly jumping coefficients}, Zeitschrift f{\"u}r angewandte
  Mathematik und Physik, 69 (2018), p.~139.

\bibitem{schleuss2020optimal}
{\sc J.~Schleu{\ss} and K.~Smetana}, {\em Optimal local approximation spaces
  for parabolic problems}, arXiv preprint arXiv:2012.02759,  (2020).

\bibitem{smetana2016optimal}
{\sc K.~Smetana and A.~T. Patera}, {\em Optimal local approximation spaces for
  component-based static condensation procedures}, SIAM Journal on Scientific
  Computing, 38 (2016), pp.~A3318--A3356.

\bibitem{tartar2007introduction}
{\sc L.~Tartar}, {\em An introduction to {S}obolev spaces and interpolation
  spaces}, vol.~3, Springer Science \& Business Media, 2007.

\end{thebibliography}

\end{document}